\documentclass[11pt]{article}
\usepackage{geometry}
\usepackage{amsmath}
\usepackage{amssymb}
\usepackage{amsthm}

\usepackage[dvipsnames]{xcolor}
\usepackage[utf8]{inputenc}
\usepackage{amsfonts}
\usepackage{tikz}
\usetikzlibrary{cd}
\usepackage{graphicx}

\usepackage{mathtools}

\usepackage{enumerate} 
\newtheorem{theorem}{Theorem}[section]
\newtheorem{lemma}[theorem]{Lemma}
\newtheorem{proposition}[theorem]{Proposition}
\newtheorem{corollary}[theorem]{Corollary}
\newtheorem{claim}{Claim}
\newtheorem*{claim*}{Claim}

\newtheorem{theoremintro}{Theorem}

\newtheorem{corollaryintro}[theoremintro]{Corollary}

\theoremstyle{definition}
\newtheorem{definition}[theorem]{Definition}

\newtheorem{example}[theorem]{Example}

\theoremstyle{remark}
\newtheorem{remark}[theorem]{Remark}
\newtheorem{question}[theorem]{Question}

\usepackage[bookmarks,bookmarksnumbered,pdfstartview=FitBH,backref=page]{hyperref}
\hypersetup{
	colorlinks=true, 
	linkcolor=blue,
	citecolor= blue,
	linktoc= all, 
}

\makeatletter
\DeclareRobustCommand{\ubigcup}{\DOTSB\mathop{\,\ubigcup@\,}\slimits@}

\newcommand{\ubigcup@}{\mathpalette\ubigcup@@\relax}
\newcommand{\ubigcup@@}[2]{%
  \begingroup
  \sbox\z@{$\m@th#1\bigcup$}%
  \sbox\tw@{$\m@th#1\uparrow$}%
  \copy\z@
  \mkern-6.3mu\ifx#1\scriptscriptstyle\mkern0.3mu\fi
  \dimen@=\dimexpr\ht\z@-\ht\tw@
  \ifx#1\displaystyle\else
    \ifx#1\scriptscriptstyle\advance\dimen@ 0.5pt\else
      \advance\dimen@ 1pt
  \fi\fi
  \raisebox{\dimen@}[0pt][0pt]{\rlap{\copy\tw@}}%
  \mkern6.3mu\ifx#1\scriptscriptstyle\mkern-0.3mu\fi
  \endgroup
}

\DeclareMathOperator{\Sub}{Sub}
\DeclareMathOperator{\rk}{rk}
\DeclareMathOperator{\rkCB}{rk_{CB}}
\DeclareMathOperator{\rkCBe}{rk^{\times}_{CB}}
\DeclareMathOperator{\Stab}{Stab}
\DeclareMathOperator{\Cay}{Cay}

\DeclareMathOperator{\PK}{\mathcal{K}}
\DeclareMathOperator{\Fix}{Fix}
\DeclareMathOperator{\Aut}{Aut}

\newcommand{\Z}{\mathbb Z}
\newcommand{\N}{\mathbb N}
\newcommand\FF{\mathbf {F}}
\newcommand\la{\langle}
\newcommand\ra{\rangle}
\newcommand\id{\mathrm{id}}

\newcommand\II{\mathcal{I}}
\newcommand\OO{\mathcal{O}}

\newcommand{\acting}{\curvearrowright}
\newcommand\Conj{\mathrm{Conj}}

\newcommand{\bs}{\backslash}

\newcommand\defin[1]{\textbf{#1}}

\usepackage{authblk}

\usepackage{lipsum}

\let\OLDthebibliography\thebibliography
\renewcommand\thebibliography[1]{
  \OLDthebibliography{#1}
  \setlength{\parskip}{0pt}
  \setlength{\itemsep}{2pt plus 0.3ex}
}

\title{Perfect kernel and dynamics: from Bass-Serre theory to hyperbolic groups}
\author{P\'{e}n\'{e}lope Azuelos and Damien Gaboriau}
\date{October 22, 2024}%

\begin{document}

\maketitle

\begin{abstract}
We introduce several approaches to studying the Cantor-Bendixson decomposition of and the dynamics on the space of subgroups for various families of countable groups. 
In particular, we uncover the perfect kernel and the Cantor-Bendixson rank of the space of subgroups of many new groups, including for instance infinitely ended groups, limit groups, hyperbolic 3-manifold groups and many graphs of groups.
We also study the topological dynamics of the conjugation action on the perfect kernel, establishing the conditions for topological transitivity and higher topological transitivity. As an application, we obtain many new examples of groups in the class $\mathcal{A}$ of Glasner and Monod, i.e. admitting faithful transitive amenable actions. This includes for example right-angled Artin groups, limit groups, finitely presented C'(1/6) small cancellation groups, random groups at density $d<1/6$, and more generally all virtually compact special groups.
\end{abstract}


\smallskip

\noindent{\small \textbf{{MSC-classification:}} 37B; 20E; 43A07; 20F67; 20F65; 37B20}

\medskip

\tableofcontents

\section{Introduction}
Let $G$ be a countable group and $\Sub(G)$ be the space of all subgroups of $G$.
There has been a recent increase in studies of the action by conjugation of $G$ on this space.
It is the crucible where certain properties of the non-free actions of G boil down, whether they are of a topological or measured nature.

When endowed with the Chabauty topology, $\Sub(G)$ is a compact metrizable totally disconnected space, thus satisfying the Baire property. The Cantor–Bendixson structure theory of Polish spaces leads us to consider the unique decomposition $\Sub(G) = \PK(G) \sqcup C$ where $\PK(G)$, the \defin{perfect kernel} of $G$, is the maximal perfect subspace of $\Sub(G)$ and $C$ is a countable set. Associated to this decomposition is an ordinal, called the \defin{Cantor-Bendixson rank} of $G$ and denoted $\rkCB(G)$. It captures after how many steps the iterative process of removing isolated points  from $\Sub(G)$ results in the stable subspace $\PK(G)$.
Precise definitions are given in Section \ref{Preliminaries}.

The action of $G$ on its space of subgroups by conjugation is continuous and the perfect kernel is invariant under this action. 
The interesting uniformly recurrent subgroups (URS) and invariant random subgroups (IRS) are supported on the perfect kernel. 

It is an instructive exercise to check that $\PK(\Z^d)=\emptyset$ and that $\rkCB(\Z^d)=d+1$ for all $d \in \mathbb{N}$.
In fact, since a non-empty perfect Polish space contains a Cantor set, $\PK(G)=\emptyset$ if and only if $G$ admits only countably many subgroups. This is the case for instance for all finitely generated nilpotent groups. More generally, this happens for all groups whose subgroups are all finitely generated (the so-called Noetherian groups), for instance for the Tarski Monsters: those non-abelian groups whose proper subgroups are all cyclic.

If $G$ is the $\mathfrak{R}$-free group of countably infinite rank of a (non-trivial) variety $\mathfrak{R}$ of groups (such as the free group $\FF_\infty$, $\oplus_\N \Z$,  $\oplus_\N \Z/m\Z$, the free Burnside group $B(\N,p)=\la (s_i)_{i\in \N}\vert g^p=1, \text{ for all } g\ra$ for $p\geq 2$,...), then  
$\PK(G)=\Sub(G)$ (see Proposition~\ref{Prop: K of count. free variety}) and $\rkCB(G)=0$.

In general, if $G$ is finitely generated then its finite index subgroups are isolated, so \[\PK(G) \subseteq \Sub_{[\infty]}(G) \text{ where } \Sub_{[\infty]}(G) := \{H \leq G \text{ s.t. } [G:H] = \infty\}.  \]
In the case where $\PK(G)=\Sub_{[\infty]}(G)$, note that $\rkCB(G)=1$.
Observe that $\Sub(G)$ may contain infinite index isolated points. There are even examples of infinite groups for which the trivial subgroup $\{\id\}$ is isolated: consider for instance the Pr\"ufer $p$-group, or see the more interesting finitely generated Example~\ref{ex: ex. with isolated id}.

It is not hard to check that the equality holds for all finitely generated non-cyclic free groups $\FF_r$:
\[\PK(\FF_r)=\Sub_{[\infty]}(\FF_r), \]
see e.g. \cite[Proposition 2.1]{CGLM-1-arxiv}. Moreover, the action of $\FF_r$ on $\PK(\FF_r)$ is highly topologically transitive (see Corollary \ref{cor:infinitely many ends} or Corollary \ref{cor: locally quasiconvex}).
Recall that a continuous action $G\acting X$ on a compact metrizable space of a countable group $G$ is called \defin{topologically $r$-transitive} for a positive integer $r$ if the diagonal action $G\acting X^r$ on the $r$-fold space is \defin{topologically transitive}; i.e. if for every $2r$-tuple $(V_1,V_2,\cdots, V_{2r})$ of non-empty open sets, there is an element $g\in G$ such that $g \cdot V_{i}\cap V_{r+i}\not=\emptyset$ for all $i=1, 2, \cdots, r$ (see Section~\ref{Preliminaries}). The action is called \defin{highly topologically transitive} if it is topologically $r$-transitive for all $r \geq 1$.

One purpose of this article is to extend the validity of these statements to several new groups.
 
 In recent years, the space of subgroups of several classes of groups have been studied by various authors. Cornulier, Guyot and Pitsch \cite{Cornulier-Guyot-Pitsch-2010} classified the spaces of subgroups of all countable abelian groups up to homeomorphism and computed their Cantor-Bendixson rank. In \cite{BGK-2015-IRS-Lamp}, Bowen, Grigorchuk and Kravchenko found the perfect kernel and Cantor-Bendixson rank of the lamplighter groups 
$(\Z/k\Z)^n\wr \Z=\oplus_{\Z}(\Z/k\Z)^n\rtimes \Z$: they showed that
$\PK((\Z/k\Z)^n\wr \Z)=\Sub(\oplus_{\Z}(\Z/k\Z)^n)$ and $\rkCB((\Z/k\Z)^n\wr \Z)=\omega$, the first infinite ordinal.
Skipper and Wesolek described in \cite{Skipper-Wesolek-2020} the perfect kernels of the Grigorchuk group, the Gupta-Sidki $3$-group 
and some other families of groups of automorphisms of infinite regular rooted trees; they also established that $\rkCB(G)=\omega$ for these groups.

Recently Carderi, Le Maître, Stalder and the second author described the perfect kernel and the dynamics for the Baumslag-Solitar groups $\mathrm{BS}(m,n)=\la b,t\vert tb^mt^{-1}=b^n\ra$ depending on the parameters $m,n$ in \cite{CGLMS-1-arxiv}. If $\lvert m \lvert, \vert n\vert >1$, then $\PK(\mathrm{BS}(m,n))$ is the space of subgroups whose action on the standard Bass-Serre tree of $\mathrm{BS}(m,n)$ has infinite quotient. Furthermore, the perfect kernel admits a countably infinite partition into $\mathrm{BS}(m,n)$-invariant subspaces. One of them is closed and not open; all the others are open (and also closed exactly when $\vert m\vert=\vert n\vert$). Moreover, the $\mathrm{BS}(m,n)$-action on each piece is topologically transitive.

\medskip
Intrinsically, this theory is less interested in the algebraic structure of subgroups so much as the ``geometry of their embedding'' into $G$. A result which exemplifies this is the following:
\\
{\em If $H \leq G$ are finitely generated groups and the quotient $H \bs G$ is quasi-isometric to $\Z$ then $H$ vanishes by the third Cantor-Bendixson derivative of $\Sub(G)$} (Theorem \ref{th: H co-QI Z then CB-e-rk(H) <4}).

\paragraph{Hyperbolic groups}
An important class of finitely generated groups is that of hyperbolic groups, introduced by Gromov \cite{Gromov87}.
Describing the space of subgroups of a hyperbolic group is in general a difficult task, largely due to the fact that finitely generated subgroups of hyperbolic groups need not respect the geometry of the group. There are however a range of tools available to study those subgroups that do respect the geometry: namely the quasiconvex subgroups. 

Given a hyperbolic group $G$, we denote by $\Sub_{\mathrm{QC}[\infty]}(G)$ the set of quasiconvex subgroups of $G$ with infinite index. 

\begin{theoremintro} [Theorems \ref{thm: kernel QC hyp} and \ref{thm: QC hyp dynamique}] \label{thmintro: hyperbolic}
Let $G$ be a non-elementary hyperbolic group. Then 
\[ \overline{\Sub_{\mathrm{QC}[\infty]}(G)} \subseteq \PK(G).\]
If in addition $G$ has no non-trivial finite normal subgroups then the action of $G$ on $ \overline{\Sub_{\mathrm{QC}[\infty]}(G)}$ is highly topologically transitive. 
\end{theoremintro}

A hyperbolic group is said to be locally quasiconvex if each of its finitely generated subgroups is quasiconvex. There are many examples of hyperbolic groups which satisfy this property, including (but not limited to): hyperbolic surface groups, hyperbolic limit groups, 
certain one-relator groups with torsion, certain small cancellation groups and graphs of free groups with cyclic edge groups which do not contain any Baumslag-Solitar subgroups (see Examples~\ref{ex: hyp limit gps are loc qc}, \ref{ex: hyperbolic gps w. a small hierarchy}, 
\ref{ex: one-relator group},
and \ref{ex: Coxeter groups}, and the references therein).

\begin{corollaryintro} [Corollary \ref{cor: locally quasiconvex}] \label{corintro: locally QC}
If a group $G$ is non-elementary hyperbolic and locally quasiconvex then
\[ \PK(G) = \Sub_{[\infty]}(G) = \overline{\Sub_{QC[\infty]}(G)}.\]
If in addition $G$ has no non-trivial finite normal subgroups then the action of $G$ on $\PK(G)$ is highly topologically transitive.
\end{corollaryintro}

Another class of hyperbolic groups of particular interest is that of hyperbolic 3-manifold groups. While these are not locally quasiconvex, their finitely generated subgroups are well understood. As a result, we are able to describe their perfect kernels in their entirety.

\begin{theoremintro} [Theorem~\ref{theorem: kernel hyp 3-manifolds}]
Let $M$ be a closed hyperbolic 3-manifold and $G = \pi_1(M)$. Then 
\[ \PK(G) = \overline{\Sub_{\mathrm{QC}[\infty]}(G)}. \]
More precisely, the perfect kernel of $G$ is the set of infinite index subgroups $H \leq G$ such that either $H$ is quasiconvex or $H$ is not finitely generated.
\\
Moreover, the action of $G$ on $\PK(G)$ is highly topologically transitive and the Cantor-Bendixson rank of $G$ is $2$ or $3$.
\end{theoremintro}

The above theorem makes use of significant discoveries which have been made about 3-manifolds over the past few decades. In particular, we use a consequence of the Tameness Theorem \cite{Agol04}, \cite{Calegari-Gabai06} and Canary's Covering Theorem \cite{Canary96} as well as the fact that closed hyperbolic 3-manifolds are virtually fibered \cite{Agol13} (by virtue of being virtually special in the sense of \cite{Haglund-Wise-2008-SCC}). 

In general, the inclusion $\overline{\Sub_{QC[\infty]}(G)} \subseteq \PK(G)$ is not an equality. Indeed, Rips' construction \cite{Rips-1982-small-canc-} allows us to find a hyperbolic group $G$ so that we can embed a $G$-invariant open subset $U$ into the complement $\PK(G) \smallsetminus \overline{\Sub_{QC[\infty]}(G)}$. Moreover, one can ensure that the action of $G$ on $U$ is as pathological as that of any finitely presented group on its perfect kernel. We discuss this in more detail in Section \ref{Section: Facts and questions}.

\paragraph{Subgroup separability}

A group $G$ is \defin{subgroup separable} (also referred to as LERF or locally extended residually finite) if every finitely generated subgroup of $G$ can be expressed as an intersection of finite index subgroups. The class of groups with this property includes free groups \cite{Hall-1949-subgr-free}, surface groups \cite{Scott78}, limit groups \cite{Wilton}, hyperbolic 3-manifold groups \cite{Agol13}, the first Grigorchuk group \cite{Grigorchuk-Wilson} and some lamplighter groups \cite{Grigorchuk-Kravchenko}.

Subgroup separability is equivalent to the density in $\Sub(G)$ of the set of finite index subgroups, as observed in \cite{GKM} (see Section~\ref{sect: LERF and 1st L2 Betti}). 
This does not provide any further information on the perfect kernel of a group, except when we can relate the rank (minimal number of generators) and the index of finite index subgroups. For this purpose, we require that the first $\ell^2$-Betti number be positive, we will see that in many cases this  condition is in fact optimal (see Remark \ref{rem: virtual fibring vs l2 Betti number}).

\begin{theoremintro}[Theorem~\ref{theorem: lerf + rk f.i. approx} and Corollary~\ref{cor: lerf + beta1}]
Let $G$ be a countable and subgroup separable group. If $G$ is not finitely generated then $\Sub(G)$ is a perfect space. If $G$ is finitely generated with positive first $\ell^2$-Betti number then $\PK (G) = \Sub_{[\infty]} (G)$ and $\rkCB(G) = 1$.
\end{theoremintro}

The degree of topological transitivity of the $G$-action on $\PK(G)$ still depends on $G$ in this case. For instance, as previously mentioned, the action of $\mathbf{F}_r$ on $\PK(\mathbf{F}_r)$ is highly topologically transitive. In contrast the $\mathrm{SL}(2,\mathbb{Z})$-action on $\PK(\mathrm{SL}(2,\mathbb{Z}))$ is not even 1-topologically transitive (see Example~\ref{ex: SL(2,Z)}).

As mentioned above, free groups and surface groups are subgroup separable. Moreover, hyperbolic surface groups have positive first $\ell^2$-Betti number, as do non-abelian free groups. Thus the above theorem implies that the perfect kernel of a finite rank non-abelian free group or a hyperbolic compact surface group coincides with its set of infinite index subgroups, and its Cantor-Bendixson rank is $1$. Note that the remaining (non-hyperbolic) compact surface groups are virtually abelian so their perfect kernel is empty.

More generally, we have the following corollary for Sela's limit groups. Recall that a group $G$ is a \defin{limit group}, or \defin{$\omega$-residually free}, if for any finite set $S \subseteq G \smallsetminus \{\id\}$ there exists a normal subgroup $N \unlhd G$ such that the quotient $G/N$ is free and $S \cap N = \emptyset$. Equivalently, limit groups are the limits of free groups in the space of marked groups \cite[Theorem 1.1]{Champetier-Guirardel05}.

\begin{corollaryintro}[Limit groups (Corollaries~\ref{cor: limit group} and~\ref{cor: groups in CCC, K(G) and HTT})]
If $G$ is a finitely generated non-abelian limit group, then $\PK(G)=\Sub_{[\infty]}(G)$ and the action $G\acting \PK(G)$ is highly topologically transitive.
\end{corollaryintro}

\paragraph{Groups with infinitely many ends and actions
 on trees}

A fundamental notion in geometric group theory is the notion of ends of a group.  
Recall that a finitely generated group has either zero, one, two or infinitely many ends. Stallings' Theorem \cite{Sta68,Stallings-1971-inf-ends} states that a group has infinitely many ends if and only if it acts non-trivially and irreducibly on a tree with finite edge stabilisers. This theorem is invaluable in the study of infinitely ended groups and, in the context of spaces of subgroups, it allows us to show the following.

\begin{theoremintro}[Infinitely many ends (Corollary~\ref{cor:infinitely many ends})]
If $G$ is a finitely generated group with infinitely many ends, then the perfect kernel of $G$ consists of its subgroups with infinite index:
\[\PK(G)=\Sub_{[\infty]}(G).\]
Moreover, if $G$ has no non-trivial finite normal subgroups then the action of $G$ on $\PK(G)$ is highly topologically transitive.
\end{theoremintro}

It is also interesting to consider actions on trees where stabilisers of individual edges are not necessarily finite or trivial but rather where the stabilisers of sufficiently long paths are trivial. These actions are known as \defin{acylindrical}. Examples of groups admitting such actions include amalgamated free products $A \ast_C B$ where $C$ is malnormal in $A$ or $B$ (See Section~\ref{Sect: Applications of dendrolog. results}).

\begin{corollaryintro}[Acylindrical actions (see Corollary~\ref{cor:acylindrical})] \label{corintro: acylindrical} 
Let $G\acting T$ be a minimal irreducible and acylindrical action of a countable group on a tree.
If $H\leq G$ is a subgroup such that the quotient graph $H\bs T$ is infinite then 
$H\in \PK(G)$. 
\end{corollaryintro}

An acylindrical action on a tree only provides information on those subgroups whose quotient is infinite. We construct examples of finitely generated groups $G$ which act acylindrically on trees and which contain infinite index subgroups $H$ whose quotients are finite. Sometimes these subgroups $H$ are in the perfect kernel and sometimes they are not  (see Remark~\ref{Rem: acylindrical act. examples}). More information about the group $G$ is therefore required to describe the perfect kernel fully.

The conclusion of Corollary \ref{corintro: acylindrical} in fact holds under a weaker condition than acylindricity of the action. Given a group $G$ acting on a tree $T$, define the set 
\[\Sub_{\vert \bullet\bs T\vert \infty}(G):=\{H \in \Sub(G)\colon \left\vert H\bs T\right\vert=\infty\},\]
of subgroups that act on $T$ with infinite quotient graph.
\begin{theoremintro} [Theorems~\ref{th: two edge-stab with finite intersection} and \ref{th: two edge-stab with finite intersection - top. dyn.}] \label{thmintro trees}
Assume $G\acting T$ is a minimal and irreducible action on a tree. 
Assume that there are two edges $f_1,f_2$ of $T$ such that $\Stab_G(f_1)\cap \Stab_G(f_2)$ is finite.
Then 
 \[\overline{\Sub_{\vert \bullet\bs T\vert \infty}(G)}\subseteq \PK(G).\]
 If, in addition, the kernel of the action $G \acting T$ is trivial then the action of $G$ on $\overline{\Sub_{\vert \bullet\bs T\vert \infty}(G)}$ is highly topologically transitive.
\end{theoremintro}

This result allows us to establish that $\PK(G) = \Sub_{[\infty]}(G)$ for a large family of groups (namely the class $\mathcal{Q}$, see Definition~\ref{def: class Q}), which includes graphs of free groups with cyclic edge groups such that at least one vertex group is non-cyclic.
It also allows us to deduce the topological transitivity properties of the action $G \acting \PK(G)$ (see Corollary~\ref{cor: groups in CCC, K(G) and HTT}).

\bigskip

The reader will probably have noticed that our results above concerning the dynamics of the action $G \acting \PK(G)$ 
exclude groups $G$ that contain a non-trivial finite normal subgroup $N$. 
They may have wondered what happens in this case. 
In fact, the action of $G$ on $\Sub(N)$ by conjugation induces a finite $G$-invariant clopen partition of $\Sub(G)$ and of $\PK(G)$ (see Proposition~\ref{prop:finite normal subgr clopen part}). 
In the context of Theorems \ref{thmintro: hyperbolic} and \ref{thmintro trees}, the subgroup $N_0$ generated by the union of all finite normal subgroups of $G$ remains finite (and normal): it is called the \defin{finite radical} of $G$. When we allow $N_0$ to be non-trivial, the action of $G$ on each piece of the resulting clopen partition is topologically transitive (see Theorem \ref{thm: QC hyp dynamique} and Theorem \ref{th: two edge-stab with finite intersection - top. dyn.}). 
However, in general the action on the pieces is not topologically $2$-transitive, see Remark~\ref{rem: obstruction to r-top transitive}.

\paragraph{An application to transitive amenable actions}

As a by-product of our investigations, we obtain information on transitive actions.
Following Glasner-Monod \cite{Glasner-Monod-2007}, let $\mathcal{A}$ be the class of countably infinite groups that admit a {\em faithful transitive and amenable} action.

This class contains all free groups and most free products (see \cite{Glasner-Monod-2007}), surface groups and hyperbolic $3$-manifolds groups (see \cite{Moon-2010}) and some instances of amalgamated free products and HNN-extension over finite or amenable groups (see \cite{{Moon-2011-free-prod-A},Moon-2011b,Fima-2014-A-tree}).

\def\RRR{R}

Our results allow us to greatly enlarge the collection of examples of groups in the class $\mathcal A$.
This is done by singling out two subgroups $H_1, H_2\in \Sub_{[\infty]}(G)$ such that $G\acting G/H_1$ is faithful while $G\acting G/H_2$ is amenable (thus isolating two separate conditions: faithfulness and co-amenability) and by 
finding (thanks to topological transitivity) subgroups whose orbit in $\Sub(G)$ accumulates on both $H_1$ and $H_2$ (see Proposition~\ref{prop: conditions for being in A}).

Our Corollary~\ref{cor:infinitely many ends} (on groups with infinitely many ends) admits as a corollary Theorem~3.4 of \cite{Glasner-Monod-2007}.
The following generalisation was suggested to us by Fran\c{c}ois Le Ma\^{i}tre, following the circulation of a previous version of this paper. 
\begin{corollaryintro}[of Proposition~\ref{prop: conditions for being in A}]
\label{cor: generalisation of 3.4 GM}
Assume that the action $G\acting \Sub_{[\infty]}(G)$ is topologically transitive. 
Assume that $G$ admits an amenable action with only infinite orbits.
Then $G\in \mathcal{A}$.
\end{corollaryintro}
\begin{proof}
Observe that the group $G$ is infinite; thus $\{\id\}\in \Sub_{[\infty]}(G)$ satisfies Condition~\eqref{it: faithful}
of Proposition~\ref{prop: conditions for being in A}. The remaining Conditions~\eqref{it: co-amenability condition for A} and ~\eqref{it: cond on H0} follow respectively from the amenability and topological transitivity conditions on $G$.
\end{proof}

Here is a (non-exhaustive) list of groups which we can prove belong to the class $\mathcal{A}$.

\begin{itemize}
    \item Any infinitely ended group without a non-trivial finite normal subgroup and with some amenable action with infinite orbits
(Corollary~\ref{cor:infinitely many ends} and Proposition~\ref{prop: conditions for being in A}).
For instance, any HNN-extension $G$ of any infinite group $H$ over a pair of  finite subgroups with trivial intersection.

\item The fundamental group of any non-tree graph of groups where some edge group is malnormal in an incident vertex group
 (See Corollary~\ref{cor: non-tree of groups w. a w.w.malnormal incident edge group}).

    \item Virtually compact special groups (see Theorem~\ref{thm: virtually compact special in A}). In particular: 
    \\
    -- Right-angled Artin groups; \\
    -- finitely presented C'(1/6) small cancellation groups; \\
    -- Right-angled and hyperbolic Coxeter groups; \\
    -- One-relator groups with torsion; \\
    -- Limit groups; \\
    -- Random groups at density $d<1/6$;\\
    -- Hyperbolic-by-cyclic groups which are themselves hyperbolic.
     
    \item Groups in the class $\mathcal{Q}'$ (Definition~\ref{def: class Q'})
     that admit an infinite index co-amenable subgroup 
(for instance if the defining graph of groups is not a tree)
 (Corollary~\ref{cor: groups in CCC, K(G) and HTT} and Proposition~\ref{prop: conditions for being in A}). 
        In particular, finite graphs of free groups with cyclic edge groups such that at least one vertex group is non-cyclic (Corollary~\ref{cor: graphs of free groups in A}).
\end{itemize}

After an earlier version of our paper was circulated, we learned that some of our results 
concerning the dynamics on $\Sub(G)$ were being studied independently by Hull, Minasyan and Osin (see \cite{HMO}). They have shown for instance that, if $G$ has trivial finite radical and admits a partially WPD action on a hyperbolic space $S$, then the action of $G$ on the closure of the set of its infinite index convex cocompact subgroups (with respect to the action $G \acting S$) is \textit{topologically $\mu$-mixing} for appropriate probability measures $\mu \in \mathrm{Prob}(G)$, which implies that it is highly topologically transitive. This recovers our Theorem~\ref{thm: QC hyp dynamique} in the case where $G$ has trivial finite radical. Note that, while the actions on trees that we consider in Theorem~\ref{th: two edge-stab with finite intersection - top. dyn.} are partially WPD, the closure of the set of convex cocompact subgroups of $G$ is, in general, a strict subspace of the space $\overline{\Sub_{\vert \bullet\bs T\vert \infty}(G)}$ 
on which we obtain high topological transitivity.

\paragraph{Acknowledgments}
We are extremely grateful to Mark Hagen for his helpful comments, for many enlightening discussions on hyperbolic groups and for asking us about the space of subgroups of hyperbolic 3-manifold groups.
The first author is particularly grateful for his ongoing support.
\\
We would like to thank Ian Biringer for pointing out an example of a hyperbolic 3-manifold group with Cantor-Bendixson rank 3 and for a discussion which motivated us to include and improve upon the results of Section \ref{section: subgroups QI to Z}.
\\
We thank Alessandro Carderi, François Le Maître and Yves Stalder for many enlightening discussions and for drawing our attention to the example of the group $G$ in Example~\ref{ex: ex. with isolated id}.
\\
We also thank François Dahmani for several discussions on hyperbolic groups.
\\
We are grateful to Eli Glasner for asking us whether our previous results about topologically transitive actions could be upgraded to weakly mixing actions and for giving us a few basics on the subject. 
This led us to consider highly topologically transitive actions. 
This question was asked during the wonderful conference ``Measured Group Theory, Stochastic Processes on Groups and Borel Combinatorics'' which took place at the CIRM in May 2023.
\\
We are grateful to Matthias Uschold for his comments on a previous version of this paper.
\\
P.~A. is supported by a University of Bristol PhD scholarship. She was supported by the LABEX MILYON (ANR-10-LABX-0070) of Université de Lyon, within the program “Investissements d’Avenir” (ANR-11-IDEX-0007) operated by the French National Research Agency
(ANR) during the academic year 2021-2022, when we initiated this project.
\\
D. G. is supported by the Centre National de la Recherche Scientifique (C.N.R.S.)

\section{Notation and preliminaries} \label{Preliminaries}

Let $G$ be a countable group. The space of subgroups $\Sub(G)$ of $G$ is a subspace of the set $\{0,1\}^G$, where each subset of $G$ is identified with an element of $\{0,1\}^G$ via the characteristic map. The Chabauty topology on $\Sub(G)$ is the subspace topology induced by the product topology on $\{0,1\}^G$. It follows that $\Sub(G)$ is compact, metrisable, separable and totally disconnected. In particular, it is a Polish space, thus satisfies the Baire category theorem.

The collection of subsets \[ \mathcal{V}_G(\II,\OO) := \{H \leq G \colon \II \subseteq H, H \cap \OO = \emptyset\}\subseteq \Sub(G)\]  for all finite subsets $\II, \OO \subseteq G$ forms a basis of clopen sets for the Chabauty topology on $\Sub(G)$.
When there is no risk of confusion we will write $\mathcal{V}(\II,\OO)$ instead of $\mathcal{V}_G(\II, \OO)$. 
The action of $G$ by conjugation ($g\cdot H= gHg^{-1}$)  on $\Sub(G)$ is continuous.

\paragraph{The Cantor-Bendixson decomposition.}
Let $X$ be a Polish space. The Cantor-Bendixson theorem \cite{Cantor1884,Bendixson1900} (see \cite{Kechris95} for an accessible account) states that there is a unique decomposition $X = \PK(X) \sqcup C$ such that $\PK(X)$ is a perfect subspace of $X$ (i.e. it is closed and has no isolated points when equipped with the subspace topology) and $C$ is countable. The subspace $\PK(X)$ is called the \defin{perfect kernel} of $X$. The derived set $X'$ of $X$ is the complement in $X$ of the set of isolated points of $X$. 
 The \defin{Cantor-Bendixson derivatives} of $X$ are defined by repeatedly applying the derived set operation using transfinite induction: Let $X^{(0)} \coloneqq X$.
Given an ordinal $\alpha$ such that the Cantor-Bendixson derivatives $X^{(\beta)}$ are defined for every $\beta<\alpha$,
if $\alpha$ is a successor $\alpha=\delta+1$ then 
let $X^{(\alpha)} := \left(X^{(\delta)}\right)'$ and if $\alpha$ is a limit ordinal then let $X^{(\alpha)} := \cap_{\beta < \alpha} X^{(\beta)}$. 

The Cantor-Bendixson theorem implies that there is an at most countable ordinal $\alpha$ such that $X^{(\alpha)} = X^{(\alpha+1)}$. In this case, $X^{(\alpha)}$ is the perfect kernel $\PK(X)$. The minimal ordinal $\alpha$ such that $X^{(\alpha)}= \PK(X)$ is called the \defin{Cantor-Bendixson rank} of $X$ and is denoted by $\rkCB(X)$.
A point $x\in X$ belongs to 
$\PK(X)$ if and only if $x$ is a condensation point, i.e. every open neighbourhood of $x$ is uncountable.

When $G$ is a countable group, we define 
\begin{eqnarray*}
\PK(G) &:= &\PK(\Sub(G)) \text{ the \defin{perfect kernel of $G$}, and}\\
\rkCB(G) &:=& \rkCB(\Sub(G)) \text{ the  \defin{Cantor-Bendixson rank of $G$}.}
\end{eqnarray*}

Let $H\leq G$ be a subgroup.
 The \defin{Cantor-Bendixson erasing rank} $\rkCBe(H;G)$ of $H$ in $G$ 
 is, when $H\notin\PK(G)$, the infimum of the ordinals $\alpha$ such that the derivative $\Sub(G)^{(\alpha)}$ does not contain $H$; and we write $\rkCBe(H;G)=\infty$
when  $H\in\PK(G)$. We make the convention that $\infty$ is larger than all the ordinals. Observe that $\rkCBe(H;G)$ is not a limit ordinal:
\[\rkCBe(H;G)=\begin{cases}
 \infty &\text{ if } H\in \PK(G)\\
 \alpha+1 &\text{ if } H\in\Sub(G)^{(\alpha)} \text{ but } 
 H\not\in\Sub(G)^{(\alpha+1)}.
 \end{cases}\]
Moreover, $\rkCB(G) = 0$ if and only if $\Sub(G)$ is perfect and $\rkCB(G) = \sup\{\rkCBe(H;G) \colon H \in \Sub(G) \smallsetminus \PK(G)\}$ otherwise.

\begin{example}\label{example rkCBe H in Z^d}
For instance, when $G=\Z^{d}$, then $\rkCBe(H;\Z^d)=d-\rk(H)+1$ where $\rk(H)$ is the \defin{minimal number of generators} of the group $H\leq \Z^d$.
In particular, $\rkCBe(\{\id\};\Z^d)=d+1$, $\rkCB(\Z^d)=d+1$, and $\PK(\Z^d) = \emptyset$.
\end{example}

Note that there is another notion of {\em (intrinsic) Cantor-Bendixson rank} $\mathrm{cb}(G)$ 
of a countable group $G$, defined in \cite{Cornulier-2011-CB-rk-metab} as follows. Let $\mathcal{N}(G)$ denote the space of normal subgroups of $G$. Then $\mathrm{cb}(G)$ is the supremum of the ordinals $\alpha$ such that $\{\id\}\in \mathcal{N}(G)^{(\alpha)}$ when such an $\alpha$ exists and $\mathfrak{C}$ otherwise, i.e. when $\{\id\}\in \PK(\mathcal{N}(G))$.
This last notion coincides with the CB-erasing rank of $G$ in the space of marked groups when $G$ is finitely presented.

\paragraph{Topological transitivity}
A continuous action $G\acting X$ of a countable discrete group $G$ on a compact metrizable space is called \defin{topologically transitive} if it satisfies one of the following equivalent conditions:
\begin{enumerate}
\item For every pair $(V_1,V_2)$ of non-empty open subsets of $X$, there is an element $g\in G$ such that $g\cdot V_1\cap V_2\not=\emptyset$;
\item There is a dense orbit, i.e. a point $x\in X$ whose orbit $G\cdot x$ is dense in $X$;
\item There is a G$_\delta$ subset of $X$ consisting of dense orbits.
\end{enumerate}

A continuous action $G\acting X$ on a compact metrizable space of a countable group $G$ is called \defin{topologically $r$-transitive} for a positive integer $r$ if the diagonal action $G\acting X^r$ on the $r$-fold space is topologically transitive; i.e. if for every $2r$-tuple $(V_1,V_2,\cdots, V_{2r})$ of non-empty open sets, there is an element $g\in G$ such that $g.V_{i}\cap V_{r+i}\not=\emptyset$ for all $i=1, 2, \cdots, r$.
A topologically $2$-transitive action is also called \defin{topologically weakly mixing}. If an action is topologically $r$-transitive for all $r \geq 1$ we say that it is \defin{highly topologically transitive}.

Contrarily to the analogous concept in ergodic theory, topological $2$-transitivity does not imply high topological transitivity when the group $G$ is not abelian.

\section{\texorpdfstring{Generalities on the topological structure of $\Sub(G)$}{Generalities on the topological structure of Sub G}}
\label{sect:Generalities on perfect kernels of groups}

In this section we highlight the effects of various properties of individual subgroups, and of their quotients, on the Cantor-Bendixson decomposition of the space of subgroups and the dynamics of the action by conjugation.

\begin{remark}\label{rem:K(subGr) in K(Gr)}
If $H\leq G$ is a subgroup of the countable group $G$, then the natural continuous embedding $\Sub(H)\subseteq \Sub(G)$ induces a continuous embedding
$\PK(H)\subseteq \PK(G)$.
In fact, the topology on $\Sub(H)$ is induced from that on $\Sub(G)$.
Condensation points of $\Sub(H)$ are thus condensation points of $\Sub(G)$.
\end{remark}

\subsection{Normal subgroups and invariant subsets}

A normal subgroup $N\trianglelefteq G$ is clearly a fixed point of the $G$-action on $\Sub(G)$ by conjugation. In the case where $N$ is finite, this property induces an invariant partition of the entire space.

\begin{proposition}[Finite normal subgroup and its clopen partition]
\label{prop:finite normal subgr clopen part}
If $N\trianglelefteq G$ is a finite normal subgroup $N$.
The set $\Conj(N)^{G}$ of $G$-conjugacy classes of subgroups of $N$ induces a finite $G$-invariant clopen partition: 
\[\Sub(G)=\bigsqcup_{C\in \Conj(N)^G} \{H\in \Sub(G)\colon H\cap N\in C\}.\]
\end{proposition}
\begin{proof}
By normality and finiteness, $G$ acts by conjugation on $\Sub(N)$ with finite orbits. These orbits form the set $\Conj(N)^{G}=\{\{\gamma M \gamma^{-1}\colon \gamma\in G\}\colon M\in \Sub(N)\}$.
Since $\gamma H \gamma^{-1}\cap N=\gamma (H\cap N) \gamma^{-1}$, the $G$-orbit of $H\cap N\in \Sub(N)$ is an invariant of the $G$-orbit of $H\in \Sub(G)$.
In addition, for $C\in \Conj(N)^{G}$, the condition $H\cap N\in C$ is a clopen condition in $\Sub(G)$: 
it defines a union of basic clopen sets: 
$\{H\in \Sub(G)\colon H\cap N\in C\}=\cup_{M\in C}\mathcal{V}(M,N\smallsetminus M)$.
\end{proof}

If $G'$ is a subgroup of the countable group $G$, then the map 
\begin{equation}\label{eq:def of intersect map P}
P:  \left(\begin{array}{ccc}\Sub(G) & \longrightarrow &\Sub(G')\\H&\longmapsto &H \cap G'\end{array}\right)
\end{equation}
 is continuous: if $\mathcal{V}_{G'}(\II',\OO')$ is a basic open set in $\Sub(G')$ (with $\II',\OO'$ finite subsets of $G'$), then $P^{-1}(\mathcal{V}_{G'}(\II',\OO'))=\mathcal{V}_G(\II',\OO')$.

\begin{proposition}\label{prop:finite index subgr}
Let $G' \leq G$ be a finite index subgroup and $P: \Sub(G) \rightarrow \Sub(G')$ be the continuous map from~\eqref{eq:def of intersect map P}. Then each fiber of $P$ is countable.
It follows that the natural embedding $\PK(G')\subseteq \PK(G)\cap \Sub(G')$ is an equality: $\PK(G')=\PK(G)\cap \Sub(G')$.
\end{proposition}

\begin{proof}[Proof of Proposition~\ref{prop:finite index subgr}]
Fix a subgroup $G_0' \leq G'$. Let $k := [G:G']$ and let $g_1, \dots, g_k \in G$ be such that $G = \sqcup_{i=1}^k g_i G'$. If $H \in P^{-1}(G_0')$ let $I \subseteq \{1, \dots k\}$ be such that $H \cap g_i G' \neq \emptyset$ if and only if $i \in I$. Then, for each $i \in I$, there is $h_i \in g_i G'$ such that $H \cap g_i G' = h_i G_0'$. Thus $H = \sqcup_{i \in I} h_i G_0'$. It follows that
\[ P^{-1}(G_0') \subseteq \bigcup_{I \subseteq \{1, \dots, k\}} \Big\{ \bigsqcup_{i \in I} h_i G_0' : h_i \in g_i G'\Big\}.\]
Thus $P^{-1}(G_0')$ is contained in a finite union of countable sets and is therefore countable.

To see that $\PK(G) \cap \Sub({G'})\subseteq \PK({G'})$, 
let $H' \in \Sub(G') \smallsetminus \PK(G')$.
Pick a countable neighbourhood $\mathcal{V}'$ of $H'$ in $\Sub(G')$. Its pre-image $P^{-1}(\mathcal{V}')$ is a countable neighbourhood of $H'$ in $\Sub(G)$ so $H' \notin \PK(G)$.
\end{proof}

\begin{proposition} \label{prop: finite normal sub}
Let $N\trianglelefteq G$ be a finite normal subgroup of the countable group $G$. Assume $G$ contains a torsion-free subgroup $H$ such that $\{\id\}\in \PK(H)$. Then $\Sub(N)\subseteq \PK(G)$.
\end{proposition}

\begin{proof}
The subgroup $H$ acting by conjugation on the finite subgroup $N$
admits a finite index subgroup $H'$ that acts trivially; i.e. $H'$ commutes with $N$, and since it is torsion-free, it intersects $N$ trivially, thus $K=H'\times N$ injects naturally as a subgroup of $G$.
If $M=\{\id\}\times M$ is a subgroup of $N=\{\id\}\times N$ and $\mathcal{V}=\mathcal{V}_{K}(M,\OO)$ is any basic neighbourhood of $M$ in $\Sub(K)$, with $\OO=\{(x_1,n_1), (x_2,n_2),\cdots, (x_j,n_j)\}$ where $x_i\in H'$ and $n_i\in N$, then $\mathcal{V}':=\mathcal{V}_{H'}(\emptyset,\{x_1, x_2, \cdots, x_j\})$ is a basic neighbourhood of $\{\id\}$ in $\Sub(H')$. 
By Proposition~\ref{prop:finite index subgr}, $\{\id\}\in \PK(H')$.
Thus $\mathcal{V}'$ is uncountable and the uncountable set $\{H'\times M\colon H'\in \mathcal{V}'\}$ is contained in $\mathcal{V}$. It follows that $M\in \PK(K)$, and so $M\in \PK(G)$ by Remark~\ref{rem:K(subGr) in K(Gr)}.
\end{proof}

Let $H\leq G$ be a subgroup.
We denote by \[\Sub_{H\leq}(G)\!:=\{L\leq G\colon H\leq L\}\] the closed subspace 
 of $\Sub(G)$ consisting of subgroups that contain $H$.
 In the case where $H$ is generated by a finite set $S$, then $\Sub_{H\leq}(G)=\mathcal{V}(S,\emptyset)$ is a clopen neighbourhood of $H$ in $\Sub(G)$. If $H$ is normal, then the action of $G$ on $\Sub(G)$ by conjugation induces natural actions on the spaces of subgroups $G \acting \Sub_{H \leq}(G)$ and $G \acting \Sub(G / H)$.
\begin{proposition}
\label{prop: f.g. normal subgr and Sub(Q)}
Let $1\to N\to G\overset{\phi}{\to} Q\to 1$ be a short exact sequence. 
 \begin{enumerate}
     \item 
 Then the bijection \[q:\left(\begin{array}{ccc}\Sub_{N\leq}(G)&\to &\Sub(Q) \\H\ \ \ &\mapsto& H/N\end{array}\right)\]  is a $G$-equivariant homeomorphism with respect to the subspace topology on $\Sub_{N\leq}(G)$.
 In particular, 
 \begin{equation} \label{eq: lifting the perfect kernel}
 q^{-1}(\PK(Q))\subseteq \PK(G).
 \end{equation}
\item Assume that $N$ is finitely generated. Then $\Sub_{N\leq}(G)$ is a clopen neighbourhood of $N$ in $\Sub(G)$.
	It follows that the restriction
\begin{equation} \label{eq: perfect kernel homeo}
q_{\restriction}: \PK(G)\cap \Sub_{N\leq}(G) \longrightarrow \PK(Q) 
\end{equation}
is a $G$-equivariant homeomorphism. 

If $H\in \Sub_{N\leq}(G)$, then the Cantor-Bendixson erasing ranks of $H$ in $G$ and of $q(H)=H/N$ in $Q$ are the same:

	\begin{equation}\label{eq: equality of rk CB X}
	\rkCBe(H;G)=\rkCBe(H/N;Q).
	\end{equation}

 In particular, if $\Sub(Q)$ is countable, and if  $H\in \Sub_{N\leq}(G)$, then
\begin{equation}\label{eq: rk CB X vs rk CB}
	\rkCBe(H;G)\leq \rkCB(Q).
	\end{equation}

 \end{enumerate}
 
\end{proposition}

\begin{proof}
Let us show that the bijection $q:\Sub_{N\leq}(G)\to \Sub(Q)$ is continuous.
If $\mathcal{V}_{Q}(\II,\OO)$ is a basic clopen set in $\Sub(Q)$, let  $\II'$ and $\OO'$ be any (bijective!) $\phi$-sections in $G$ of the finite sets $\II$ and $\OO$.
Then $q^{-1}\left(\mathcal{V}_{Q}(\II,\OO)\right)=\mathcal{V}_{G}(\II',\OO')\cap \Sub_{N\leq}(G)$. To see that \eqref{eq: lifting the perfect kernel} holds, recall that if $H \in q^{-1}(\PK(Q))$ then every open neighbourhood of $q(H)$ in $\Sub(Q)$ is uncountable. Since $q$ is a homeomorphism, it follows that every open neighbourhood of $H$ in $\Sub_{N \leq}(G)$ is uncountable and thus so is every open neighbourhood of $H$ in $\Sub(G)$. Hence $H \in \PK(G)$.

If $N$ is generated by the finite set $S$, then $\Sub_{N\leq}(G)$ coincides with the basic clopen set $\mathcal{V}_{G}(S,\emptyset)$.

If $Y$ is an open subset of a Polish space $X$ then a point $y \in Y$ is isolated with respect to the subspace topology on $Y$ if and only if $y$ is isolated in $X$. It follows by transfinite induction that, for every $y\in Y$ and every ordinal $\alpha$, we have $y\in Y^{(\alpha)}$ if and only if $y\in X^{(\alpha)}$. Thus $Y^{(\alpha)} = Y \cap X^{(\alpha)}$. Again by transfinite induction, one can show that, if $f: X \rightarrow Y$ is a homeomorphism between Polish spaces $X,Y$, then $f(X^{(\alpha)}) = Y^{(\alpha)}$ for all $\alpha$. In particular, $f(\PK(X)) = \PK(Y)$ and the Cantor-Bendixson erasing rank is invariant under homeomorphism.

Thus $q(\PK(G) \cap \Sub_{N \leq}(G)) = q(\PK(\Sub_{N \leq}(G))) = \PK(Q)$ and the restricted homeomorphism~\eqref{eq: perfect kernel homeo} is well-defined. The equality  \eqref{eq: equality of rk CB X} also follows since, for $H \in \Sub_{N \leq G}(G)$, 
\[ \rkCBe(H;G) = \infty \Leftrightarrow H \in \PK(G) \cap \Sub_{N \leq G}(G) \Leftrightarrow q(H) \in \PK(Q) \Leftrightarrow \rkCBe(H/N;Q) = \infty \]
and, if $H \notin \PK(G)$ then $H/N \notin \PK(Q)$ so

\begin{align*}
\rkCBe(H;G) &= \inf\left\{\alpha \geq 1 : H \notin \Sub(G)^{(\alpha)}\right\}\\
&= \inf\left\{ \alpha \geq 1: H \notin \Sub_{N \leq}(G)^{(\alpha)} \right\}\\
&= \inf\left\{ \alpha \geq 1 : q(H) \notin \Sub(Q)^{(\alpha)} \right\} \\
&= \rkCBe(H/N; Q).
\end{align*}
If $\Sub(Q)$ is countable, then $\PK(Q)=\emptyset$ and then $\rkCB(Q)=\sup\left\{\rkCBe(L;Q)\colon L\in \Sub(Q)\right\}$.
This proves the inequality~\eqref{eq: rk CB X vs rk CB}.
\end{proof}

Propositions \ref{prop:finite index subgr} and \ref{prop: f.g. normal subgr and Sub(Q)} imply that virtual algebraic fibering allows one to find finitely generated, infinite index subgroups which do not belong to the perfect kernel. More generally:

\begin{corollary} \label{cor: virtual fibre}
Let $G$ be a countable group and $H \leq G$ be a finitely generated subgroup. Suppose that there exists a finite index subgroup $G' \leq G$ and a finitely generated subgroup $H' \leq H \cap G'$ such that there is a short exact sequence 
\[ \begin{tikzcd}
1 \arrow[r] &H' \arrow[r] &G' \arrow{r} &Q \arrow[r] & 1
\end{tikzcd} \] 
where $\Sub(Q)$ is countable. Then $H \notin \PK(G)$.
\end{corollary}
\begin{proof}
By Proposition \ref{prop: f.g. normal subgr and Sub(Q)}, $\Sub_{H' \leq}(G')$ is a countable neighbourhood of $H'$. Let $P: G \rightarrow G'$ be defined by $P(L) = L \cap G'$. Then, by Proposition \ref{prop:finite index subgr}, $P^{-1}(\Sub_{H' \leq}(G'))$ is a countable neighbourhood of $H$.
\end{proof}

\subsection{Examples}
Note that the existence of a short exact sequence with finitely generated kernel does not in general provide any information about subgroups which do not contain the kernel. In fact, there exist pairs $N\trianglelefteq G$ with $N$ finite and normal in the countable group $G$ such that $\{\id\}\in \PK(G/N)$,
yet the subgroup $\{\id\}$ is isolated in $\Sub(G)$.

\begin{example}[Example with $\{\id\}$ isolated]\label{ex: ex. with isolated id}
In order to allow for the possibility that $\{\id\}$ be isolated, it is necessary that $G$ contains no infinite order element. 
A very interesting example is a central extension by $\Z/p\Z$ of the Burnside group $B(m,p):=\langle h_1, h_2, \cdots, h_m\vert g^p=1, \forall g\rangle$.
Let $m,p\in \N$, $m\geq 2$ and $p$ a prime number $\geq 665$.
There is a central extension 
\[
\begin{tikzcd}
1 \arrow[r] &\Z/p\Z \arrow[r] & G \arrow[r] & B(m,p) \arrow[r] & 1
\end{tikzcd}
\]
in which all the elements have order $p^2$ except the elements of the center $\Z/p\Z$.
Indeed, first take the Adian group $A(m,p)$ -- this is a torsion free $m$-generated group whose center is an infinite cyclic group $\la c\ra$ and with quotient $A(m,p)/\la c\ra\simeq B(m,p)$ -- and set $G=A(m,p)/\la c^p\ra$, see \cite[Example 1.25]{Zelenyuk-Ultraf-top-gp-2011}.
The $p$-th power of every element belongs to the center.
It follows that every non-trivial subgroup contains the center. The trivial subgroup is thus isolated in $\Sub(G)$.

On the other hand,  by a theorem of \v{S}irvanjan \cite{Sirvanjan--Burnside-gps-1976}, when $m\geq 2$  and $p$ is larger than $665$ and odd, $B(m,p)$ contains a copy of the countably infinite rank free Burnside group $B(\N,p)=\la s_i, i\in \N\vert g^p=1 \text{ for all } g\ra$.
Since $\{\id\}\in \PK(B(\N,p))=\Sub(B(\N,p))$, as shown in Proposition~\ref{Prop: K of count. free variety} below, then $\{\id\}\in \PK(B(m,p))$.
By Proposition~\ref{prop: f.g. normal subgr and Sub(Q)}, $ \Z/p\Z\in \PK(G)$.
\end{example}

Recall that the $\mathfrak{R}$-free group with generators indexed by $I$ of a variety $\mathfrak{R}$ of groups is the group $\mathfrak{F}^{\mathfrak{R}}_I=\la s_i, i\in I\ \vert\  \mathfrak R\ra$ where $\mathfrak R$ represents a collection of relators of the form $R(g_1,g_2, \cdots, g_p)=1$ for all $p$-tuples $(g_1,g_2, \cdots, g_p)$ of elements of the (usual) free group on $(s_i)_{i\in I}$.
Examples of $\mathfrak{R}$-free group $\mathfrak{F}^{\mathfrak{R}}_{\N}$ of countably infinite rank include the free group $\FF_\infty$, the free abelian group $\oplus_\N \Z$, the free abelian group $\oplus_\N \Z/m\Z$ of exponent $m$, the free nilpotent group of class $c$, the free Burnside group $B(\N,p)=\la s_i, i\in \N\vert g^p=1, \text{ for all } g\ra$,...
By {\em non-trivial} variety, we simply mean that $\mathfrak{F}^{\mathfrak{R}}_{\N}$ is not finitely generated.

\begin{proposition}[Free group of a variety]
\label{Prop: K of count. free variety}
Let $\mathfrak{F}^{\mathfrak{R}}_{\N}=\la s_1, s_2, \cdots \vert \mathfrak{R} \ra$ be the $\mathfrak{R}$-free group of countably infinite rank of a non-trivial variety $\mathfrak{R}$ of groups. 
Then 
$\PK(\mathfrak{F}^{\mathfrak{R}}_{\N})=\Sub(\mathfrak{F}^{\mathfrak{R}}_{\N})$.
\end{proposition}

\begin{proof}
For $I\subseteq \N$, consider the subgroup $H_I:=\la s_i\colon i\in I\ra\leq \mathfrak{F}^{\mathfrak{R}}_{\N}$, and let $\mathfrak{F}^{\mathfrak{R}}_I:=\la t_i \colon i\in I\vert  \mathfrak{R}\ra $ be the $\mathfrak R$-free group with generators indexed by $I$.

Consider the natural morphisms, well-defined by the universal property of the $\mathfrak R$-free groups, $\mathfrak{F}^{\mathfrak{R}}_{\N} \overset{u}{\to} \mathfrak{F}^{\mathfrak{R}}_I \overset{v}{\to}  \mathfrak{F}^{\mathfrak{R}}_{\N}$ defined by $u(s_i)=t_i$ if $i\in I$ and $u(s_i)=\id$ otherwise, and by $v(t_i)=s_i$. 

Since $v$ is the inverse of $u$  when restricted to $H_I$,
this ensures that $H_I\simeq \mathfrak{F}^{\mathfrak{R}}_I$ and $w_I:=v\circ u$ is a {\em retraction} $\mathfrak{F}^{\mathfrak{R}}_{\N}\to H_I$ (i.e. it is the identity on $H_{I}$).
Similarly, when $I\subseteq J$, the obvious morphism
$w_{I,J}: H_{J} \overset{}{\to}  H_{I}$ is a retraction.

Let $L$ be a finitely generated subgroup of $\mathfrak{F}^{\mathfrak{R}}_{\N}$. 
Then, there is a finite subset $I\subset \N$ such that $L\leq H_{I}$.
For every $i\in \N\smallsetminus I$, let $L_i=\la L, s_i\ra\leq H_{I\sqcup \{i\}}$. 
Observe, by non-triviality of the variety, that $s_i\not\in H_{I}$ and thus $L_i\not=L$.
We claim that the sequence $(L_i)_i$ tends to $L$ when $i$ tends to $\infty$:
If $g\in L_i$, then write it as a word $m$ using elements of $L$ and $s_i$.
When $j>i$, if $g$ continues to belong to $L_j$ then under the retraction $H_{I\sqcup \{i,j\}}\to H_{I\sqcup \{j\}}$, we see that the word $m$ continues to represent $g$ even after the erasing of the letters $s_i$. Thus $g$ can be written as a word using only elements of $L$, i.e. $g\in L$.
It follows that $L$ is a non-trivial limit of finitely generated subgroups of $\mathfrak{F}^{\mathfrak{R}}_{\N}$ and thus all finitely generated subgroups belong to $\PK(\mathfrak{F}^{\mathfrak{R}}_{\N})$. 
\\
Now, if $L$ is a non finitely generated subgroup of $\mathfrak{F}^{\mathfrak{R}}_{\N}$, then it is a limit of finitely generated groups. 
We have proved that $\PK(\mathfrak{F}^{\mathfrak{R}}_{\N})=\Sub(\mathfrak{F}^{\mathfrak{R}}_{\N})$.
\end{proof}

When $\mathfrak{F}^{\mathfrak{R}}_{\N}$ is abelian, its action on $\Sub(\mathfrak{F}^{\mathfrak{R}}_{\N})$ is trivial, thus not topologically transitive, 
while $\FF_\infty\acting \PK(\FF_\infty)$ is highly topologically transitive (see Corollary~\ref{cor: F infty act K(G) is top. r-transit}).

\begin{question}[Free Burnside groups and dynamics]
For the free Burnside groups, $B(\N,p)$ of countably infinite rank and $B(m,p)$ of rank $m$:
\begin{enumerate}
\item 
Is the action $B(\N,p)\acting \Sub(B(\N,p))$ topologically $r$-transitive for some $r$?
\item
    For $m\geq 2$ and $p$ such that $B(m,p)$ is infinite, what is perfect kernel $\PK(B(m,p))$? How transitive is the corresponding action of $B(m,p)$?
\end{enumerate}
\end{question}

\subsection{\texorpdfstring{Subgroups with Schreier graph quasi-isometric to $\N$ or $\Z$}{Subgroups with Schreier graph quasi-isometric to Z}} \label{section: subgroups QI to Z}

We now consider subgroups which are not necessarily normal but whose ``quotients'' are nonetheless equipped with a strong geometric structure: namely those with Schreier graph quasi-isometric to $\N$ or $\Z$. 

Recall that, given two metric spaces $(\mathcal{X},d_\mathcal{X})$, $(\mathcal{Y}, d_{\mathcal{Y}})$ and a constant $C \geq 1$, a map $f:\mathcal{X} \rightarrow \mathcal{Y}$ is a $C$-quasi-isometry if $\frac{1}{C} d_\mathcal{X}(x,y) -C\leq d(f(x), f(y))\leq C d_\mathcal{X}(x,y) +C$ for all $x,y \in \mathcal{X}$, and $d(z,f(\mathcal{X}))\leq C$ for all $z \in \mathcal{Y}$. The metric spaces $\mathcal{X}$ and $\mathcal{Y}$ are said to be quasi-isometric if there exists a $C$-quasi-isometry for some $C \geq 1$. This relation is symmetric (allowing a different constant $C$).

If $G$ is a finitely generated group, $H$ is a subgroup and $\mathcal{Y}$ is a metric space, we say that $H\bs G$ is quasi-isometric to $\mathcal{Y}$ if for a (equivalently any) Cayley graph $X$ of $G$ associated with a finite generating set, the \defin{Schreier graph}
$H\bs X$ is quasi-isometric to $\mathcal{Y}$.

\begin{theorem}
\label{th: H co-QI Z then CB-e-rk(H) <4} 
Let $H\leq G$ be a pair of finitely generated groups. If $H \bs G$ is quasi-isometric to $\mathbb{N}$, then $\rkCBe(H;G) \leq 2$, and if $H\bs G$ is quasi-isometric to $\Z$, then $\rkCBe(H;G) \leq 3$, where $\rkCBe(H;G)$ is the Cantor-Bendixson erasing rank of $H$ in $G$.
\end{theorem}

Note that the above theorem does not hold if one assumes only that $H \bs G$ is 2-ended. Indeed, if $G$ is a hyperbolic surface group and $H \leq G$ is a cyclic subgroup then $H \bs G$ is 2-ended but $H \in \PK(G)$. 
It also does not hold if one removes the assumption that $H$ is finitely generated: if $G$ is a non-abelian free group and $H$ is the kernel of a morphism onto $\Z$, then  $H \in \PK(G)$.  We are however working on extending this result to finitely generated subgroups $H$ such that $H \bs G$ is quasi-isometric to $\Z^d$ for $d \geq 2$.

\begin{example}
The trivial group $\{\id\}$ in $\Z$ and in the infinite dihedral group $D_\infty$ furnishes examples of subgroup with Schreier graph quasi-isometric with $\Z$ and Cantor-Bendixson erasing rank $2$ and $3$, respectively. 
Moreover any copy of $\Z / 2\Z$ in $D_\infty$ has Schreier graph quasi-isometric to $\N$. Examples of subgroups with Cantor-Bendixson erasing rank 1 and Schreier graphs quasi-isometric to $\N$ and $\Z$ can be constructing using Houghton's groups, see \cite[Section~6, Example~6.2]{Azuelos}.
\end{example}

Let $G$ be a group with finite generating set $S$.
Let $X:=\Cay(G,S)$ be the associated Cayley graph. 
The ball of radius $R$ centered at $g$ in $\Cay(G,S)$ is denoted by  $B_X(g,R)$.

For every tower of groups $H\leq K \leq G$, we shall consider the corresponding covering maps 
\[
\begin{tikzcd}
 X \arrow[r, "p_H"]\ & H\bs  X \arrow[r, "p_{K,H}"] & K \bs X  \arrow[r, "p_{G,K}"] & G\bs  X.
\end{tikzcd} \]
The subgroups $H$ and $K$ will not be assumed to be finitely generated in the following until the proof of Theorem~\ref{th: H co-QI Z then CB-e-rk(H) <4}.

\begin{proposition}
\label{prop: fibres w. bounded diameter}
Let $H\leq G=\la S\ra$ be any subgroup. 
Let $x_0\in H\bs  X$ and let $D>0$ be a positive integer. There are at most $2^{\vert B_X(\id ,D) \vert}$ intermediate subgroups
$H\leq K\leq G$ such that the fiber $p_{K,H}^{-1}(p_{K,H}(x_0))$ of $p_{K,H}:H\bs  X\to K\bs  X$ has diameter $\leq D$.
\end{proposition}

\begin{proof}
Consider any such $K$. The finite set $p_{K,H}^{-1}(p_{K,H}(x_0))=\{x_0, x_1, x_2, \cdots, x_n\}$ is contained in the ball $B(x_0,D)$ of $H\bs  X$.
Let $\tilde x_0\in  X$ be any pre-image of $x_0\in H\bs  X$ under the covering map $p_H:  X\to H\bs  X$. 
By lifting under $p_H$ some geodesic paths from $x_0$  to $x_i$ to geodesic paths from $\tilde x_0$ to $\tilde x_i$, one obtains a family of points $\{\tilde{x}_0 , \tilde{x} _1, \tilde{x}_2, \cdots, \tilde{x}_n\}$ such that 
$p_H(\tilde x_i)=x_i$. Each $\tilde x_i$ belongs to the ball $B_X(\tilde {x}_0,D)$ of $ X$.

Let $g_0=\id, g_1, g_2, \cdots, g_n\in G$ such that $g_i \tilde x_0=\tilde x_i$.
We claim that $K=\la H, g_1, g_2, \cdots, g_n\ra$.
\\
In fact,  $p_K(g_i\tilde x_0)=p_K(\tilde {x}_0)$ implies that $g_i\in K$. Thus $K\geq  \la H, g_1, g_2, \cdots, g_n\ra$.
\\
Since $p_K=p_{K,H}\circ p_H$, it follows that
\[K\cdot \tilde{x}_0=H \cdot \{\tilde{x}_0, \tilde{x} _1, \tilde{x}_2, \cdots, \tilde{x}_n\}=
\bigcup_{i=0}^n H \cdot g_i ( \tilde {x}_0)=(\bigcup_{i=0}^n H g_i) \cdot \tilde {x}_0.\]
Since the action $G\acting  X$ is free, it follows that $K\leq  \la H, g_1, g_2, \cdots, g_n\ra$.

Now let $A_D:=\{ \gamma_1, \gamma_2, \cdots, \gamma_m\}$ be the set of elements of $G$ such that $B_X(\tilde{x}_0,D)=\{\tilde {x}_0\}\cup \bigcup_{i=1}^m \{\gamma_i (\tilde{x}_0)\} \subseteq X$.
If $K \leq G$ is a subgroup containing $H$ such that the $p_{K,H}$-fiber of $p_{K,H}(x_0)$ has diameter $\leq D$, then 
$K=\la H, A_K\ra$ for some subset $A_K\subseteq A_D$. 
The number of such subgroups $K$ is thus bounded by the cardinality $2^{\vert B_X(\id, D)\vert}$ of the  power set of $B_X(\id, D)$.
\end{proof}

\begin{lemma}
\label{lem: 2 if delta1 meets the ball}
Consider subgroups $H\leq K\leq G$.
 Let $\bar x\in  K \bs X$ and $x_1,x_2\in  H\bs X$ such that $p_{K,H}(x_1)=p_{K,H}(x_2)=\bar x$.
Let $\bar \delta$ be a geodesic path in $ K  \bs X$ starting at $\bar x$. Its pull-back $\delta_1$ in $H\bs X$ starting at $x_1$ is again a geodesic.
If $\delta_1$ meets the ball $B_{H\bs X}(x_2,r)$ then $d(x_1,x_2)\leq 2r$.
\end{lemma}

\begin{proof} Let $z_2\in \delta_1\cap B_{H\bs X}(x_2,r)$. Then, since (successively) $\delta_1, \bar\delta$ are geodesics, $p_{K,H}(x_1)=p_{K,H}(x_2)$ and $p_{K,H}$ is 1-Lipschitz, we have
\[ d(x_1, z_2)=d(p_{K,H}(x_1),p_{K,H}(z_2))=d(p_{K,H}(x_2),p_{K,H}(z_2))\leq d(x_2,z_2)\leq r.\]
Thus, $d(x_1,x_2)\leq d(x_1, z_2)+d(z_2, x_2)\leq 2r$.
\end{proof}

Given $C \geq 1$, let $C_1 \coloneqq 3C^3 + C^2 + 3C$ and $C_2 \coloneqq 2C_1$.

\begin{lemma}
\label{lem: geod. towards xi remain at bounded distance} 
Consider subgroups $H\leq K\leq G$ and a constant $C \geq 1$.
Suppose that there exists a $C$-quasi-isometry $f:\mathcal{Y} \rightarrow H\bs  X$, where $\mathcal{Y} = \N$ or $\Z$. Assume that $K \bs X$ is infinite.
Let $\bar x\in  K \bs X$ and $x_1,x_2\in  H\bs X$ such that $p_{K,H}(x_1)=p_{K,H}(x_2)=\bar x$ and let $\bar \delta$ be an infinite geodesic path in $ K  \bs X$ starting at $\bar x$.
If the pull-backs $\delta_1, \delta_2$ starting at $x_1, x_2$ (respectively) point towards the same end $\xi$ of $H\bs  X$, then 
 \[d(\delta_1(t), \delta_2(t))\leq C_2\] for every $t\in [0,\infty)$ where $\delta_1, \delta_2:[0, \infty) \to H\bs  X$ are parametrised by arc length.
In particular, $d(x_1,x_2)\leq C_2$.
\end{lemma}

\begin{proof}
It follows from the definition of a $C$-quasi-isometry that, if $\mathcal{Y} = \N$ and $y \in H \bs G$ and $d(f(1),y) > C_1$ then $B(y,C_1)$ separates $H \bs X$ into at least two connected components; exactly one of these is infinite,
while the component containing $f(1)$ is finite.
Meanwhile if $\mathcal{Y} = \Z$ then, for all $y \in H \bs X$, $B(y,C_1)$ separates $H \bs X$ into at least two connected components, exactly two of which are infinite.

Relabel $x_1$ and $x_2$ if necessary so that $d(f(1),x_1) \leq d(f(1),x_2)$.
Suppose that $d(f(1),x_2) > C_1$ and $\delta_1\cap B(x_2,C_1)=\emptyset$. Then the ball $B(x_2,C_1)$ separates $H \bs X$ into at least two connected components and the end $\xi$ is contained in the connected component of $(H\bs  X)\smallsetminus B(x_2,C_1)$ that contains $\delta_1$, and thus $x_1$, $B(x_1,C_1)$ and $f(1)$. This implies in particular that $\mathcal{Y} = \Z$.
Therefore $B(x_1,C_1)$ separates $x_2$ from $\xi$. Thus $\delta_2\cap B(x_1,C_1)\not=\emptyset$.
In conclusion, if $d(f(1), x_2) > C_1$ then we have either $\delta_2\cap B(x_1,C_1)\not=\emptyset$ or $\delta_1\cap B(x_2,C_1)\not=\emptyset$.
By Lemma~\ref{lem: 2 if delta1 meets the ball}, we conclude that $d(x_1,x_2)\leq 2C_1 = C_2$. On the other hand, if $d(f(1),x_2) \leq C_1$ then $d(f(1),x_1) \leq C_1$ so $d(x_1,x_2) \leq C_2$.

More generally, since $p_{K,H}(\delta_1(t))=p_{K,H}(\delta_2(t))=\bar \delta(t)$ for all $t \in [0,\infty)$, the points $\delta_1(t), \delta_2(t)$ are the initial points of infinite geodesic paths to which the above argument applies.
\end{proof}

\begin{proposition}
\label{prop: G>K>H QI w. Z -> finite index}
Consider subgroups $H\leq K\leq G$.
Assume $H\bs  X$ is quasi-isometric to $\N$ or $\Z$.
Then either $[K:H]$ is finite or $[G:K]$ is finite.
\end{proposition}

\begin{proof}
Assume  $[G:K]$ is infinite (i.e. $K\bs  X$ is infinite) and let $C \geq 1$ be such that there exists a $C$-quasi-isometry from $\N$ or $\Z$ to $H \bs X$. 
Let $\bar x\in  K \bs X$ and let $\bar \delta$ be an infinite geodesic path in $ K  \bs X$ starting at $\bar x$.
If all the $p_{K,H}$-pull-backs of $\bar \delta$ point towards the same end of $H\bs  X$, then $p_{K,H}^{-1}(\bar x)$ has diameter bounded by $C_2$ (by Lemma~\ref{lem: geod. towards xi remain at bounded distance}).
Assume now that there are two points $u_1,u_2\in H \bs X$ such that  $p_{K,H}(u_i)=\bar x$ and such that the pull-backs of $\bar \delta$ starting at $u_1,u_2$ point towards different ends of $H\bs  X$. Then for all other points $u\in p_{K,H}^{-1}(\bar x)$, the pull-back of $\bar \delta$ starting at $u$ point towards the same end as either that starting at $u_1$ or that starting at $u_2$.
It follows that $d(u, \{u_1,u_2\})\leq  C_2$ (by Lemma~\ref{lem: geod. towards xi remain at bounded distance}).
In any case, $p_{K,H}^{-1}(\bar x)$ is finite. It follows that $[K:H]$ is finite.
\end{proof}

\begin{proposition}
\label{prop: finitely many subgroups for H or for K}
Consider a subgroup $H \leq G$ such that there exists a $C$-quasi-isometry $f: \mathcal{Y} \rightarrow H \bs X$ for some $C \geq 1$.
\begin{enumerate}
    \item If $\mathcal{Y} = \N$ then, for every intermediate subgroup $H\leq K\leq G$ such that $[K:H] < \infty$, 
    the fibers of $p_{K,H}$ have diameter $\leq C_2$, and there are only finitely many such subgroups $K$.
    \item Suppose that $\mathcal{Y} = \Z$ and consider an intermediate subgroup $H\leq K\leq G$ such that $[K:H] < \infty$. If there is a fiber of $p_{K,H}$ with diameter $> C_2$, then there are fibers of $p_{K,H}$ of arbitrarily large diameter and $K \bs G$ is quasi-isometric to $\N$.
\end{enumerate}
\end{proposition}

\begin{proof}
Let $K$ an intermediate subgroup such that $[K:H]<\infty$.
Suppose that there exists $\bar x\in K \bs  X$ and $x_1,x_2\in p_{K,H}^{-1}(\bar x)$ such that $d(x_1,x_2)>C_2$.
Let $\bar{\delta}$ be a geodesic path in $K \bs  X$ starting at $\bar x$ and let $\delta_1, \delta_2$ be its lifts in $H\bs  X$ starting at $x_1$ and $x_2$ respectively.
By Lemma~\ref{lem: geod. towards xi remain at bounded distance}, $\delta_1, \delta_2$ point towards different ends $\xi_1,\xi_2$, respectively, of $H\bs  X$. Therefore, if $p_{K,H}$ has a fiber with diameter $> C_2$ then $\mathcal{Y} \neq \N$. By Proposition~\ref{prop: fibres w. bounded diameter}, this implies that, if $\mathcal{Y} = \N$ then there are at most $2^{\vert B_X(\id ,C_2) \vert}$ such intermediate subgroups.

Assume from now on that $\mathcal{Y} = \mathbb{Z}$ and that there exists $\bar x\in K \bs  X$ and $x_1,x_2\in p_{K,H}^{-1}(\bar x)$ such that $d(x_1,x_2)>C_2$. Let $\bar \delta, \delta_1, \delta_2, \xi_1$ and $\xi_2$ be as above.
Then $p_{K,H}(\delta_1(t))=p_{K,H}(\delta_2(t))$ for every $t\in [0,\infty)$ while the distance 
$d(\delta_1(t),\delta_2(t))$ tends to $\infty$ as $t$ tends to $\infty$ : there is no bound on the diameter of the $p_{K,H}$-fibers. Moreover, $H \bs X$ is $(C_1 + d(x_1,x_2))$-quasi-isometric to $\delta_1 \cup \delta_2$ equipped with the metric induced from $H \bs X$, and it follows that $K \bs X$ is $(C_1+d(x_1,x_2))$-quasi-isometric to $\bar \delta$.
\end{proof}

Let us now give the proof of Theorem~\ref{th: H co-QI Z then CB-e-rk(H) <4}.

\begin{proof}[Proof of Theorem~\ref{th: H co-QI Z then CB-e-rk(H) <4}]
Let $H \leq G$ be a finitely generated pair of groups.
First suppose that $H \bs G$ is quasi-isometric to $\mathbb{N}$. If $\rkCBe(H;G) > 1$ then let $(H_n)_{n \in \mathbb{N}} \subseteq \Sub(G)$ be a sequence of pairwise distinct subgroups converging towards $H$. Since $H$ is finitely generated, we can assume, up to removing finitely many elements, that $H \leq H_n$ for all $n$. By Proposition~\ref{prop: finitely many subgroups for H or for K}, for all but finitely many $n \in \mathbb{N}$, we have $[H_n:H] = \infty$ which, by Proposition~\ref{prop: G>K>H QI w. Z -> finite index}, implies that $[G:H_n] < \infty$. Since $G$ is finitely generated, each such $H_n$ is isolated in $\Sub(G)$ so $H$ is isolated in $\Sub(G)'$ and $\rkCBe(H;G) = 2$.

Assume that $H \bs G$ is quasi-isometric to $\Z$ and $\rkCBe(H;G)>2$. Let $C \geq 1$ be such that there exists a $C$-quasi-isometry $\Z \rightarrow H \bs X$.
Let $(H_n)_{n \in \N}$ be a non-stationary sequence of subgroups of $G$ that converges to $H$ and such that $\rkCBe(H_n;G)\geq 2$ for all $n$.
Finite index subgroups of $G$ have Cantor-Bendixson erasing rank 1 so $[G:H_n]=\infty$ for all $n$.
Since $H$ is finitely generated, up to extraction one can assume that each $H_n$ contains $H$, and so by Proposition~\ref{prop: G>K>H QI w. Z -> finite index}, $[H_n:H]<\infty$.
By Proposition~\ref{prop: fibres w. bounded diameter}, there are only finitely many subgroups $H\leq K\leq G$ for which the fibers of $p_{K,H}$ have diameter $\leq C_2=2C_1$.
Thus, up to extraction one can assume that each $H_n$ gives rise to a $p_{H_n,H}$-fiber of diameter $>C_2$. By Proposition~\ref{prop: finitely many subgroups for H or for K} (2) this implies that $H_n \bs G$ is quasi-isometric to $\N$ and therefore $\rkCBe(H_n:G) \leq 2$. 
Hence every non-stationary approximation of $H$ must contain infinitely many subgroups with $\rkCBe(H_n:G) \leq 2$, so $H$ is isolated in $\Sub(G)''$ and $\rkCBe(H;G) \leq 3$.
\end{proof}

\section{\texorpdfstring{Subgroup separable groups and the first $\ell^2$-Betti number}{Subgroup separable groups and the first L2 Betti number}}
\label{sect: LERF and 1st L2 Betti}

Let $G$ be a countable group. A subgroup $H \leq G$ is \defin{separable} if it is equal to the intersection of a sequence of finite index subgroups of $G$, or equivalently if $H$ is a closed subset of $G$ with respect to the profinite topology. The group $G$ is said to be \defin{subgroup separable} if every finitely generated subgroup is separable.
Equivalently, $G$ is subgroup separable if and only if the subspace $\Sub_{[<\infty]}(G)$ of finite index subgroups of $G$ is dense in $\Sub(G)$, as observed in \cite{GKM}:
the closure of the finite index subgroups contains all the finitely generated subgroups and these are dense in $\Sub(G)$.
Applied to the subgroup $\{\id\}$ this property implies that $G$ is residually finite.
It is clear from the first definition that  {\em subgroup separability passes to subgroups}.
In addition it easily follows, for example via the profinite characterisation, that {\em subgroup separability is an invariant of commensurability} (alternatively, see \cite[Lemma 1.1]{Scott78}).

Assuming $G$ is a finitely generated subgroup separable group, the isolated points of $\Sub(G)$ are exactly the finite index subgroups; i.e. 
$H$ has finite index in $G$ if and only if $\rkCBe(H;G)=1$.
We present a condition under which $\Sub(G)'$ has no isolated points.
Here, $\rk(H)$ denotes the minimal number of generators of the group $H$.

\begin{theorem} \label{theorem: lerf + rk f.i. approx}
Let $G$ be a subgroup separable group. 
\begin{enumerate}[(a)]
\item  If $G$ is not finitely generated then $\Sub(G)$ is a perfect space.

\item \label{item: neighb w. greater rank}
If $G$ is finitely generated and every finitely generated $H\in  \Sub_{[\infty]}(G)$ admits a neighbourhood $\mathcal{U}_H$ where every subgroup $K\in \mathcal{U}_H$ of finite index in $G$ satisfies $\rk(K)\geq \rk(H)+2$. Then $\Sub_{[\infty]}(G)=\PK(G)$.

\end{enumerate}
 
 \end{theorem}

\begin{proof}[Proof of Theorem \ref{theorem: lerf + rk f.i. approx}]
First suppose that $G$ is not finitely generated and let $H \leq G$. If $H$ is finitely generated then $H$ has infinite index in $G$. Moreover, there is a sequence of finite index subgroups $(K_n)_n$ of $G$ such that $H=\cap_n K_n$. In other words, the sequence of finite index subgroups $H_i=\cap_{n\leq i} K_n$ converges toward $H$ and is non-stationary.
Therefore $H$ is not isolated in $\Sub(G)$. If $H$ is not finitely generated then it is a non-stationary limit of a sequence of finitely generated subgroups,
so $H$ is again not isolated. Thus $\Sub(G)$ is perfect.

Now suppose that $G$ is finitely generated.
Every finite index subgroup of $G$ is an isolated point of $\Sub(G)$ so $\PK (G) \subseteq \Sub_{[\infty]}(G)$. 
 Let $H \leq G$ be a subgroup with infinite index. If $H$ is not finitely generated and $H = \langle s_n : n \in \mathbb{N} \rangle$ then $H = \lim_{n \rightarrow \infty} \langle s_1, \dots, s_n \rangle$ and, for each $n$, $\langle s_1, \dots, s_n \rangle \leq H$ has infinite index so $H$ is not isolated in $\Sub_{[\infty]}(G)$. Next, suppose that $H$ is finitely generated, i.e. $\rk(H) \in \mathbb{N}$. 
By subgroup separability of $G$, the subgroup $H$ is the intersection of countably many finite index subgroups of $G$ and thus
$H$ is the limit of a strictly decreasing sequence $(K_n)_{n \in \mathbb{N}}$ of finite index subgroups of $G$.
For each $n$, let $H_n := \la H, k_n\ra$ for some $k_n \in K_n \smallsetminus H$. Since $H \lneq H_n \leq K_n$ for all $n$ it follows that $H_n \rightarrow H$ non trivially.
For large enough $n$, the subgroup $H_n$ belongs to $\mathcal{U}_H$ and is generated by $\rk(H)+1<\rk(H)+2$ elements;  it thus has infinite index in $G$ by Condition~\eqref{item: neighb w. greater rank}. Thus $H$ is not isolated in  $\Sub_{[\infty]}(G)$.
\end{proof}

Assume $G$ is finitely generated, subgroup separable, and admits a map $a:\Sub(G)\to \mathbb{R}$ such that $a(K)\leq \rk(K)$ and such that 
if $[G:K]\rightarrow \infty$ then $a(K) \rightarrow \infty$. Then Condition~\eqref{item: neighb w. greater rank} of Theorem~\ref{theorem: lerf + rk f.i. approx} is satisfied:
Since a finitely generated group $G$ has finitely many subgroups of any given index, any infinite index subgroup $H\leq G$ admits a sequence of neighbourhoods $(\mathcal{U}_j)_j$ such that the index in $G$ of all subgroups $K\in \mathcal{U}_j$ is greater than $j$. In particular, Condition~\eqref{item: neighb w. greater rank} of Theorem~\ref{theorem: lerf + rk f.i. approx} is satisfied when 
$G$ is finitely generated with positive first $\ell^2$-Betti number $\beta_1^{(2)}(G)$.
We refer the reader to \cite{CG-86} for the precise definition of this invariant, but note that it verifies the following properties:
(i) $\beta_1^{(2)}(H)\leq \rk(H)$ and (ii) if $H \leq G$ is a finite index subgroup, then $\beta_1^{(2)}(H) = [G: H] \ \beta_1^{(2)}(G)$ (cf. \cite[Proposition~2.6]{CG-86}).

\begin{corollary} \label{cor: lerf + beta1}
If $G$ is a finitely generated subgroup separable group
with positive first $\ell^2$-Betti number, then $\PK (G) = \Sub_{[\infty]} (G)$ and thus $\rkCB(G) = 1$.
\end{corollary}

\begin{remark} \label{rem: virtual fibring vs l2 Betti number}
Kielak showed in \cite{kielak} that if $G$ is a finitely generated group which is virtually residually finite rationally solvable (or RFRS) in the sense of \cite{Agol08} then $\beta_1^{(2)}(G) = 0$ if and only if $G$ virtually fibers (i.e. $G$ has a finite index subgroup which surjects onto $\mathbb{Z}$ with finitely generated kernel). This allows us to sharpen our result in the following sense: 

{\em If $G$ is finitely generated, subgroup separable and virtually RFRS then $\PK(G) = \Sub_{[\infty]}(G)$ if and only if $\beta_1^{(2)}(G) > 0$.}\\
One direction is given by Corollary~\ref{cor: lerf + beta1} while the other follows from Corollary \ref{cor: virtual fibre} since it implies that the kernel of a virtual fibering is an infinite index subgroup which is not contained in the perfect kernel of $G$.
In particular, we characterise the property of having maximal perfect kernel in terms of the first $\ell^2$-Betti number for a large class of subgroup separable groups, including all those which virtually embed into right-angled Artin groups.
\end{remark}

 The following corollary is a generalisation of a well-known fact about free groups (see e.g. \cite[Proposition 2.1]{CGLM-1-arxiv}).

\begin{corollary} \label{cor: virtually free}
If $G$ is a virtually free group of rank $\geq 2$ then 
\[ \PK (G) = 
\begin{cases}
\Sub_{[\infty]}(G) &\text{if } G \text{ is finitely generated} \\
\Sub(G) &\text{otherwise}.
\end{cases} \]
\end{corollary}
\begin{proof}
If $\FF_r$ is a free group of rank $r \geq 2$ then $\beta_1^{(2)}(\FF_r) = r -1 > 0$. Moreover it was shown in \cite{Hall-1949-subgr-free} that $\FF_r$ is subgroup separable. This implies that any group $G$ containing $\FF_r$ as a finite index subgroup is subgroup separable and has positive first $\ell^2$-Betti number so the result follows by Corollary~\ref{cor: lerf + beta1}.
\end{proof}

\begin{corollary}\label{cor: limit group}
If $G$ is virtually a finitely generated non-abelian limit group, then $\PK(G)=\Sub_{[\infty]}(G)$.
\end{corollary}
\begin{proof}
Limit groups are subgroup separable by \cite{Wilton} and a finitely generated, non-abelian limit group has positive first $\ell^2$-Betti number by \cite[p. 645, l. 2-3]{pichot-2006-semi-contin}. The result therefore follows from Corollary~\ref{cor: lerf + beta1}.
\end{proof}

 In the special case of surface groups, the perfect kernel was also described using a different proof in unpublished work by A. Carderi, the second author and F. Le Ma\^{i}tre.

 \begin{example}[Product of free groups over a cyclic subgroup] \label{ex: free product of free groups}
 Let $G= \FF_p \ast_\Z \FF_q$ for some finite $p,q \geq 2$. By \cite{Brunner-Burns-Solitar-1984}, free products of subgroup separable groups over cyclic subgroups are subgroup separable. Moreover, using the Mayer-Vietoris formula from \cite[p. 204, §4]{CG-86}, the first $\ell^2$-Betti number of $G$ satisfies $\beta^{(2)}_1(G) = p+q-2 > 0$ 
 so, by Corollary~\ref{cor: lerf + beta1}, $\PK(G) = \Sub_{[\infty]}(G)$.
 For information about the dynamics, see Corollary~\ref{cor: groups in CCC, K(G) and HTT}.
 \end{example}

\section{Actions on trees} \label{Section: actions on trees}

We will refer throughout to the geometric realisation of an abstract graph as a graph. More precisely, a \defin{graph} is a one-dimensional cell complex $X$. We denote by $V(X)$ its set of vertices (or 0-cells) and by $E(X)$ its set of edges (or 1-cells). For each $e \in E(X)$ let $\varphi_e: \partial e \rightarrow V(X)$ be the corresponding attaching map and define maps:
\[ o,t : E(X) \rightarrow V(X)\] 
where for each $e = [0,1] \in E(X)$ its \defin{endpoints} are the vertices $o(e) = \varphi_e(0)$ and $t(e) = \varphi_e(1)$. Note that $X$ is endowed with a natural orientation where each edge $e$ is directed from $o(e)$ to $t(e)$. A \defin{path} $P$ in $X$ is a sequence of edges $(e_n)_{n \in N}$, where $N = \mathbb{N}=\mathbb{N}_{>0}$ or $N = \{1,\dots,k\}$ for some $k \in \mathbb{N}$, such that $e_n = \gamma([n-1,n])$ for each $n \in N$ and for some continuous map $\gamma : \cup_{n \in N} [n-1, n] \rightarrow X$.
If $N$ has cardinality $k \in \mathbb{N}$ then $k$ is the \defin{length} of $P$ and if $N = \mathbb{N}$ we say that $P$ is \defin{infinite}. If $\gamma$ is injective then we say that $P$ is \defin{simple}.

If $X,Y$ are graphs and $\varphi: X \rightarrow Y$ is a bijection then $\varphi$ is an isometry if it preserves the cellular structure of $X$. It follows that $\varphi$ preserves the length metrics $d_X$ and $d_Y$ on $X$ and $Y$ respectively. All actions on graphs are assumed to be by orientation preserving isometries.  A \defin{tree} is a simply connected graph.

Recall that the action of a group $G$ on a tree $Y$ is \defin{minimal} if $Y$ does not contain any proper $G$-invariant subtree.
The action is \defin{reducible} if one of the following holds:
(1) there is a global fixed point; (2) $G$ fixes an end of $Y$; or
    (3) there is a $G$-invariant line in $Y$.
Otherwise, the action is \defin{irreducible} (aka of general type).
If $x$ is a vertex or an edge of $Y$, we denote its stabiliser under the action of a subgroup $H\leq G$ by $\Stab_H(x)$. 
The \defin{kernel} of an action $G\acting T$ is the normal subgroup consisting of the elements which fix every vertex of $T$.

The structure theorem of Bass-Serre Theory \cite[\S 5.4, Theorem 13]{Serre-Trees-1980} implies that a group $G$ admits a minimal action on a tree with a single orbit of edges if and only if it splits non-trivially as an amalgamated free product $G = A \ast_C B$ or an HNN extension $G = A \ast_C$. 
The $G$-tree $Y$ associated with such a splitting is called its Bass-Serre tree. 
In the former case $G$ acts on $Y$ with two orbits of vertices, $C$ is the stabiliser of an edge and $A,B$ are the stabilisers of its endpoints and in the latter $G$ acts with a single orbit of vertices, $C$ is the stabiliser of an edge and $A$ is the stabiliser of one of its endpoints.

We will make use of the following standard characterisation of irreducibility.

\begin{proposition}\label{prop: irreducibility vs indices}
Consider a minimal action $G\acting T$ on a tree that is transitive on the set of edges.
The action is irreducible if and only if, for some (equivalently every) edge $e$ of $T$ with endpoints $v$ and $w$, the indices of the stabilisers satisfy (up to relabelling $v$ and $w$):
	\begin{enumerate}
		\item \label{it: HNN} $[\Stab_{G}(v):\Stab_G(e)]\geq 2$ and $[\Stab_{G}(w):\Stab_G(e)]\geq 2$, when $G\bs T$ has one vertex.
		\item \label{it: amalgam} $[\Stab_{G}(v):\Stab_G(e)]\geq 3$ and $[\Stab_{G}(w):\Stab_G(e)]\geq 2$, when $G\bs T$ has two vertices.
	\end{enumerate}
\end{proposition}  
\begin{proof}
Recall that for an edge $e$ of $T$ and an endpoint $v$ of $e$, the set of edges in the $G$-orbit of $e$ and adjacent to $v$ is indexed by the cosets $\Stab_G(v)/\Stab_G(e)$. 

Assume $G\bs T$ has one edge and two vertices. If $[\Stab_{G}(v):\Stab_G(e)]= [\Stab_{G}(w):\Stab_G(e)]= 2$, then $T$ is a line and the action is reducible.
If $[\Stab_{G}(v):\Stab_G(e)]=1$, then $v$ has degree $1$ in $T$: this contradicts the minimality of the action.

Assume $G\bs T$ has one edge and one vertex. If $[\Stab_{G}(v):\Stab_G(e)]= 1$, then up to reversing the orientation of all the edges, every vertex of $T$ is the startpoint of a single edge: for every $x\in V(T)$, there is a single edge $f\in E(T)$ such that $o(f)=x$. There is a unique infinite simple path $P = (e_n)_{n \in \mathbb{N}}$ with startpoint $x$ that follows the orientations of the edges. For any two vertices, the corresponding infinite simple positive path must eventually agree. They thus define the same end of $T$. This uniquely defined end is fixed by $G$ and the action on $T$ is reducible.

We have proved that irreducibility of the action implies the desired conditions on the stabilisers.

Conversely,  if $T$ admits an edge $e$ with an endpoint $v$ such that the index of the stabiliser satisfies
$[\Stab_{G}(v):\Stab_G(e)]\geq 3$, it follows that $T$ is not a line: $v$ has degree $\geq 3$.
If $G\bs T$ has only one vertex and $[\Stab_{G}(v):\Stab_G(e)]= [\Stab_{G}(w):\Stab_G(e)]= 2$, then every vertex has degree $\geq 4$, thus $T$ is not a line.
There is no global fixed point or $G$-invariant line in $T$ by minimality of the action. 
Moreover, if $P = (e_n)_{n \in \mathbb{N}}$ is an infinite simple path starting at $v_0$ and $g \in \Stab_G(v_0) \smallsetminus \Stab_G(e_1)$ then $gP$ is an infinite path in $T$ whose intersection with $P$ is $\{v_0\}$. Therefore $P$ and $gP$ define distinct ends of $T$ so no end of $T$ is fixed by $G$.
\end{proof}

Let $T$ be a tree and $\varphi: T \rightarrow T$ be a graph automorphism. Recall that $\varphi$ is either \defin{elliptic} (i.e. $\varphi$ fixes a point in $T$) or \defin{hyperbolic} (i.e. the translation length of $\varphi$ is realised in $T$ and is strictly positive) \cite[Proposition~2.4]{Serre-Trees-1980}. In the later case, there is a unique bi-infinite line $\mathrm{Axis}(\varphi)$ in $T$ on which $\varphi$ acts by translation.

\begin{lemma}
\label{lem: existence of hyperbolic in Z}
Let $G\acting T$ be a minimal irreducible action of a countable group. Let $e$ be an edge of $T$ and $Z$ be one of the two connected components of $T\smallsetminus \{e\}$.  Let $L$ be a finite segment of edges $[f_1, f_2]$.
Then there are three hyperbolic elements ${\gamma}_0,{\gamma}_1, \gamma_2 \in G$ such that
\\-- $\mathrm{Axis} ({\gamma}_0)$ contains $e$;
\\-- $\mathrm{Axis}({\gamma}_1)$ is contained in $Z$;
\\--  $\mathrm{Axis}({\gamma}_0)\cap \mathrm{Axis}({\gamma}_1)$ is non-empty and bounded (the hyperbolic elements are ``transverse'');
\\-- the segment $\gamma_2(L)$ is contained in $Z$ and $\gamma_2\cdot f_1$ separates $\gamma_2\cdot f_2$ from $e$.
\end{lemma}

\begin{proof}
By shrinking to a point each connected component of $T\smallsetminus G\cdot e$, the complement in $T$ of the $G$-orbit of $e$, 
we obtain a tree $T'$ equipped with a $G$-action that is transitive on the set of edges and for which the map $T\to T'$ is $G$-equivariant. 
By minimality of the action $G\acting T$, the convex hull of $G\cdot e$ is the whole of $T$; thus  $G\acting T'$ is minimal and irreducible and both actions have the same kernel.
If ${\gamma}_0,{\gamma}_1$ satisfy the conditions of the lemma for $T'$, then their action on $T$ also satisfies these conditions. 

Thus, one can start by assuming that $T$ has a single orbit of edges. 
There is a hyperbolic element ${\gamma}_0$ whose axis contains $e$: 
by irreducibility $G\acting T$ admits some hyperbolic element $h_0$, and for any edge $f\in \mathrm{Axis}(h_0)$, there is $h_1\in G$ such that  $h_1\cdot f = e$; then $h_1(\mathrm{Axis}(h_0))$ is  the axis of ${\gamma}_0:=h_1 h_0 h_1^{-1}$ and it contains $e$.

Recall from \cite[I.6.4, p. 62]{Serre-Trees-1980} that two edges $f$ and $f'$ are \defin{coherent} if they are oriented in the same way in the unique geodesic segment which connects them.
Let $a,b$ be edges of $Z\smallsetminus \mathrm{Axis}({\gamma}_0)$ that are adjacent to $\mathrm{Axis}({\gamma}_0)$ but at different vertices of $\mathrm{Axis}({\gamma}_0)$.
If $a$ and $b$ are not coherent, one replaces $b$ by an edge $b'$ that is adjacent to $b$, that is separated from $\mathrm{Axis}({\gamma}_0)$ by $b$ and that is coherent with $a$ (every vertex is adjacent to at least two incoherent edges -- see Proposition~\ref{prop: irreducibility vs indices}).
Thus, one can take $a$ and $b$ to be coherent. Choose an element ${\gamma}_1\in G$ sending $a$ to $b$: then by \cite[Corollary, p. 63]{Serre-Trees-1980} ${\gamma}_1$ is hyperbolic and $[a,b]$ belongs to its axis. In particular $\mathrm{Axis}({\gamma}_1)$ has a non-trivial and bounded intersection with $\mathrm{Axis}({\gamma}_0)$. Moreover $[a,b]$ separates $\mathrm{Axis}({\gamma_1})\smallsetminus [a,b]$ from $e$, in particular, $\mathrm{Axis}({\gamma_1})\subset Z$.

There is an integer $r$ such that $\gamma_0^r(L)\subset Z$.
If $ \gamma_0^r\cdot f_1$ separates $\gamma_0^r\cdot f_2$ from $e$, set $\gamma_2\coloneqq \gamma_0^r$. 
Otherwise, the connected component $Z_2$ of $T\smallsetminus \{\gamma_0^r\cdot f_1\}$ that does not contain $\gamma_0^r\cdot f_2$ is entirely contained in $Z$.
Choose $\gamma'_1\in G$ whose axis is entirely contained in $Z_2$. 
Then the path from $\gamma'_1\gamma_0^r\cdot f_2$ to $e$ traverses $\gamma'_1\gamma_0^r\cdot f_1$ and $\gamma_2:=\gamma'_1\gamma_0^r$ satisfies the condition. It is easy to check that $\gamma_2$ is hyperbolic (the projection of $L$ onto $\mathrm{Axis}({\gamma}_0)$ belongs to the axis of $\gamma_2$). 
\end{proof}

\begin{proposition}\label{prop: finite kernel in Sub(Gamma)}
If $G\acting T$ is a minimal irreducible action on a tree with finite kernel $N$, then $\Sub(N)\subseteq\PK(G)$.
\end{proposition}
\begin{proof}
By \cite[Theorem 2.7]{Culler-Morgan-87} the irreducibility of the action implies that $G$ has a free subgroup $\FF_2$ of rank $2$. Therefore $\Sub(N) \subseteq \PK(G)$ by Proposition \ref{prop: finite normal sub}.
\end{proof}

\subsection{Main technical dendrological lemma}

\begin{lemma}[Main technical dendrological lemma]\label{main dendrological lemma}
Let $G\acting T$ be a minimal action on a tree $T$.
Let $\Lambda\leq G$ be a subgroup. Assume that $\Lambda$ admits a proper invariant subtree $T_\Lambda\subsetneq T$
for which there  are sequences $(H_n)_n$, $({S_{n}})_n$ and $(e_n)_n$ such that for every $n$
\begin{enumerate}
\item 
\label{it: Hn is a subgroup}
$H_n$ is a subgroup of $G$;

\item
\label{it: Sn is Hn-inv}
 ${S_{n}}$ is an $H_n$-invariant subtree of $T$; 
\item    
$e_n\in E(T)$ is an edge such that:
\begin{enumerate}
	\item 
	\label{it assumpt dist >n}
	$e_n$ is  at distance greater than $n+1$ from ${T_{\Lambda}}$ 
	
	\item 
	\label{it en separates}
	$e_n$ separates ${T_{\Lambda}}$ from ${S_{n}}$;
	
	\item 
	\label{it: assumpt stab e_n are equal} 
	$\Stab_{\Lambda}(e_n)=\Stab_{H_n}(e_n)$, thus $\Lambda\cap H_n=\Stab_{\Lambda}(e_n)$;
	
	\item 
	\label{it: assumpt stab e-n H-n not = H-n}
	 $\Stab_{H_n}(e_n)\not= H_n$;
	
	\item 
	\label{it en does not separates T-Lambda from the ek}
	The $\Lambda$-orbit of $e_n$ does not separate ${T_{\Lambda}}$ from the other $e_k$, $k\in \mathbb{N}\smallsetminus \{n\}$.
\end{enumerate}
\end{enumerate}
Let $\Lambda_n\coloneqq \langle \Lambda, H_n\rangle$. 
Let ${\mathcal L}_{T}$ be the set of subgroups of $G$ that satisfy assumptions 1-3.
Then: 
\begin{enumerate}[(I)]
\item \label{it: I}
Each $\Lambda_n$
splits as  $\Lambda_n= \Lambda*_{\Stab_{\Lambda}(e_n)} H_n$ and $\Stab_{\Lambda_n}(e_n)=\Stab_{\Lambda}(e_n)$. 
\\
More generally, $\Stab_{\Lambda_n}(e_k)=\Stab_{\Lambda}(e_k)$ for all $k\in \mathbb{N}$ and $n\in \mathbb{N}$.

\item \label{it: II}
The sequence $(\Lambda_n)_n$ tends non trivially to $\Lambda$.
\item \label{it: III}
Each $\Lambda_n$ belongs to ${\mathcal L}_{T}$, witnessed by a subsequence of  $(H_n)_n$, $({S_{n}})_n$ and $(e_n)_n$.

\item \label{it: IV} The set ${\mathcal L}_{T}$ has no isolated points for its induced topology. In particular ${\mathcal L}_{T}\subseteq \PK(G)$.

\item \label{it: V} For every $\Delta\in {\mathcal L}_{T}$, the quotient $\Delta\bs T$ is infinite.
\end{enumerate}
\end{lemma}

\begin{proof}

\textbf{Step \eqref{it: I}.} 
We fixed an integer $n$. We will first construct a tree $A_n$ on which the subgroup $\Lambda_n\coloneqq \langle \Lambda, H_n\rangle$  of $G$ acts. 
Let $v_n$ be the endpoint of $e_n$ nearest to ${T_{\Lambda}}$ and let $w_n$ be the other endpoint of $e_n$.

Denote by ${T^{*}_{n,\Lambda}}$ the connected component of $T\smallsetminus \{\Lambda\cdot e_n\}$ that contains ${T_{\Lambda}}$ and by ${T^{*}_{H_n}}$  the connected component of $T\smallsetminus \{H_n\cdot e_n\}$ that contains ${S_{n}}$.

Notice (by Assumption~\eqref{it en separates})  that ${T^{*}_{n,\Lambda}}$ and ${T^{*}_{H_n}}$ are disjoint subtrees of $T$  that are $\Lambda$-invariant and $H_n$-invariant respectively. Moreover, $e_n$ connects them.
Also observe for later use that by Assumption~\eqref{it en does not separates T-Lambda from the ek}, all the $e_k$, $k\in \mathbb{N}\smallsetminus \{n\}$ are contained in ${T^{*}_{n,\Lambda}}$.

By an immediate induction based on the length of an element of $\Lambda_n= \langle \Lambda, H_n\rangle$ as a product of elements from $\Lambda$ and $H_n$, we obtain that the $\Lambda_n$-orbit of $e_n$ splits $T$ into connected components that are either a $\Lambda_n$-image of ${T^{*}_{n,\Lambda}}$ or a $\Lambda_n$-image of ${T^{*}_{H_n}}$.
Contracting each of these connected components to a point, we obtain a $\Lambda_n$-equivariant map 
$$P_n:T\to A_n$$
from $T$ to a tree $A_n$ 
equipped with the induced $\Lambda_n$-action, with a single $\Lambda_n$-orbit of edges and two $\Lambda_n$-orbits of vertices. 

Let $\bar v_n \coloneqq P_n({T^{*}_{n,\Lambda}})$ and $\bar w_n \coloneqq P_n({T^{*}_{H_n}})$. Let $\bar e_n$ be the image of $e_n$; its endpoints are $\bar v_n$ and $\bar w_n$.
It follows from general Bass-Serre theory that $\Lambda_n$ splits as 
\begin{equation}\label{eq: amalgam decomp of Lambda n}
\Lambda_n=\Stab_{\Lambda_n}(\bar v_n)*_{\Stab_{\Lambda_n}(\bar e_n)}\Stab_{\Lambda_n}(\bar w_n).
\end{equation}
Our goal is now to identify this splitting with the splitting $\Lambda_n=\Lambda *_{\Stab_{\Lambda}(e_n)} H_n$ of \eqref{it: I}.

The projection map $P_n:T\to A_n$ is injective on the $\Lambda_n$-orbit of $e_n$, thus the stabilisers of $e_n$ and $\bar e_n$ are the same for $\Lambda$ and $H_n$, respectively:
\begin{equation*}
\Stab_{\Lambda}(\bar e_n)=\Stab_{\Lambda}(e_n) \ \ \ \text{ and } \ \ \ \Stab_{H_n}(\bar e_n)=\Stab_{H_n}(e_n).
\end{equation*}
Thus, by Assumption \eqref{it: assumpt stab e_n are equal}, $\Stab_{\Lambda}(\bar e_n)=\Stab_{H_n}(\bar e_n)$.
By definition $\Stab_{\Lambda_n}(\bar e_n)\cap \Lambda=\Stab_{\Lambda}(\bar e_n)$ and 
$\Stab_{\Lambda_n}(\bar e_n)\cap H_n=\Stab_{H_n}(\bar e_n)$.
It follows that:
\begin{equation}\label{eq: Lambda-n Stab intersct Lambda and H-n}
\Stab_{\Lambda_n}(\bar e_n)\cap \Lambda
=
\Stab_{\Lambda_n}(\bar e_n)\cap H_n=\Lambda\cap H_n.
\end{equation}

\medskip

Since ${T^{*}_{n,\Lambda}}$ is $\Lambda$-invariant and ${T^{*}_{H_n}}$ is $H_n$-invariant in $T$,  
the stabilisers of their images $P({T^{*}_{n,\Lambda}})=\bar v_n$ and $P({T^{*}_{H_n}})=\bar w_n$ 
for the $\Lambda_n$-action on $A_n$ satisfy 
\begin{equation}\label{eq: stab of vertices in An contain...}
\Lambda\leq \Stab_{\Lambda_n}(\bar v_n) \ \ \ \text{ and } \ \ \ H_n\leq \Stab_{\Lambda_n}(\bar w_n).
\end{equation}
We shall make use of the normal form theorem \cite[Theorem 2.6, Section IV.2, p. 187]{LS77} for the amalgamated free product splitting \eqref{eq: amalgam decomp of Lambda n}.
 Since $\la \Lambda, H_n\ra=\Lambda_n$, every non-trivial element $g\in \Stab_{\Lambda_n}(\bar v_n)$ can be written in the form $g=c_1c_2\cdots c_j$ where
 \begin{itemize}
 \item each $c_i$ is in one of the ``factor'' subgroups  $\Lambda$ or $H_n$,
 \item successive $c_i,c_{i+1}$ come from different factors, and
 \item if $j>1$ then no $c_i$ is in the intersection $ \Lambda\cap H_n$. 
 \end{itemize}
 Then (by \eqref{eq: stab of vertices in An contain...} and \eqref{eq: Lambda-n Stab intersct Lambda and H-n}) $c_1c_2\cdots c_j$ is a reduced sequence for the splitting \eqref{eq: amalgam decomp of Lambda n}.
On the other hand,  $g^{-1}c_1c_2\cdots c_j=1$ is not reduced for  the splitting \eqref{eq: amalgam decomp of Lambda n}. 
If $j>1$, then $g^{-1}c_1\in \Stab_{\Lambda_n}(\bar e_n)$, and successively we get $g^{-1}c_1c_2\cdots c_{j-1}\in \Stab_{\Lambda_n}(\bar e_n)$.
In the end, $c_j$ also belongs to $\Stab_{\Lambda_n}(\bar e_n)$, and by \eqref{eq: Lambda-n Stab intersct Lambda and H-n}, $c_j\in \Lambda\cap H_n$, a contradiction. Thus $j=1$ and $g=c_1\in \Lambda$ and $\Lambda= \Stab_{\Lambda_n}(\bar v_n)$.
 Similarly, the symmetric argument with the other factor shows that  $H_n=\Stab_{\Lambda_n}(\bar w_n)$.
We have proved that
\begin{equation}\label{eq: stabilisers in A-n}
\Stab_{\Lambda_n}(\bar v_n)=\Lambda, \ \ \ \ \Stab_{\Lambda_n}(\bar w_n)=H_n \ \ \ \ \text{ and } \ \ \ \ \Stab_{\Lambda_n}(e_n)=\Stab_{\Lambda}(e_n).
\end{equation}
and the splitting \eqref{eq: amalgam decomp of Lambda n} becomes
\begin{equation}
\Lambda_n=\Lambda *_{\Stab_{\Lambda}(e_n)} H_n.
\end{equation}

The ``More generally" part of \eqref{it: I}, i.e. $\Stab_{\Lambda_n}(e_k)=\Stab_{\Lambda}(e_k)$ for all  the other $k\in \mathbb{N}\smallsetminus\{n\}$, is proved later as Equation~\eqref{eq: stab e k in Lambda n}. 

\medskip
\noindent
\textbf{Step \eqref{it: II}.}
All the $\Lambda_n$ contain $\Lambda$, thus in order to show that $(\Lambda_n)_n$ tends to  $\Lambda$, it is enough to prove that for every $g\in G$, there is $n_0$ such that if $g\in \Lambda_n$ for some $n\geq n_0$, then $g\in\Lambda$.

Let $g\in G$, let $v$ be a vertex $v$ in ${T_{\Lambda}}$ and let $n_0$ be the distance in $T$ between $v$ and $g\cdot v$. 
Assume that $g\in \Lambda_n$ for some $n\geq n_0$.
By Assumption \eqref{it assumpt dist >n}
the $n$-neighbourhood of ${T_{\Lambda}}$ is contained in ${T^{*}_{n,\Lambda}}$, thus both $v$ and $g\cdot v$ belong  to ${T^{*}_{n,\Lambda}}$. It follows that their images in $A_n$ both equal $\bar v_n$. By equivariance of the quotient map $T\to A_n$, the element $g$ fixes the vertex $\bar v_n$. Since $\Stab_{\Lambda_n}(\bar v_n)=\Lambda$ (by Equation \eqref{eq: stabilisers in A-n}), it follows that $g \in \Lambda$. We have proved that $\Lambda_n \rightarrow \Lambda$ as $n \rightarrow \infty$.

By Assumption \eqref{it: assumpt stab e-n H-n not = H-n}, the injection $\Lambda\leq \Lambda_n=\Lambda*_{\Stab_{\Lambda}(e_n)} H_n$ is strict, so the convergence of the sequence $(\Lambda_n)_n$ is non-trivial.

\medskip
\noindent
\textbf{Step \eqref{it: III}.}
Let us show that $\Lambda_n\in {\mathcal L}_{T}$.

The invariant tree we propose for $T_{\Lambda_n}$ is the $\Lambda_n$-saturation of the convex hull of ${T_{\Lambda}}\cup {S_{n}}$:
$$T_{\Lambda_n}\coloneqq \Lambda_n\cdot \mathrm{Hull}({T_{\Lambda}}\cup {S_{n}}).$$
 
As for the sequences of subgroups $(H'_k)_k$, subtrees $(S'_k)_k$, and edges $(e'_k)_k$ associated to $\Lambda_n$, we will take subsequences of those associated with $\Lambda$, obtained by shifting the indices by $r_n$, the distance between ${T_{\Lambda}}$ and $v_n$,
i.e.  $r_n := d(T_\Lambda, v_n)$, $H'_k:=H_{k+r_n}$, $S'_k:=S_{k+r_n}$ and $e'_k:=e_{k+r_n}$.

As observed earlier, all the $e_k$, for $k\in \mathbb{N}\smallsetminus \{n\}$ are contained in ${T^{*}_{n,\Lambda}}$ by Assumption~\eqref{it en does not separates T-Lambda from the ek}.
If an element $g$ of $\Lambda_n$ were to send some $e_i$ to $g\cdot e_i$ so that $g\cdot e_i$ separated some $e_j$ from $T_{\Lambda_n}$ (for $i,j\not= n$), then $g\cdot e_i$ would also belong to ${T^{*}_{n,\Lambda}}$ and $g$ would fix the vertex $\bar v_n$ of $A_n$. Thus $g$ would belong to $\Lambda$  (by Equation \eqref{eq: stabilisers in A-n}). But Assumption~\eqref{it en does not separates T-Lambda from the ek} for $\Lambda$ and all $n$ prevents the $\Lambda$-orbit of $e_i$ from separating ${T_{\Lambda}}$ and $e_j$.
Thus the sequence $(e'_k)_k$ satisfies Assumption~\eqref{it en does not separates T-Lambda from the ek} for $\Lambda_n$, and in particular the $\Lambda_n$-invariant tree $T_{\Lambda_n}$ is strictly contained in $T$.
It also follows that $e'_k$ separates $T_{\Lambda_n}$ from $S'_k$, i.e. Assumption~\eqref{it en separates} for $\Lambda_n$.
Since $T_{\Lambda_n}\cap {T^{*}_{n,\Lambda}}$ is contained in the $r_n$-neighbourhood of ${T_{\Lambda}}$, this shift in the indices permits to check Assumption~\eqref{it assumpt dist >n} for $\Lambda_n$.

An element $g$ of $\Lambda_n$ that fixes some $e_k$, for $k\in \mathbb{N}\smallsetminus \{n\}$ (which again is an edge of ${T^{*}_{n,\Lambda}}$) belongs to the stabiliser of $\bar v_n$, thus to $\Lambda$ by Equation \eqref{eq: stabilisers in A-n}. Thus for all $k\in \mathbb{N}$:
\begin{equation}\label{eq: stab e k in Lambda n}
\Stab_{\Lambda_n}(e_k)=\Stab_{\Lambda}(e_k).
\end{equation}
Assumption~\eqref{it: assumpt stab e_n are equal} for $\Lambda$ ensures that $\Stab_{\Lambda}(e_k)=\Stab_{H_k}(e_k)$, so $\Stab_{\Lambda_n}(e_k)=\Stab_{H_k}(e_k)$.
Thus the sequence $(e'_k)_k$ satisfies Assumption~\eqref{it: assumpt stab e_n are equal} for $\Lambda_n$.
 Assumption~\eqref{it: assumpt stab e-n H-n not = H-n} for $\Lambda_n$, i.e.  $\Stab_{H_k'}(e'_k)\not= H_k'$ is immediate from the analogous condition  $\Stab_{H_{k+r_n}}(e_{k+r_n})\not= H_{k+r_n}$ for $\Lambda$. We have proved that $\Lambda_n\in {\mathcal L}_{T}$. 

Therefore, every point of ${\mathcal L}_{T}$ is a non-trivial limit of points of ${\mathcal L}_{T}$.  The proof of Item~\eqref{it: IV} follows immediately.

For $\Delta\in {\mathcal L}_{T}$ and any edge $e\in E(T)$, the distance from $e$ to $T_\Delta$ is an invariant of its $\Delta$-orbit $\Delta\cdot e$.
This is precisely the distance in $\Delta\bs T$ between $\Delta\bs T_\Delta$ and the image of $e$ in $\Delta\bs T$.
By Assumption\eqref{it assumpt dist >n}, the sequence of edges $(e_n)$ thus delivers infinitely many edges in $\Delta\bs T$.
This proves Item~\eqref{it: V}, and concludes the proof of the main technical dendrological Lemma~\ref{main dendrological lemma}.
\end{proof}

\subsection{Main dendrological theorems}

\begin{theorem} \label{th: a full subtree with stab in the kernel}
Let $G\acting T$ be a minimal and irreducible action with finite kernel $N$ on a tree $T$. 
Let $\Lambda\leq G$ be a subgroup and suppose that there exists an edge $e_\Lambda\in E(T)$ such that 
$\Stab_\Lambda(e_\Lambda)=\Lambda\cap N$ and 
 one of the connected components of  $T\smallsetminus \{e_\Lambda\}$ contains a $\Lambda$-invariant subtree $T_\Lambda \subseteq T$.
Then $\Lambda$ belongs to $\PK(G)$.
\\
More precisely, the set ${\mathcal N}_T$ of subgroups that satisfy the assumptions of this theorem for the action $G\acting T$ has no isolated points with respect to its induced topology.
\end{theorem}

\begin{proof}
 We will show that $\Lambda$ satisfies the conditions of Lemma~\ref{main dendrological lemma}.

Let $w_0$ be the endpoint of $e_\Lambda$ that is contained in the component $Z$ of $T \smallsetminus e_\Lambda$ which does not contain $T_\Lambda$, and let $w_1$ be the other endpoint of $e_\Lambda$.
Since $e_\Lambda$ separates $Z$ from a $\Lambda$-invariant subtree, it follows that if $u_1,u_2$ are vertices of $Z$ and $\lambda\in \Lambda$ such that $\lambda \cdot u_1=u_2$, 
then $\lambda$ fixes $w_0$ and $e_\Lambda$, so $\lambda\in \Lambda\cap N$. In particular $\lambda$ belongs to the kernel of the action, thus the quotient map $T\to \Lambda\bs T$ is injective on $Z$. Moreover, (taking $u_1=u_2$) for every vertex $v\in V(Z)$ and every edge $e\in E(Z)$ the stabilisers satisfy:
\begin{equation}\label{eq: stab L (v)= L cap N}
\Stab_\Lambda(v)=\Lambda\cap N=\Stab_\Lambda(e).
\end{equation}

 By Lemma~\ref{lem: existence of hyperbolic in Z}, there is a hyperbolic element $g$ whose axis is contained in $Z$. 
 Thus the sequence of edges $(e_n)_{n \in \N} \coloneqq (g^n\cdot e_\Lambda)_{n \in \N}$, satisfies two conditions of Lemma~\ref{main dendrological lemma}: \eqref{it assumpt dist >n} (distance to $T_{\Lambda}$) and \eqref{it en does not separates T-Lambda from the ek} (the $\Lambda$-orbit of the edge $e_n$ intersects $Z$ only in $e_n$, thus does not separate any other $e_k$ from $T_{\Lambda}$).

 By Lemma~\ref{lem: existence of hyperbolic in Z}, there is a hyperbolic element $h_n$ for each $n \in \mathbb{N}$ whose axis is contained in the connected component of $T\smallsetminus \{e_n\}$ that does not contain $T_\Lambda$.
We define the subtree $S_n$ as the axis of $h_n$.
The sequence $(e_n)_n$ satisfies \eqref{it en separates} of Lemma~\ref{main dendrological lemma}.

Since $N$ is finite and normal in $G$, there is a power $r_n$ of $h_n$ which commutes with $N$. 
Let $H_n$ be the subgroup generated by $h_n^{r_n}$ and $\Lambda\cap N$:
$$H_n\coloneqq \la h_n^{r_n}, \Lambda\cap N \ra.$$
Since $h_n^{r_n}$ is hyperbolic and commutes with $\Lambda \cap N$ and since $\Lambda \cap N$ is in the kernel of the action, $g \in H_n$ fixes $e_n$ (or any other edge) if and only if $g \in \Lambda \cap N$.
Thus $\Stab_{H_n} (e_n)=\Lambda\cap N=\Stab_\Lambda(e_n)$.
This shows condition~\eqref{it: assumpt stab e_n are equal} of Lemma~\ref{main dendrological lemma}.

As for condition~\eqref{it: assumpt stab e-n H-n not = H-n}, it is clear since $\Stab_{H_n} (e_n)$ is contained in the kernel of the action $G\acting T$ while $H_n$ contains the hyperbolic element $h_n^{r_n}$.
Thus $\Lambda$ satisfies the conditions of Lemma~\ref{main dendrological lemma}.

By Conclusion~\eqref{it: I} of Lemma~\ref{main dendrological lemma}, the group $\Lambda$ is the non-trivial limit of $$\Lambda_n\coloneqq \la \Lambda, H_n\ra=\Lambda*_{(\Lambda\cap N)} H_n.$$ 
Since $\Lambda_n\cap N$ is contained in the edge stabiliser $\Lambda\cap N$ which is in turn contained in $N$, we have $\Lambda\cap N=\Lambda_n\cap N$.
By Conclusion~\eqref{it: III} of Lemma~\ref{main dendrological lemma}, $\Lambda_n$ belongs to ${\mathcal L}_{T}$ for a subsequence of the $(e_n)_n$ and a proper subtree $T_{\Lambda_n}\subsetneq T$ that does not contain these edges. For any $k \in \mathbb{N}$, we have $\Stab_{\Lambda_n}(e_k)=\Stab_{\Lambda}(e_k)$ (by the ``More generally" part of Conclusion~\eqref{it: I} of Lemma~\ref{main dendrological lemma}) and, since all the $e_k$ belong to $E(Z)$, Equation~\eqref{eq: stab L (v)= L cap N} gives $\Stab_{\Lambda_n}(e_k)=\Lambda\cap N=\Lambda_n\cap N$.
Thus each $\Lambda_n$ satisfies the assumptions of Theorem~\ref{th: a full subtree with stab in the kernel}, thereby completing the proof that ${\mathcal N}_T$ has no isolated points. This concludes the proof of Theorem~\ref{th: a full subtree with stab in the kernel}.
\end{proof}

If $G\acting T$ is a minimal and irreducible action on a tree, let us denote by
\[\Sub_{\vert \bullet\bs T\vert \infty}(G):=\{\Lambda \in \Sub(G)\colon \left\vert \Lambda\bs T\right\vert=\infty\}\]
the set of subgroups whose action on $T$ has infinitely many orbits of edges. Observe that $\Sub_{\vert \bullet\bs T\vert \infty}(G)$ as well as its closure $\overline{\Sub_{\vert \bullet\bs T\vert \infty}(G)}$ are invariant under the action $G\acting \Sub(G)$.

\begin{remark}\label{rem: H fg minimal -> T/H compact}
Bass-Serre Theory asserts that if $\Lambda\acting T $ is a \emph{minimal} action and if $\Lambda$ is finitely generated, then the quotient graph $\Lambda\bs T$ is compact.
Indeed, by writing $\Lambda\bs T$ as an increasing union of compact connected subgraphs $\kappa_n$, we can write $G$ as the increasing union of the subgroups $G_n$ corresponding to the graphs of groups of $\kappa_n$. Since $G$ is finitely generated, this sequence of subgroups is stationary: $G_n=G_{n_0}$ for $n$ greater than some $n_0$. The pull-back of $\kappa_{n_0}$ is a $G$-invariant cocompact subtree which coincides with $T$ by minimality.
\end{remark}

\begin{question}
What is the class of groups and minimal irreducible actions $G\acting T$ such that $\Lambda\leq G$ has infinite index if and only if $\Lambda\bs T$ is infinite, i.e. for which \[\Sub_{\vert \bullet\bs T\vert \infty}(G)=\Sub_{[\infty]}(G) ?\]
Compare with Remarks~\ref{rem: when G bs T is infinite} and~\ref{Rem: acylindrical act. examples}, and with Proposition~\ref{claim: L/T finite -> L/G finite}.
\end{question}
Thinking about this problem led us to the following question.
\begin{question}
Let $H,A$ be infinite subgroups of the countable group $G$. Assume that $A$ is amenable and that $\beta_1^{(2)}(G)>0$ and 
$\beta_1^{(2)}(H)<\infty$. If $A\bs G /H$ is finite, does this imply that $H$ has finite index in $G$?
\end{question}

\begin{theorem}[A segment with finite stabiliser] 
\label{th: two edge-stab with finite intersection}
Assume $G\acting T$ is a minimal and irreducible action on a tree. 
Assume that there are two edges $f_1,f_2$ of $T$ such that $\Stab_G(f_1)\cap \Stab_G(f_2)$ is finite.
Then the set $\Sub_{\vert \bullet\bs T\vert \infty}(G)$ has no isolated point, thus its closure $\overline{\Sub_{\vert \bullet\bs T\vert \infty}(G)}$ is a perfect set. In particular:
 \[\overline{\Sub_{\vert \bullet\bs T\vert \infty}(G)}\subseteq \PK(G).\]
\end{theorem}
\begin{remark}\label{rem: when G bs T is infinite}
Observe that $G\bs T$ may be infinite in the above theorem and in this case \[{\Sub_{\vert \bullet\bs T\vert \infty}(G)} = \Sub(G)=\PK(G).\]
On the contrary, if $G\bs T$ is finite, then ${\Sub_{\vert \bullet\bs T\vert \infty}(G)}\subseteq \Sub_{[\infty]}(G)$.
\end{remark}

The subset $\overline{\Sub_{\vert \bullet\bs T\vert \infty}(G)}$ is invariant under the conjugacy $G$-action. 
We will see that it moreover admits a natural finite partition into perfect $G$-invariant topologically transitive pieces.
Indeed, recall from Proposition~\ref{prop:finite normal subgr clopen part} that for any
finite normal subgroup $N$ there is
a finite clopen partition of $\Sub(G)$ indexed by the set $\Conj(N)^{G}$ of $G$-conjugacy classes of subgroups of $N$.
This partition induces a $G$-invariant partition of $\overline{\Sub_{\vert \bullet\bs T\vert \infty}(G)}$:
\[\overline{\Sub_{\vert \bullet\bs T\vert \infty}(G)} = \bigsqcup_{C \in \Conj(N)^{G}}{\mathcal F}_C^{T}
 \]
 where, for each 
 $C \in \Conj(N)^G$, the subset  
 \begin{equation}\label{eq: def F-C}
{\mathcal F}^{T}_C:=\overline{\Sub_{\vert \bullet\bs T\vert \infty}(G)}\cap \{\Lambda\in \Sub(G)\colon \Lambda\cap N\in C\}
\end{equation}
 is a perfect set by Theorem~\ref{th: two edge-stab with finite intersection}.
 \begin{theorem}[Topological transitivity]
 \label{th: two edge-stab with finite intersection - top. dyn.}
Suppose that $G\acting T$ is a minimal and irreducible action on a tree. 
Assume that there are two edges $f_1,f_2$ of $T$ such that $\Stab_G(f_1)\cap \Stab_G(f_2)$ is finite, thus the kernel $N_0$ of the action $G\acting T$ is finite.
Then 
\begin{enumerate}
\item[(i)] For each $C\in \Conj(N_0)^{G}$ the action $G\acting {\mathcal F}^{T}_C$ is topologically transitive, and
\item[(ii)]  If $C \in \Conj(N_0)^G$ has cardinality 1 then the action $G \acting {\mathcal F}_C^{T}$ is highly topologically transitive.
\end{enumerate}
 In particular, if $N_0$ is trivial, then the action $G\acting \overline{\Sub_{\vert \bullet\bs T\vert \infty}(G)}$ is highly topologically transitive.
\end{theorem}

\begin{remark}
The condition ``{\em $C \in \Conj(N_0)^G$ has cardinality $1$}'' means that $M\in C$ is normalised by $G$; it  
is satisfied for instance when $N_0$ is cyclic: a subgroup $M\leq N_0$ is uniquely defined by its cardinality. 
It is also satisfied for some finite index subgroup $G'\leq G$ that acts trivially on $\Sub(N_0)$ by conjugation. Then $G'$ satisfies Theorem~\ref{th: two edge-stab with finite intersection - top. dyn.}(ii), but for a finer partition.
\end{remark}

\begin{remark}[Obstruction to topological $2$-transitivity]\label{rem: obstruction to r-top transitive}
Note that the condition is necessary in part (ii) of Theorem~\ref{th: two edge-stab with finite intersection - top. dyn.}. Indeed, if $C$ contains at least two distinct elements $M,M'$, then the action $G\acting {\mathcal F}^{T}_C$ is not topologically $2$-transitive:
Consider the four non-empty open sets $V_1=V_2=V_3={\mathcal F}^{T}_C\cap \{\Lambda\in \Sub(G)\colon \Lambda\cap N_0=M\}$ and $V_4={\mathcal F}^{T}_C\cap \{\Lambda\in \Sub(G)\colon \Lambda\cap N_0=M'\}$. For every element $\gamma\in G$ such that $\gamma V_1\cap V_3\not=\emptyset$, then $\gamma M\gamma^{-1}=M$ thus $\gamma V_2\cap V_4=\emptyset$.
\end{remark}

\begin{question}
Does the conclusion of the above  Theorem~\ref{th: two edge-stab with finite intersection} still hold if we replace the assumption about the edge stabilisers with the assumption that the kernel $N_0$ of $G\acting T$ is finite and the action of $G/N_0$ on the boundary of $T$ is topologically free? François Le Maître pointed out that Baumslag-Solitar groups $\mathrm{BS}(m,n)$, 
$2\leq \vert m\vert <n$ give counter-examples to Theorem~\ref{th: two edge-stab with finite intersection - top. dyn.} under this weaker assumption, see  \cite[Theorem B]{CGLMS-1-arxiv}. 
\end{question}

\begin{proof}[Proof of Theorem~\ref{th: two edge-stab with finite intersection}]
Let us start by showing that, up to replacing $f_2$ with another edge, we can assume that $\Stab_G(f_1)\cap \Stab_G(f_2)=N_0$, the  kernel of $G\acting T$.
In the connected component $Z_2$ of $T\smallsetminus \{f_2\}$ that does not contain $f_1$, let us choose an edge $f_3$ such that the intersection $\Stab_G(f_1)\cap \Stab_G(f_3)$ has minimal cardinality; this is possible by the finiteness of $\Stab_G(f_1)\cap \Stab_G(f_2)$ and because for every edge $f\in Z_2$, the edge $f_2$ belongs to the segment $[f_1,f]$. Let us call $P:=\Stab_G(f_1)\cap \Stab_G(f_3)$.
Let $Z_3$ be the connected component of $T\smallsetminus \{f_3\}$ that does not contain $f_1$.
Every edge $f$ in $Z_3$ satisfies 
\begin{equation}\label{eq: stability of Stab}
\Stab_G(f_1)\cap \Stab_G(f)=P.
\end{equation}
By Lemma~\ref{lem: existence of hyperbolic in Z}, there is an element $\gamma_0\in G$ whose axis contains $f_3$:  a ``half of the axis'' of $\gamma_0$ is in $Z_3$. 
Up to replacing $\gamma_0$ by its inverse, one can assume that $\gamma_0^{-1}\cdot f_3$ belongs to $Z_3$.

For every edge $e$ of $T$, there is a positive integer $s$ such that $e\subseteq \gamma_0^r\cdot Z_3$ for every $r\geq s$. Indeed, first note that $Z_3 \subseteq \gamma_0^r Z_3$ for all $r \geq 1$. Now if $e \notin Z_3$, let $y$ be the projection of $e$ on the axis of $\gamma_0$, i.e. the vertex of $\mathrm{Axis}(\gamma_0)$ that is the closest to $o(e)$. For every large enough positive power $r$ of $\gamma_0$, the vertex $y$ lies between $f_3$ and $\gamma_0^r\cdot f_3$; and then $y$ and $e$ belong to $\gamma_0^{r+1}\cdot Z_3$. 
In particular, for every $r$ larger than some $s_0$, the segment $[f_1,f_3]$ is contained in $\gamma_0^r\cdot Z_3$.
Every edge $e$ in $\gamma_0^r\cdot Z_3$ satisfies $\Stab_G(\gamma_0^r \cdot f_1)\cap \Stab_G(e)=\gamma_0^r P \gamma_0^{-r}$, the $\gamma_0^r$-translated version of \eqref{eq: stability of Stab}. In particular, for the edges $f_1,f_3$ and $r\geq s_0$:
$\Stab_G(\gamma_0^r \cdot f_1)\cap \Stab_G(f_1)=\gamma_0^r P \gamma_0^{-r}=\Stab_G(\gamma_0^r \cdot f_1)\cap \Stab_G(f_3)$. It follows that 
$\gamma_0^r P \gamma_0^{-r}\leq \Stab_G(f_1)\cap \Stab_G(f_3)=P$.
But $P$ and $\gamma_0^r P \gamma_0^{-r}$ have the same finite cardinal, thus $P=\gamma_0^r P \gamma_0^{-r}$. It follows that $P$ fixes every edge of $T$ so that $P\leq N_0$.
Since the kernel $N_0$ in contained in both $\Stab_G(f_1)$ and $ \Stab_G(f_3)$, we also have $P\geq N_0$.
We have proved that
\begin{equation}
\Stab_G(f_1)\cap \Stab_G(f_3)=N_0.
\end{equation}
So, replace $f_2$ with $f_3$.

\medskip
If $\Lambda\in \Sub_{\vert \bullet\bs T\vert \infty}(G)$ is not finitely generated, then $\Lambda$ is a limit of finitely generated subgroups $\Lambda_n\leq \Lambda$ which satisfy a fortiori that $\Lambda_n\bs T$ is infinite.
This is the case, in particular, if $G\bs T$ is not finite: by minimality, $G$ is not finitely generated.
 
It is thus enough to prove that $\Lambda\in \PK(G)$ when $\Lambda\in \Sub_{\vert \bullet\bs T\vert \infty}(G)$ is finitely generated, by showing that $\Lambda$ satisfies the conditions of Theorem~\ref{th: a full subtree with stab in the kernel}.
Since $\Lambda$ is finitely generated, it admits a cocompact (thus proper) invariant subtree $T_\Lambda\subsetneq T$ (see Remark~\ref{rem: H fg minimal -> T/H compact}).

\begin{claim}\label{Claim: f1-f2 in Z}
 Let $f$ be any edge of $T\smallsetminus T_\Lambda$ and let $Z_f$ be the connected component of $T\smallsetminus \{f\}$ that does not contain $T_\Lambda$.
 Let $L$ be the segment $[f_1, f_2]$. One can assume, up to replacing $L$ with some $G$-translate, that $L\subseteq Z_f$ and that $f_1$ separates $f_2$ from $f$ and $T_\Lambda$.
\end{claim}
\begin{proof}[Proof of Claim~\ref{Claim: f1-f2 in Z}]
\renewcommand{\qedsymbol}{$\diamond$}
Since $N_0$ is normal, for any $\gamma\in G$ the pair $(\gamma\cdot f_1, \gamma\cdot f_2)$ satisfies the same assumption as $(f_1,f_2)$. 
Then Claim~\ref{Claim: f1-f2 in Z} is a direct application of the fourth item of Lemma~\ref{lem: existence of hyperbolic in Z} (existence of $\gamma_2$).
\end{proof}

\begin{claim}\label{claim: Stab e-Lambda}
Suppose that $f, Z_f$ and $L$ are as in the Claim~\ref{Claim: f1-f2 in Z} and define the edge $e_\Lambda:=f_2$. Then
 \begin{equation}
\Stab_\Lambda(e_\Lambda)= \Lambda\cap N_0.
\end{equation}
\end{claim}
\begin{proof}[Proof of Claim 2]
\renewcommand{\qedsymbol}{$\diamond$}
If $\lambda$ belongs to $\Stab_\Lambda(e_\Lambda)$, then (since $T_\Lambda$ is $\Lambda$-invariant) the whole path from $e_\Lambda$ to $T_\Lambda$ must be point-wise fixed by $\lambda$. 
This path uses the subpath $L$ whose edges $f_1$ and $f_2$ satisfy $\Stab_\Lambda(f_1)\cap \Stab_\Lambda(f_2)\leq  \Stab_G(f_1)\cap \Stab_G(f_2)=N_0$.
Thus $\lambda\in \Lambda\cap N_0$, i.e. $\Stab_\Lambda(e_\Lambda)\leq \Lambda\cap N_0\leq  \Stab_\Lambda(e_\Lambda)$ since $N_0$ is the kernel of the action. Claim~\ref{claim: Stab e-Lambda} is proved.
\end{proof}
\renewcommand{\qedsymbol}{$\Box$}

We have proved that $\Lambda$ (when finitely generated) satisfies the conditions of Theorem~\ref{th: a full subtree with stab in the kernel}, i.e. $\Lambda\in {\mathcal N}_T$, in the terminology of that theorem. It follows that $\overline{\Sub_{\vert \bullet\bs T\vert \infty}(G)}\subseteq \overline{{\mathcal N}_T}$.
Observe that for every $\Lambda'\in {\mathcal N}_T$, the quotient $\Lambda'\bs T$ is infinite (one of the components of $T \smallsetminus \{e_{\Lambda'}\}$ is sent injectively in $\Lambda'\bs T$), thus ${\mathcal N}_T\subseteq \Sub_{\vert \bullet\bs T\vert \infty}(G)$.
Hence 
\begin{equation}
\overline{\Sub_{\vert \bullet\bs T\vert \infty}(G)}= \overline{{\mathcal N}_T}.
\end{equation}
Recall from Theorem~\ref{th: a full subtree with stab in the kernel} that ${\mathcal N}_T$ has no isolated points with respect to its induced topology.
This concludes the proof of Theorem~\ref{th: two edge-stab with finite intersection}.
\end{proof}

We now prove Theorem~\ref{th: two edge-stab with finite intersection - top. dyn.}, keeping the same notation and numbering of claims.

\begin{proof}[Proof of Theorem~\ref{th: two edge-stab with finite intersection - top. dyn.}]
Fix $C \in \Conj(N_0)^G$. 

In order to prove (i), we will have to show that if  $V_1, V_2$ are two arbitrary non-empty open sets of ${\mathcal F}^{T}_C$, then there is $\gamma\in G$ such that $\gamma V_1\gamma^{-1}\cap V_2\not=\emptyset$.

Finitely generated subgroups are dense in $\Sub_{\vert \bullet\bs T\vert \infty}(G)$ and $\{\Lambda\in \Sub(G)\colon \Lambda\cap N_0\in C\}$ is open, so there exist finitely generated $\Lambda_i \in V_i$ for $i = 1,2$. 
By Theorem~\ref{th: two edge-stab with finite intersection}, $\Sub_{\vert \bullet\bs T\vert \infty}(G)$ has no isolated points, so we can assume that $\Lambda_1, \Lambda_2 \notin \Sub(N_0)$. 
Moreover, we can assume that $ \Lambda_1 \cap N_0 = \Lambda_2 \cap N_0=:M$. Indeed, if $M_i := \Lambda_i \cap N_0$ for $i = 1,2$ then there exists $\gamma_2\in G$ such that $\gamma_2  M_2 \gamma_2^{-1}=M_1$, and we can replace $V_2$ by $V'_2:=\gamma_2 V_2\gamma_2^{-1}$: if there exists $\gamma$ such that $\gamma V_1\gamma^{-1}\cap V'_2 \not=\emptyset$ then $\gamma_2^{-1} \gamma V_1 \gamma^{-1} \gamma_2 \cap V_2 \neq \emptyset$.

To prove part (ii), we need to show that, if $C = \{M\}$ (i.e. $M$ is normal in $G$) then, for all positive integers $r$ and for all non-empty open sets $V_1, \dots, V_{2r} \subseteq \mathcal{F}_C$, there exists $\gamma \in G$ such that $\gamma V_i \cap V_{r+i} \neq \emptyset$ for each $i \in \{1, \dots, r\}$. As above, we can pick $\Lambda_i\in V_i$ that is finitely generated and not contained in $\Sub(N_0)$. Moreover $\Lambda_i \cap N_0 = M$ for each $i$ by definition. Thus both part (i) and part (ii) follow from the following statement: 

\medskip
{\em Fix $C \in \Conj(N_0)^G$ and $M \in C$ (we don't assume that $|C|=1$). Let $r \geq 1$ and $V_1, \dots V_{2r} \subseteq \mathcal{F}_C$ be open subsets such that, for each $i \in \{1, \dots, 2r\}$, there exists a finitely generated subgroup $\Lambda_i \in V_i$ such that $\Lambda_i \cap N_0 = M$ and $\Lambda_i \neq M$. Then there is an element $\gamma \in G$ such that $\gamma V_i \cap V_{r+i} \neq \emptyset$ for each $i \in \{1, \dots, r\}$.} 

\medskip

Let us prove that this statement indeed holds. For any subgroup $H\leq G$ let $T_H$ denote a 
minimal $H$-invariant subtree of $T$.
Recall that either $T_H$ is a single vertex or $T_H$ is the union of the axes of the hyperbolic elements of $H$.
\begin{claim} \label{Claim: half-tree in T-min}
If $T_H$ contains a {\em half-tree}, i.e. a connected component $Z$ of $T\smallsetminus \{e\}$ for some edge $e\in E(T)$, then $T_H=T$.
If $H$ is moreover finitely generated, then $H\bs T$ is finite.
\end{claim}
 \begin{proof}[Proof of Claim~\ref{Claim: half-tree in T-min}]
 \renewcommand{\qedsymbol}{$\diamond$}
Let $f$ be an edge in $Z$ and $h\in H$ a hyperbolic element whose axis $\mathrm{Axis}(h)$ contains $f$.
Then, up to replacing $h$ by $h^{-1}$, the image $h \cdot e$ belongs to $Z$. 
Let $g'\in H$ be a hyperbolic element whose axis contains $h \cdot e$.
Then the axis of the conjugate $g=h^{-1} g' h \in H$ contains $e$. Thus $e\in T_H$ and $T_H$ contains all the half-trees $g^k(Z)$ for $k\in \Z$. 
It follows that $T=\cup_{k} \: g^k(Z)=T_H$.
When $H$ is finitely generated, then (Remark~\ref{rem: H fg minimal -> T/H compact}) there is a finite connected subgraph $W\subseteq H\bs T$ whose associated graph of groups has fundamental group equal to $H$. If $H\bs T$ is infinite, then the pull-back of $W$ in $T$ is a proper $H$-invariant subtree of $T$.
This proves Claim~\ref{Claim: half-tree in T-min}.
\end{proof}

\begin{claim} \label{Claim: half-trees}
Assume that $G\acting T$ is a minimal and irreducible action on a tree and
 that $T_i := T_{\Lambda_i}\subsetneq T$ are proper $\Lambda_i$-invariant subtrees for finitely generated subgroups $\Lambda_i\leq G$ for $i \in \{1, \dots, 2r\}$. 
Then, there is an edge $e\in E(T)$ such that 
$\cup_{i=1}^{2r} T_i$ is contained in one connected component of $T\smallsetminus\{e\}$ and such that 
 $\Stab_{\Lambda_i}(e)=M$ for all $i$.
\end{claim}

\begin{proof}[Proof of Claim \ref{Claim: half-trees}]
\renewcommand{\qedsymbol}{$\diamond$}
If $U,V$ are subtrees of $T$ that do not contain any half-tree, then the convex hull of $U\cup V$ does not contain any half-tree. To see this,
pick an edge $e_0\in E(T)$ and let $K$ be one of its half-trees. Since $U$ does not contain $K$ nor any half-tree therein, there is an edge $e_1\in E(K)$ such that $U$ is in the same half-tree of $T\smallsetminus \{e_1\}$ as $e_0$. Let $K_2$ be the other half-tree of $T\smallsetminus \{e_1\}$ and note that $K_2 \subseteq K$. Since no half-tree is contained in $V$, there exists an edge $e_2\in E(K_2)$ such that $V$ is in the same half-tree of $T\smallsetminus \{e_2\}$ as $e_1$ (and $e_0$). If $K_3 \subseteq T \smallsetminus \{e_2\}$ is the half-tree which does not contain $e_1$ then $K_3 \subseteq K$ and $K_3$ has empty intersection with the convex hull of $U \cup V$. Thus the half-tree $K$ is not entirely contained in the convex hull of $U\cup V$.

Since minimal subtrees $T_i$ of $\Lambda_i$ do not contain any half-tree by Claim~\ref{Claim: half-tree in T-min}, the same holds for their convex hull. It follows by induction that $\cup_{i=1}^{2r} T_i$ is contained in a half-tree of $T\smallsetminus\{f\}$ for some $f\in E(T)$. This remains true if we replace $f$ by any edge in the connected component of $T \smallsetminus \{f\}$ which does not contain $\cup_{i=1}^{2r} T_i$. 

By Claim~\ref{Claim: f1-f2 in Z}, one can assume that the segment $L=[f_1,f_2]$ is contained in the connected component of $T\smallsetminus \{f\}$ that does not contain $\cup_{i=1}^{2r} T_i$ and that $f_1$ separates $f_2$ from $f$, thus from each $T_i$.
By Claim~\ref{claim: Stab e-Lambda}, the edge $e=f_2$ satisfies $\Stab_{\Lambda_i}(e)= \Lambda_i\cap N_0 = M$ for each $i$ and this concludes the proof of Claim~\ref{Claim: half-trees}.
\end{proof}

We take an edge $e\in E(T)$ as in the Claim~\ref{Claim: half-trees} and let $Z$ be the connected component of $T\smallsetminus \{e\}$ that does not contain $\cup_{i=1}^{2r} T_i$. 
By Lemma~\ref{lem: existence of hyperbolic in Z}, there is a hyperbolic element  $g\in G$ whose axis is contained in $Z$. 

We now fix $i \in \{1, \dots, r\}$ and show that $\Lambda_i$ satisfies the assumptions of Lemma~\ref{main dendrological lemma} where the $\Lambda_i$-invariant proper subtree is $T_{\Lambda_i}= T_i$ and the subgroups $H_n$ are of the form $g^{q} \Lambda_{r+i}g^{-q}$. The sequences of subtrees $(S_n)_n$ and of edges $(e_n)_n$ will be defined shortly.

Since the action by conjugation of $G$ on $\Sub(N_0)$ has finite orbits, there is a
non-zero integer $k$ such that $g^k M g^{-k}=M$.
It follows that $g^{kn} \Lambda_{r+i} g^{-kn}\cap N_0=g^{kn} (\Lambda_{r+i} \cap N_0) g^{-kn}=g^{kn} M g^{-kn}=M$ 
for every integer $n$.

Let $e_n\coloneqq g^{kn}\cdot e$. Then the distance of $e_n$ from $T_{\Lambda_i}$ is greater than $n +1$ times the translation length of $g$.
The subtree $S_n\coloneqq g^{kn}\cdot T_{r+i}$ is contained in $Z$  and $e_n$ separates $S_n$ from $T_i$.
The subtree $g^{kn}\cdot T_{r+i}$ is an invariant subtree for the subgroup $H_n\coloneqq g^{kn} \Lambda_{r+i} g^{-kn}$.
We have defined $(H_n)$, $(S_n)$ and $(e_n)$ and checked Assumptions~\ref{it: Hn is a subgroup},  \ref{it: Sn is Hn-inv}, \ref{it assumpt dist >n} and \ref{it en separates} of Lemma~\ref{main dendrological lemma}.
The assumption we made that $\Lambda_{r+i}$ is not a subgroup of $N_0$ implies that $\Stab_{H_n}(e_n)\not= H_n$, thus checking Assumption~\ref{it: assumpt stab e-n H-n not = H-n}. 
If an element $\lambda\in \Lambda_i$ sends some $e_n$ to an edge $\lambda\cdot e_n$ belonging to the path from some other $e_k$ to $T_i$ then $\lambda$ must fix the projection $p$ of $e_k$ in $T_i$ and $\lambda\cdot e_n$ belongs to $Z$. It follows that $\lambda$ must fix the edge $e$ and thus $\lambda\in \Stab_{\Lambda_i}(e)=M\leq  N_0$. So $\Lambda_i\cdot e_n$ does not separate $e_k$ from the tree $T_i$.
Therefore, $(e_n)$ verifies condition \ref{it en does not separates T-Lambda from the ek} of Lemma~\ref{main dendrological lemma}.
Concerning the stabilisers:
$\Stab_{\Lambda_i}(e_n)=g^{kn} \Stab_{\Lambda_i}(e) g^{-kn}=g^{kn} M g^{-kn}=M$ while $\Stab_{H_n}(e_n)=\Stab_{g^{kn} \Lambda_{r+i} g^{-kn}}(g^{kn} \cdot e)=\Stab_{\Lambda_{r+i}}(e)=M$.
The Assumption~\ref{it: assumpt stab e_n are equal} is satisfied.

\medskip
Thus, by Lemma~\ref{main dendrological lemma}~\eqref{it: I} and \eqref{it: II}, $\Delta^{i}_n\coloneqq \langle \Lambda_i, g^{kn} \Lambda_{r+i} \, g^{-kn}\rangle= \Lambda_i*_M g^{kn} \Lambda_{r+i}\,  g^{-kn}$ and we have the non-trivial convergence:
\[\Delta^{i}_n= \Lambda_i*_{M} g^{kn} \Lambda_{r+i} \, g^{-kn}\underset{n\to \infty}{\longrightarrow} \Lambda_i.\]
Since $\{\Lambda\leq G\colon \Lambda\cap N_0=M\}$ is open, we have $\Delta^{i}_n\cap N_0=\Lambda_i\cap N=M$ for all large enough $n$.
By \eqref{it: V} of Lemma~\ref{main dendrological lemma}, the quotient $\Delta^{i}_n\bs T$ is infinite, i.e. $\Delta^{i}_n\in \Sub_{\vert \bullet\bs T\vert \infty}(G)$.
Thus $\Delta^{i}_n\in {\mathcal F}^{T}_C$.
It follows that for $n$ larger than some $n_i$, the subgroup $\Delta^{i}_n\in V_i$.

It remains to show that we have the non-trivial convergence:
\[g^{-kn}\Delta^{i}_n\, g^{kn}= g^{-kn}\Lambda_i \, g^{kn}*_{M}  \Lambda_{r+i} \underset{n\to \infty}{\longrightarrow} 
\Lambda_{r+i}.\]
It is enough to observe that $g^{-kn}\Delta^{i}_n\, g^{kn}= g^{-kn}\Lambda_i \, g^{kn}*_{M}  \Lambda_{r+i}$ is obtained by the exact symmetric construction by exchanging the roles of $\Lambda_i, \Lambda_{r+i}$ and replacing $g$ by $g^{-1}$. So the symmetric application of Lemma~\ref{main dendrological lemma} \eqref{it: II} shows that 
$\lim_{n\to \infty}g^{-kn}\Delta^{i}_n\, g^{kn} = \Lambda_{r+i}$.
It follows as above that for all $n$ larger than some $n_{r+i}$, the subgroup $g^{-kn}\Delta^{i}_n\, g^{kn}\in V_{r+i}$.

Hence, for any $n\geq \max\{n_j : j \in \{1, \dots, 2r\}\}$ and for each $i \in \{1, \dots, r\}$, we have:
\begin{center}$\Delta^{i}_n\in V_i$ and $g^{-kn}\Delta^{i}_n\, g^{kn}\in V_{r+i}$, \ \ 
and so \ \ 
$g^{-kn}V_i\, g^{kn}\cap V_{r+i}\neq \emptyset$.
\end{center}
This completes the proof of Theorem~\ref{th: two edge-stab with finite intersection - top. dyn.}.
\end{proof}

\subsection{Applications}
\label{Sect: Applications of dendrolog. results}

Recall that the $G$-action $G\acting T$ on a tree is \defin{acylindrical} if there is $R>0$ such that the stabilisers of any path of length at least $R$ is trivial. 

 Examples of $3$-acylindrical actions include the actions on the Bass-Serre tree of non-trivial amalgamated free products $A \ast_C B$ where $C$ is \defin{malnormal} in $A$ (i.e. for all $g \in A \smallsetminus C$, $C \cap g C g^{-1}$ is trivial).

\begin{corollary}[Acylindrical actions]
\label{cor:acylindrical}
Let $G\acting T$ be a minimal irreducible and acylindrical action of a countable group on a tree.
If $\Lambda\leq G$ is a subgroup such that the quotient graph $\Lambda\bs T$ is infinite, then 
$\Lambda\in \PK(G)$.
Moreover, $G\acting \overline{\Sub_{\vert \bullet\bs T\vert \infty}(G)}$ is highly topologically transitive.
\end{corollary}
\begin{proof} The $G$-action on $T$ satisfies the assumptions of Theorem~\ref{th: two edge-stab with finite intersection} and Theorem~\ref{th: two edge-stab with finite intersection - top. dyn.}: the intersection of the stabilisers of any pair of edges at distance greater than the acylindricity constant $R$ is trivial.
\end{proof}
\begin{remark}\label{Rem: acylindrical act. examples}
There exist acylindrical actions $G\acting T$ as in Corollary~\ref{cor:acylindrical} where $G$ admits a finitely generated subgroup $\Lambda \leq G$ with infinite index but where the quotient $\Lambda \bs T$ is finite.
Moreover, there are examples where $\Lambda$ belongs to $\PK(G)$ and others where it does not.

 For example, let $C$ be any finitely generated infinite torsion-free group and let $A, B$ be finitely generated non-trivial groups such that $C$ is malnormal in $G_1 = A \rtimes C$ and $G_2 = B \rtimes C$ respectively. Such groups exist by D. Osin's Appendix in \cite[Theorem 15]{HW14}. 
 
 Then let $G = G_1 *_C G_2$ where the injections $C \hookrightarrow G_i$ are given by the natural inclusion maps and consider the action of $G$ on the Bass-Serre tree $T$ associated to this decomposition. It follows from the malnormality of $C$ in $G_1$ and $G_2$ that the action $G \acting T$ is $2$-acylindrical. Moreover, the normal subgroup $\Lambda = \la A, B \ra \trianglelefteq G$ is finitely generated, has infinite index in $G$ since $G/\la A, B \ra\simeq C$, and acts transitively on the set of edges of $T$, so $\Lambda \bs T$ consists of a single edge.
 
 Moreover, by Proposition~\ref{prop: f.g. normal subgr and Sub(Q)}, the finitely generated normal subgroup $\Lambda=\la A,B\ra=A*B$ belongs to the perfect kernel $\PK(G)$ if and only if the trivial subgroup $\{\id\}$ belongs to the perfect kernel $\PK(C)$.
If we choose for instance $C = \Z$ then $\Lambda\not\in\PK(G)$ while if $C = \FF_2$, then  $\Lambda\in \PK(G)$.

\end{remark}

The results of the previous section imply that:
\\
{\em If $G$ is a finitely generated non-elementary free product\footnote{A free product  $G = A \ast B$ is \defin{non-elementary} if  $\vert A\vert, \vert B\vert\geq 2$ and at least one of $A$ and $B$ has cardinality at least $3$.} then 
$\PK(G)=\Sub_{[\infty]}(G)$ and the action of $G$ on its perfect kernel is highly topologically transitive.}
More generally,
\begin{corollary}[Infinitely many ends]
\label{cor:infinitely many ends}
If $G$ is a finitely generated group with infinitely many ends, then the perfect kernel of $G$ consists of its subgroups of infinite index:
\[\PK(G)=\Sub_{[\infty]}(G).\]
Moreover, $G$ has a greatest finite normal subgroup $N_0$
 and the set $\Conj(N_0)^G$ defines a finite clopen partition of $\PK(G)$ into $G$-invariant and topologically transitive perfect subsets:
\[\PK(G)=\bigsqcup_{C\in \Conj(N_0)^G} \mathcal{S}_C \text{, \ \ where \ } 
\mathcal{S}_C:=\Sub_{[\infty]}(G)\cap \left\{\Lambda\colon \Lambda\cap N_0\in C\right\}.\]
If $N_0=\{\id\}$ (resp. $\vert C\vert=1$)  then $G\acting \PK(G)$ (resp. $G\acting \mathcal{S}_C$) is highly topologically transitive.
\end{corollary}
\begin{proof}
By Stallings' theorem \cite{Sta68,Stallings-1971-inf-ends}, $G$ admits a minimal and irreducible action $G\acting T$ on a tree $T$ with finite edge stabilisers.
If $M$ is a finite normal subgroup of $G$, then it is contained in the kernel of the action $G\acting T$.
Conversely, this kernel is finite and normal: this is $N_0$.
If $\Lambda \in \Sub(G)$ has infinite index, then the image in $\Lambda\bs T$ of the $G$-orbit any edge $e$ of $T$ is in bijection with the double classes $\Lambda\bs G/\Stab_{G}(e)$, which is infinite since $\Stab_{G}(e)$ is finite. Thus $\Lambda\bs T$ is infinite and $\mathcal{S}_C$ coincides with ${\mathcal F}^{T}_C$ (defined in Equation~\ref{eq: def F-C}).
Since the edge stabilisers are finite, any two distinct edges $f_1,f_2$ of $T$ satisfy a fortiori that $\Stab_G(f_1)\cap \Stab_G(f_2)$ is finite.
We conclude by applying Theorem~\ref{th: two edge-stab with finite intersection} and Theorem~\ref{th: two edge-stab with finite intersection - top. dyn.}.
 \end{proof}
\begin{example}[$\mathrm{SL}(2,\Z)$]
\label{ex: SL(2,Z)}
This corollary applies for instance to $G=\mathrm{SL}(2,\Z)$ for which $N_0=\{\pm \id\}\simeq \Z/2\Z$ is the centre. The perfect kernel decomposes into two invariant pieces 
$$\PK(G)=\Sub_{[\infty]}(\mathrm{SL}(2,\Z))=
 \left\{\Lambda \colon  N_0\leq\Lambda \right\}
\sqcup \left\{\Lambda \colon  N_0\not\leq\Lambda\right\}$$
on which the action is  highly topologically transitive.
\end{example}

\begin{corollary}\label{cor: F infty act K(G) is top. r-transit} 
The action of the free group $\FF_\N$ on countably many generators on its perfect kernel $\PK(\FF_\N)=\Sub(\FF_\N)$ is highly topologically transitive. 
\end{corollary}
\begin{proof}
By Proposition~\ref{Prop: K of count. free variety}, $\PK(\FF_\N)=\Sub(\FF_\N)$.
Every non-empty open subset $V$ of $\Sub(\FF_\N)$ contains a finitely generated subgroup, thus for every $2r$-tuple $(V_1, V_2, \cdots, V_{2r})$ of non-empty open sets, there is a finite subset $J\Subset \N$ with $\vert J\vert \geq 2$ such that $V_i\cap \Sub_{[\infty]}(\FF_J)\not=\emptyset$ for all $i$.
By Corollary~\ref{cor:infinitely many ends}, there is $g\in \FF_J$ such that $g\cdot (V_i\cap \Sub_{[\infty]}(\FF_J))\cap (V_{r+i}\cap \Sub_{[\infty]}(\FF_J))\not=\emptyset$ for all $i$; in particular, $g\cdot V_i\cap V_{r+i}\not=\emptyset$.
\end{proof}

\medskip

\begin{example}[Weakly malnormal edge group]
\label{ex: amalg or HNN with non s-normal edge gp}
A subgroup $C \leq G$ is said to be \defin{weakly malnormal} if there exists $g\in G$ such that $g C g^{-1}\cap C$ is finite.
Let $G$ be a free product with amalgamation $G=A*_CB$  or a HNN-extension $G=A*_C$ with irreducible associated Bass-Serre tree $T$. Assume that $C$ is weakly malnormal in $G$. It follows that the action of $G$ on $T$ verifies the conditions of Theorem~\ref{th: two edge-stab with finite intersection} (and of Theorem~\ref{th: two edge-stab with finite intersection - top. dyn.}). Thus every subgroup $\Lambda$ such that $\Lambda\bs T$ is infinite belongs to $\PK(G)$.

For instance, if $A=\FF_p$ is a non-abelian free group, then any finitely generated subgroup $C$ with infinite-index is weakly malnormal  in $A$ and thus in $G=A*_C B$ for any $B$.
More generally, by \cite[Theorem 5.12]{2011-Peterson-Thom} (with $\beta_1^{(2)}(.)$ the first $\ell^2$-Betti numbers): {\em If $C\leq G$ is an infinite index subgroup with $\beta_1^{(2)}(C)<\infty$ and if $\beta_1^{(2)}(G)>0$, then $C$ is weakly malnormal .}

This leads to plenty of examples. Take $A\geq C$ and $B\geq C$ such that $[A:C]\geq 2$ and $[B:C]\geq 3$: this will ensure the irreducibility of $G\acting T$ by Proposition~\ref{prop: irreducibility vs indices}. Assume that $\beta_1^{(2)}(C)<\infty$ and $\beta_1^{(2)}(A)+\beta_1^{(2)}(B)-\beta_1^{(2)}(C)>0$: an application of the Mayer-Vietoris formula from Cheeger-Gromov \cite[p. 204, §4]{CG-86} gives (when $C$ is infinite) $\beta_1^{(2)}(A*_CB)\geq \beta_1^{(2)}(A)+\beta_1^{(2)}(B)-\beta_1^{(2)}(C)$. Thus $C$ is weakly malnormal  in $G=A*_CB$ and Theorems~\ref{th: two edge-stab with finite intersection} and \ref{th: two edge-stab with finite intersection - top. dyn.} apply.
\end{example}

\subsection{A rich class of groups with maximal perfect kernel}

\newcommand\CCC{\mathcal{C}}

In their study of the subgroups of direct products of limit groups, Bridson, Howie, Miller, and Short  \cite{BHMS-2009} introduced a class $\mathcal{C}$ of finitely presented groups defined in a hierarchical manner as the union of classes $\mathcal{C}_n$ and proved a series of properties of this class, in particular their Proposition 5.2.
A close look at their proof reveals that various properties that are of interest to us are verified by the larger class $\mathcal{Q}$ defined below, with essentially the same arguments.
Compare with \cite[Section 2.2]{BHMS-2009}: $\CCC_0$ is the class of non-elementary free products of finitely presented groups, and a group lies in $\CCC_n$ if and only if it is the fundamental group of a finite, acylindrical graph of finitely presented groups, where all of the edge groups are cyclic, and at least one of the vertex groups lies in $\CCC_{n-1}$.

\begin{definition}[The classes  $\mathcal{Q}$ and $\mathcal{Q}'$]
\label{def: class Q}
The class $\mathcal{Q}_0$ consists of finitely generated groups with infinitely many ends. 
A group lies in $\mathcal{Q}_n$ if it is the fundamental group of a finite graph of groups such that
\begin{itemize} 
\item every vertex group is finitely generated,
\item every edge group is finitely generated virtually abelian, \item at least one vertex group lies in $\mathcal{Q}_{n-1}$ and 
\item the associated Bass-Serre tree $T$ admits a pair $(f_1,f_2)$ of edges whose stabilisers have finite intersection.
\end{itemize}
The class $\mathcal{Q}$ is the union of the classes $\mathcal{Q}_n$.
The class $\mathcal{Q}'$ is the subclass of groups in $\mathcal{Q}$ without any non-trivial finite normal subgroup.
\label{def: class Q'}
\end{definition}

\medskip

By Stallings' Theorem a group in $\mathcal{Q}_0$ admits a minimal irreducible action on a tree with a single orbit of edges and with finite edge stabilisers.
We call any such $G$-tree a ``Stallings tree'' of $G$. 
A $G$-tree $T$ witnessing that $G\in \mathcal{Q}$ will be called ``the'' Bass-Serre tree of $G$, although it is not uniquely defined.

Observe that the class $\mathcal{Q}$ does not contain any infinite virtually abelian group as all its elements contain a non-abelian free subgroup.

\begin{example}
\label{ex: class Q}
The class $\mathcal{Q}'$ contains all non-abelian limit groups (\cite[Corollary 2.2]{BHMS-2009}). 
Another important class of groups contained in $\mathcal{Q}'$ is that of graphs of free groups with cyclic edges such that at least one vertex group is non-cyclic. 
\end{example}

The following is a corollary of Theorems~\ref{th: two edge-stab with finite intersection}
 and~\ref{th: two edge-stab with finite intersection - top. dyn.}.
\begin{corollary}\label{cor: groups in CCC, K(G) and HTT}
Let $G$ be any group in the class $\mathcal{Q}$.
Then $\PK(G)=\Sub_{[\infty]}(G)$.
Moreover, the conclusions of Theorem~\ref{th: two edge-stab with finite intersection - top. dyn.} apply. 
In particular, if $G\in \mathcal{Q}'$, then the action $G\acting \PK(G)$ highly topologically transitive.
\end{corollary}

We postpone the proof of this corollary until after the following two necessary propositions and 
after their proofs (which mimic that of \cite[Proposition 5.2]{BHMS-2009}).
Observe that the assumptions of the propositions both repeat those defining the class $\mathcal{Q}$ each with one condition missing (the third and the fourth respectively).

\begin{proposition}
\label{claim: C, Lamba double classes finite -> Lambda / T finite}
Let $G$ be the fundamental group of a finite graph of finitely generated groups,   
where all of the edge groups are finitely generated virtually abelian and the associated Bass-Serre tree $T$ admits a pair $(f_1,f_2)$ of edges whose stabilisers have finite intersection. Let $\Lambda, C\leq G$ be finitely generated subgroups where $C$ is virtually abelian.
If the set of double classes $\Lambda\bs G/C$ is finite, then the quotient graph $\Lambda\bs T$ is finite.
\end{proposition}
This proposition applies in particular to any $G\in \mathcal{Q}$ and its Bass-Serre tree.

\begin{proposition}
\label{claim: L/T finite -> L/G finite}
Let $G$ be the fundamental group of a finite graph of finitely generated groups,   
where all of the edge groups are virtually abelian.
Assume that the associated Bass-Serre tree $T$ admits a vertex $w$ satisfying $\Stab_G(w)\in \mathcal{Q}_{n}$, for $n\geq 0$.
Then the quotient graph $\Lambda\bs T$ is finite if and only if the index $[G:\Lambda]$ is finite.
\end{proposition}

\begin{proof}[Proof of Proposition~\ref{claim: C, Lamba double classes finite -> Lambda / T finite}]
We consider two cases.
\\
\textbf{Case 1. Suppose that $C$ fixes a vertex $v$ of $T$.}
Then, by the double-coset hypothesis, the $G$-orbit of $v$ consists of only
finitely many $\Lambda$-orbits $\Lambda \cdot v_i$. Since the action of $G$ on T is cocompact, every point in $T$ is at a bounded distance from $G\cdot v$, thus $\Lambda\bs T$ has finite diameter.
 The fundamental group of the topological graph underlying  $\Lambda\bs T$ is finitely generated since $\Lambda$ is finitely generated and maps onto it.
 By minimality of $G\acting T$, there are no valence 1 edges in $T$. Moreover, the valence of a vertex $z$ in the tree $T$ is the sum of the indices $[\Stab_\Lambda(z):\Stab_\Lambda(e)]$ over a subset of the edges $e$ incident to $z$ containing a representative from each $\Lambda$-orbit.
 It follows that if $z$ is a vertex with $\Stab_\Lambda(z)=\Stab_\Lambda(e)$ for some incident edge $e$ then the image of $z$ in $\Lambda\bs T$ does not have valence 1.
 By finite generation of $\Lambda$, only finitely many valence 1 vertices of $\Lambda\bs T$ can have a vertex group larger than the group of their incident edge.
It follows that $\Lambda\bs T$ is finite. 

\noindent
\textbf{Case 2. Suppose that $C$ does not fix a vertex. }  
Since $C$ is finitely generated, it admits a hyperbolic element $c\in C$ \cite[Corollary 3, p. 65]{Serre-Trees-1980}. 
Let $A$ be the axis of $c$.
Let $C'\leq C$ be a finite index abelian subgroup. For some integer $j$, we have $c':=c^j\in C'$, and $A$ is also the axis of $c'$.
Let $a\in C$. The image $a\cdot A$ is the axis of the hyperbolic element $a c' a^{-1}$. 
In case $a\in C'$ then $a c' a^{-1}=c'$, so the line $A$ is $C'$-invariant and $A$ is the axis of all hyperbolic elements of $C'$.
If $a\in C$, a certain power $k$ satisfies $(a c a^{-1})^k\in C'$, thus the axis of $a c a^{-1}$ is also $A$.
It follows that $A$ is $C$-invariant and is the minimal subtree of $C$.

Let $e\in E(T)$ be an edge that belongs to $A$. 
Contract to one point each connected component of $T\smallsetminus G\cdot e$ in order to produce the $G$-equivariant factor map $P: T\to S$ to the edge-transitive $G$-tree $S$. 
Up to enlarging the segment $L\coloneqq [f_1,f_2]$ of $T$ so as to start and finish with edges in $G\cdot e$, the image $P(L)=[P(f_1),P(f_2)]$ satisfies $\vert \Stab_G(P(L))\vert <\infty$.
Let $\bar e :=P(e)$. The minimal tree of $C$ in $S$ is the image $\bar A:=P(A)=C\cdot \bar e $ of the axis $A$. 
The two maps $(G\to E(S) \: ; \:  g\mapsto g\cdot \bar e )$ and $(\Lambda\bs G\to E(\Lambda\bs S) \: ; \: \Lambda g\mapsto \Lambda (g\cdot \bar e ))$ are onto and the finiteness of the double cosets $\Lambda\bs G/C$ thus implies the following:

{\em 
 The quotient graph $\Lambda\bs S$ is covered by finitely many images $\Lambda g_i C\cdot \bar e $, i.e. by finitely many images $Q_{\Lambda}(g_i (\bar A ))$ of lines under the quotient map $Q_{\Lambda}:S\to \Lambda\bs S$.
} 

Let $S_\Lambda\subseteq S$ be a minimal $\Lambda$-invariant subtree.
Assume for a contradiction that $ S_\Lambda\not= S$. Let $\varepsilon$ be an edge of $S\smallsetminus S_\Lambda $ and $Z_1$ the half-tree of $S\smallsetminus\{\varepsilon\}$ that does not contain $S_\Lambda$. 
By Lemma~\ref{lem: existence of hyperbolic in Z}, there is $\gamma_2\in G$ such that $\gamma_2([P(f_1),P(f_2)])\subseteq Z_1$ and $\gamma_2(P(f_1))$ separates $\gamma_2(P(f_2))$ from $\varepsilon$.
We show that the map $Q_{\Lambda}$ is boundedly finite-to-one in restriction to the half-tree $Z_2$ of $S\smallsetminus\{\gamma_2(P(f_2))\}$ that does not contain $S_\Lambda$.
Indeed an element $\lambda\in \Lambda$ sending a vertex $u$ of $Z_2$ inside $Z_2$ must fix the projection $v$ of $u$ on $S_\Lambda$ and thus the segment $[v,\gamma_2(P(f_2))]$. But this segment contains the segment $\gamma_2(P(L))$. Thus $\lambda$ belongs to the $\gamma_2$-conjugate $\gamma_2\Stab_G(P(L))\gamma_2^{-1}$ of the point-wise stabiliser of $P(L)$, which is finite. 
The map $Z_2\to Q_{\Lambda}(Z_2)=(\gamma_2\Stab_G(P(L))\gamma_2^{-1}\cap \Lambda)\bs Z_2$ has fibers of uniformly bounded size, thus $Q_{\Lambda}(Z_2)$ has infinitely many ends.
The graph $\Lambda\bs S$ contains $Q_{\Lambda}(Z_2)$, thus it has infinitely many ends, and it cannot be covered by finitely many images of lines, contrarily to the above Claim.
Thus $S_\Lambda=S$.

Since $\Lambda$ is finitely generated, $\Lambda\bs S$ is finite (Remark~\ref{rem: H fg minimal -> T/H compact}). 
It follows that $\Lambda\bs T$ is finite.
Proposition~\ref{claim: C, Lamba double classes finite -> Lambda / T finite} is proved.
\end{proof}

\begin{proof}[Proof of Proposition~\ref{claim: L/T finite -> L/G finite}]
We will argue by induction on $n$.

The initialisation of the induction is that groups in $\mathcal{Q}_0$ satisfy the conclusion of Proposition~\ref{claim: L/T finite -> L/G finite}: 
If $G\in \mathcal{Q}_0$, i.e. its  Stallings tree is an irreducible edge transitive action 
$G\acting T$ with a finite edge stabiliser $K$, then for any $\Lambda\leq G$, the quotient $\Lambda\bs T$ is finite if and only if $\Lambda\bs G/K$ is finite,  if and only if $\Lambda\bs G$ is finite.

Assume the conclusion of Proposition~\ref{claim: L/T finite -> L/G finite} holds for all groups satisfying the assumptions up to some $n-1$.
In particular, this conclusion holds for all groups in $\mathcal{Q}_{n}$. 

Let now $G$ be a group satisfying the assumptions for $n$ with $\Stab_\Lambda(w)\in \mathcal{Q}_{n}$. Let $\Lambda\leq G$ and assume $\Lambda\bs T$ is finite. 
Observe that $\Stab_\Lambda(w)$ is also finitely generated since 
(a) subgroups of finitely generated (virtually) abelian groups are finitely generated and (b) 
 a finitely generated graph of groups (such as $\Lambda$) with finitely generated edge groups must have finitely generated vertex groups by \cite[p.43, Propositions 29 and 35]{Cohen-book-1989}. 
 
Let $e$ be some edge incident at $w$. 
Then $\vert \Stab_\Lambda(w)\bs \Stab_G(w)/ \Stab_G(e)\vert $ is bounded above by the finite number of edges of  $\Lambda \bs T$ that are  incident at $\Lambda\cdot w \in \Lambda\bs T$.
Since $\Stab_G(w)\in \mathcal{Q}_{n}$, since $\Stab_\Lambda(w)$ and $\Stab_G(e)$ are finitely generated and $\Stab_G(e)$ is virtually abelian, 
then applying Proposition~\ref{claim: C, Lamba double classes finite -> Lambda / T finite}
to the action of $\Stab_G(w)$ on its own Bass-Serre tree $T'$ implies that $\Stab_\Lambda(w)\bs T'$ is finite.
Then, by the induction hypothesis (including the initialisation when $n=0$), $\Stab_\Lambda(w)$ has finite index in $\Stab_G(w)$.
And similarly, for each $g\in G$, the subgroup $\Stab_\Lambda(gw) $ has finite index in $\Stab_G(gw) $: Indeed $\Stab_G(gw) \simeq\Stab_G(w)\in \mathcal{Q}_{n}$ and $gw$ could have played the role of $w$.

Consider the action of $\Stab_G(w)$ by right multiplication on $\Lambda\bs G$: the orbits are the double cosets $\Lambda\bs G/\Stab_G(w)$ and hence are finite in number because they index a subset of the vertices of $\Lambda\bs T$ ; moreover, the stabilizer of $\Lambda g$ is $g^{-1}\Lambda g \cap \Stab_G(w)=g^{-1}(\Lambda \cap g \Stab_G(w)g^{-1})g =g^{-1} \Stab_\Lambda(gw) g$, which we have just seen has finite index in $g^{-1}\Stab_G(gw)g=\Stab_G(w)$. Thus $\Lambda\bs G$ is finite. 
The converse is clear: if $\Lambda\bs G$ is finite, the quotient $\Lambda\bs T$ is finite since $G\bs T$ is finite. Proposition~\ref{claim: L/T finite -> L/G finite} is proved.
\end{proof}

\begin{proof}[Proof of Corollary~\ref{cor: groups in CCC, K(G) and HTT}]
An element $G\in \mathcal{Q}$ is equipped with its defining Bass-Serre $G$-tree $T$. %
Observe that the action $G\acting T$ can be assumed to be minimal and irreducible: take the minimal $j$ such that $G\in \mathcal{Q}_j$ and note that
the stabiliser of $w$ is not virtually abelian, thus the stabiliser of any edge incident to $w$ has infinite index in $\Stab_G(w)$.
Irreducibility of the action then follows Proposition~\ref{prop: irreducibility vs indices}.
By Theorem~\ref{th: two edge-stab with finite intersection} $\Sub_{\vert \bullet\bs T\vert \infty}(G)\subseteq \PK(G)\subseteq \Sub_{[\infty]}(G)$.
It is enough to check that 
$\overline{\Sub_{\vert \bullet\bs T\vert \infty}(G)}\supseteq  \Sub_{[\infty]}(G)$.

If $\Lambda\in  \Sub_{[\infty]}(G)$ is finitely generated then $\Lambda \bs T$ is infinite: if $n=0$ this is because the action of $G$ on the edges of its Stallings tree is transitive with finite stabilisers, and otherwise this follows from Proposition~\ref{claim: L/T finite -> L/G finite}. Thus $\Lambda\in \Sub_{\vert \bullet\bs T\vert \infty}(G)$.
If $\Lambda\in  \Sub_{[\infty]}(G)$ is not finitely generated, it is a limit of proper finitely generated subgroups $\Lambda_n \leq \Lambda$, which are thus also of infinite index, so that $\Lambda_n\in \Sub_{\vert \bullet\bs T\vert \infty}(G)$ and $\Lambda\in \overline{\Sub_{\vert \bullet\bs T\vert \infty}(G)}$.
Then Theorem~\ref{th: two edge-stab with finite intersection - top. dyn.} applies concerning the dynamics.
\end{proof}

\section{Hyperbolic groups}

\subsection{Generalities on the space of subgroups of hyperbolic groups}

Recall from  \cite[5.3]{Gromov87} that a subset $Y$ of a geodesic metric space $X$ is \defin{quasiconvex} if its convex hull is contained within a bounded neighbourhood of $Y$. 

Let $G$ be a hyperbolic (aka word-hyperbolic) group. 
A subgroup $H\leq G$ is \defin{quasiconvex} if it is a quasiconvex subset of the Cayley graph $X$ of $G$ associated with a  finite generating set of $G$. Note that this does not depend on the choice of Cayley graph as quasiconvexity in hyperbolic spaces is invariant under quasi-isometry (this is not true for metric spaces in general). In fact $H \leq G$ is quasiconvex if and only if $H$ admits a finite generating set $T$ and the inclusion $H \hookrightarrow X$ is a quasi-isometric embedding with respect to the word metric defined by $T$.

Given a hyperbolic group $G$, let \[\Sub_{\mathrm{QC}[\infty]}(G) \subseteq \Sub_{[\infty]}(G)\]  denote the set of quasiconvex subgroups 
of $G$ with infinite index. Observe that $\Sub_{\mathrm{QC}[\infty]}(G) $ is invariant under the $G$-action by conjugation. If $G$ is non-elementary, we prove that this subset is in fact a $G$-invariant subspace of the perfect kernel of $G$.

\begin{theorem} \label{thm: kernel QC hyp}
Let $G$ be a non-elementary hyperbolic group. Then 
\[ \overline{\Sub_{\mathrm{QC}[\infty]}(G)} \subseteq \PK(G).\]
\end{theorem}
Recall that a non-elementary hyperbolic group $G$ has a finite radical $N \unlhd G$ called the \defin{finite radical} of $G$ and, by Proposition \ref{prop:finite normal subgr clopen part}, the set $\Conj(N)^{G}$ of $G$-conjugacy classes of subgroups of $N$ defines a finite $G$-invariant clopen partition: 
$\Sub(G)=\bigsqcup_{C\in \Conj(N)^G} \{H \leq G \colon H\cap N\in C\}$.
In analogy with Theorem \ref{th: two edge-stab with finite intersection - top. dyn.}, we denote for each $C \in \Conj(N)^G$ the subspace
\begin{equation}\label{eq: def of F-C}
 \mathcal{F}_C := \overline{\Sub_{\mathrm{QC}[\infty]}(G)} \cap \{H \leq G \colon H \cap N \in C \}
\end{equation}
(which is perfect by Theorem~\ref{thm: kernel QC hyp}) so that the clopen partition of $\Sub(G)$ induces a $G$-invariant partition of $\overline{\Sub_{\mathrm{QC}[\infty]}(G)}$:
\[\overline{\Sub_{\mathrm{QC}[\infty]}(G)} = \bigsqcup_{C \in \Conj(N)^G} \mathcal{F}_C\]
whose dynamics satisfy:
\begin{theorem} \label{thm: QC hyp dynamique}
Let $G$ be a non-elementary hyperbolic group and fix $C \in \Conj(N)^G$. Then
\begin{itemize}
\item[(i)] the action of $G$ on $\mathcal{F}_C$ is topologically transitive, and
\item[(ii)] if $|C|=1$ then the action $G \acting \mathcal{F}_C$ is highly topologically transitive. 
\end{itemize}
\noindent In particular, if $N = \{\id\}$ then the action $G\acting \overline{\Sub_{\mathrm{QC}[\infty]}(G)}$ is highly topologically transitive.
\end{theorem}

Note that the assumption that $|C| = 1$ is a necessary condition for a topologically $r$-transitive action on $\mathcal{F}_C$ whenever $r \geq 2$ (see Remark \ref{rem: obstruction to r-top transitive}). 

We delay the proofs of Theorems \ref{thm: kernel QC hyp} and \ref{thm: QC hyp dynamique} until the next section. The following is an immediate consequence. 

A hyperbolic group is \defin{locally quasiconvex} if every finitely generated subgroup of $G$ is quasiconvex.

\begin{corollary} \label{cor: locally quasiconvex}
If a group $G$ is non-elementary hyperbolic and locally quasiconvex, then
\[ \PK(G) = \Sub_{[\infty]}(G)=\overline{\Sub_{\mathrm{QC}[\infty]}(G)}.\]
If, in addition, the finite radical of $G$ is trivial then the action of $G$ on $\PK(G)$ is highly topologically transitive.
\end{corollary}

\begin{example}
\label{ex: hyp limit gps are loc qc}
Hyperbolic limit groups are locally quasiconvex by \cite[Proposition~4.6]{Dahmani03} and torsion free. Therefore, the above corollary provides another proof that the perfect kernel of a hyperbolic limit group $G\not=\Z$ is the space $\Sub_{[\infty]}(G)$ and that the $G$-action on $\PK(G)$ is highly topologically transitive.
\end{example}

\begin{example}
\label{ex: hyperbolic gps w. a small hierarchy}
If $G$ is non-elementary hyperbolic and admits a small hierarchy then by \cite[Theorem~3.6]{Brigdely-Wise13} it is locally quasiconvex and Corollary \ref{cor: locally quasiconvex} applies.

For instance, if $G$ is not virtually cyclic and splits non-trivially as a graph of groups where each vertex group is a free group, each edge group is cyclic and $G$ does not contain any subgroup which is a Baumslag-Solitar group then $G$ is non-elementary hyperbolic
 by \cite{Bestvina-Feighn92} and locally quasiconvex by the above. 
Compare with the drastically different behaviour when $G$ is a Baumslag-Solitar group \cite{CGLMS-1-arxiv}.
\end{example}

\begin{example}
\label{ex: one-relator group}
Let $S$ be a finite set with $\vert S\vert \geq 2$ and consider the one-relator group $G = \la S \: | \: w^m \ra$ where $w$ is a freely and cyclically reduced word in $S \cup S^{-1}$ and $m \geq |w|_S$. Such one-relator groups with torsion are non-elementary hyperbolic and have no non-trivial normal subgroups \cite[Theorems 3 and  2]{Newman68}, and they are locally quasiconvex by \cite[Theorem~1.2]{Hruska-Wise01} so, by Corollary~\ref{cor: locally quasiconvex}, $\PK(G) = \Sub_{[\infty]}(G)$ and the action $G \acting \PK(G)$ is highly topologically transitive. 
\end{example}

\begin{example}
\label{ex: Coxeter groups}
Let $G$ be the Coxeter group $\la s_1, \dots, s_n | s_i^2, (s_i s_j)^{m_{i,j}} \; \forall \; i < j \ra$ 
where $\infty\geq m_{i,j}> n\geq 3$.
Then $G$ is hyperbolic \cite[Theorem 17.1]{Moussong88} and  locally quasiconvex \cite[Theorem~12.2]{McCammond-Wise08}, so $\PK(G) = \Sub_{[\infty]}(G)$. Moreover, it follows from \cite{Paris-Varghese23} that $G$ has no non-trivial finite normal subgroups so the action $G \acting \PK(G)$ is highly topologically transitive.
\end{example}

\begin{remark}
\label{rem: rel hyp. gps}
It is likely that the results of this section can be extended to include relatively hyperbolic groups with suitable parabolic subgroups, for instance finitely generated abelian groups. See Question~\ref{question: generalisations}.
\end{remark}

\subsection{Proof of Theorems \ref{thm: kernel QC hyp} and \ref{thm: QC hyp dynamique}}
The proofs of both the main results of this section depend on the fact that, in a non-elementary hyperbolic group $G$, 
any infinite index quasiconvex subgroup $H$ is a factor in a splitting  of a quasiconvex subgroup $L$ of $G$ over a finite subgroup (see Proposition~\ref{prop: combination result}). 
Moreover, there is a great deal of freedom in the choice of $L$. In the torsion-free case, this result was first observed by Gromov in \cite[5.3]{Gromov87} and was developed in a detailed and effective manner by Arzhantseva in \cite{Arzhantseva01}. Other combination theorems allowing torsion were proved by Gitik in \cite{Gitik99} and were generalised to relatively quasiconvex subgroups of relatively hyperbolic groups by Martínez-Pedroza in \cite{Martinez-Pedroza09}. In fact, even when restricted to hyperbolic groups, the results of Martínez-Pedroza improve upon Gitik’s results in some respects (see for instance Theorems~\ref{thm: MP1} and \ref{thm: MP2} below). These 
improvements are needed for various aspects of our proofs, and they allow us to obtain Proposition~\ref{prop: combination result} for hyperbolic groups with torsion. 

The Gromov boundary $\partial G$ of a non-elementary hyperbolic group $G$ is a perfect compact metrisable space; 
it is equipped with a $G$-action that extends the left multiplication action of $G$ on itself.
Given a subgroup $H$, let $\Lambda(H) \subseteq \partial G$ denote the limit set of $H$. 
If $H$ is virtually infinite cyclic then $\Lambda(H)$ consists of precisely two points (the limit points of any infinite order element of $H$). Let $g \in G$ be an infinite order element and $P = \Stab(\Lambda(\la g \ra))$. Then $\la g \ra$ is a finite index subgroup of $P$ \cite[Theorem 30] {dlHarpe-Ghys90} which implies that $P$ is quasiconvex \cite[8.1D]{Gromov87}. Moreover if $a\in G\smallsetminus P$ then $P \cap a Pa^{-1} = \Stab(\Lambda(\la g \ra) \cap \Stab(\Lambda( \la ag \ra))$ where $\Lambda(\la g \ra) \neq \Lambda(\la ag \ra)$. It follows that there is a finite index subgroup $Q$ of $P \cap aPa^{-1}$ which fixes three points of $\partial G$; this implies that $Q$ fixes a bounded subset of $G$ and is therefore finite. 
Thus $P$ is \defin{almost malnormal} (i.e. $P \cap P^a$ is finite for all $a \in G \smallsetminus P$;
here and in what follows, $P^a$ denotes the conjugate $aPa^{-1}$).
 By \cite[Theorem~7.11]{Bowditch12} this implies that $G$ is hyperbolic relative to $P$. Since subgroups of conjugates of $P$ are virtually cyclic and therefore quasiconvex it follows from \cite[Theorem 1.1]{Martinez-Pedroza12} that quasiconvexity in $G$ is equivalent to quasiconvexity relative to $P$. In particular, the following theorems are special cases of the results of \cite{Martinez-Pedroza09}.

\begin{theorem}[Martínez-Pedroza, {\cite[Theorem 1.1]{Martinez-Pedroza09}}] \label{thm: MP1}
Let $G$ be a hyperbolic group and fix a finite generating set $S$ for $G$. Let $g \in G$ be an infinite order element and $P = \Stab(\Lambda(\la g \ra))$. Let $H \leq G$ be a quasiconvex subgroup. Then there is a constant $C_1 \geq 0$ depending only on $H$, $P$ and $S$ such that the following holds. If $Q \leq P$ is a subgroup such that 
\begin{enumerate}
\item $H \cap P \leq Q$ and
\item $|k|_S \geq C_1$ for all $k \in Q \smallsetminus H$
\end{enumerate}
then $\la H, Q \ra$ splits as the amalgamated free product $\la H, Q \ra = H \ast_{H \cap Q} Q$ and $\la H, Q \ra$ is quasiconvex in $G$.
\end{theorem}

\begin{theorem}[Martínez-Pedroza, {\cite[Theorem 1.2]{Martinez-Pedroza09}}] \label{thm: MP2}
Let $G$ be a hyperbolic group and fix a finite generating set $S$ for $G$. Let $g \in G$ be an infinite order element and $P = \Stab(\Lambda(\la g \ra))$. Let $H_1, H_2 \leq G$ be quasiconvex subgroups such that $M := H_1 \cap P = H_2 \cap P$. Then there exists a constant $C_2 \geq 0$ depending only on $H_1, H_2, P$ and $S$ such that the following holds. If $p \in P$ is such that 
\begin{enumerate}
\item $pMp^{-1} = M$
\item $|k|_S \geq C_2$ for all $k \in M p M$
\end{enumerate}
then $\la H_1, p H_2 p^{-1} \ra$ splits as the amalgamated free product  $\la H_1, p H_2 p^{-1} \ra= H_1 \ast_M pH_2p^{-1} $ and $\la H_1, p H_2 p^{-1} \ra$ is quasiconvex in $G$.
\end{theorem}

We are now in a position to prove the combination result we need.
In order to produce the relevant subgroup $P\leq G$, we adapt a standard argument to produce a suitable hyperbolic element of $G$ by taking advantage of the ``holes'' in the boundary of $G$ left by the subgroups $H_i$.

\begin{proposition} \label{prop: combination result}
Let $G$ be a non-elementary hyperbolic group and $H_1, \dots, H_r \leq G$ be quasiconvex subgroups with infinite index. Let $N \unlhd G$ be the finite radical of $G$ and suppose that $M := H_1 \cap N = H_i \cap N$ for each $i$. Then the following hold.

\begin{enumerate}
\item For each $i$, the subset $\Lambda(H_i) \subseteq \partial G$ is closed with empty interior. Thus the complement $\partial G \smallsetminus \cup_{i=1}^r \Lambda(H_i)$ is a dense open set.
\item For any non-empty open set $Y \subseteq \partial G$, there exists an infinite order element $p \in G$ whose limit points are contained in $Y$ and such that, if we denote $P = \Stab(\Lambda(\la p \ra))$, then $P \cap H_i = M$ for each $i$.
\item If $u \in G$ is a sufficiently high power of $p$ then, for each $n \in \mathbb{N}$, 
the subgroup $Q_n := \la u^n, M \ra$ generated by $u^n$ and $M$ splits as the direct product  $Q_n=\la u^n \ra \times M$, and for each $i \in \{1, \dots, r\}$,
the subgroup $\la H_i, Q_n \ra $ splits as the amalgamated free product 
\[\la H_i, Q_n \ra = H_i \ast_M Q_n\]
and is quasiconvex and has infinite index in $G$.
\end{enumerate}
\end{proposition}
\begin{proof}
Recall that the finite radical $N \unlhd G$ is the kernel of the action of $G$ on $\partial G$ (see e.g. \cite[Lemma 11.130]{Drutu-Kapovich2018}). Let $H \leq G$ be an infinite index quasiconvex subgroup. Then the limit set $\Lambda(H)$ is a (possibly empty, if $H$ is finite) closed proper subspace of $\partial G$. Suppose for a contradiction that there is an open subset $U\subseteq \partial G$ such that $U\subseteq \Lambda (H)$. Then by \cite[8.2G]{Gromov87}, there is an infinite order element $h\in H$ such that $\Lambda (\la h\ra)\subseteq U$. Since $h$ acts on $\partial G$ with north-south dynamics (\cite[8.1G]{Gromov87}), a certain power $h^n$ satisfies $h^n U\cup U=\partial G$. But $h^n U\cup U\subseteq h^n \Lambda(H)\cup \Lambda(H)=\Lambda(H)$.
Thus $\Lambda(H)=\partial G$ which contradicts the fact that $H$ has infinite index in $G$. Thus $\Lambda(H)$ has empty interior. The same holds for a finite union of closed sets with empty interior.
This proves part 1 of the lemma.

To prove part 2, we first show that, 
{\em for any $g \in G \smallsetminus N$, the fixed point set $\Fix(g) \coloneqq \{x \in \partial G : gx = x\}$ is closed with empty interior.} 
Fix an element $g \in G$. First, $\Fix(g)$ is closed.
Suppose that there is a non-empty open set $W \subseteq \Fix(g)$ and let $L \leq G$ 
be the pointwise stabiliser of $W$. Since $\partial G$ is perfect, $W$ contains more than two points so every element of $L$ has finite order. By \cite[Corollaire 36, Chapitre 8]{dlHarpe-Ghys90} this implies that $L$ is finite. Let $h \in G$ be an infinite order element whose limit points $h^+, h^-$ are in $W$ and let $W^+, W^- \subseteq W$ be open neighbourhoods of $h^+, h^-$ respectively such that $W^+ \cap W^- = \emptyset$. Let $L^-$ be the pointwise stabiliser of $W^-$ and note that, as above, $L^-$ is finite. Moreover, for sufficiently large $n \in \mathbb{N}$, $h^n W^- \cup W^+ = \partial G$ and in particular $W^- \subseteq h^n W^-$, so $h^n L^- h^{-n} \leq L^-$. Since $L^-$ is finite, this implies that $h^n L^- h^{-n} = L^-$.
But then $g$ is in the pointwise stabiliser of $h^n W^-$ and of $W^+$ so $g \in N$.

Recall that any subgroup $H \leq G$ acts properly discontinuously on $\partial G \smallsetminus \Lambda(H)$ \cite{Coornaert-1989}. 
Let $Y \subseteq \partial G$ be non-empty and open and let $x_0 \in Y' := Y \cap (\partial G \smallsetminus \cup_{i=1}^r \Lambda(H_i))$. Then for each $i$ there is an open neighbourhood $x_0 \in U_i \subseteq Y'$ such that $|\{h \in H_i : hU_i \cap U_i \neq \emptyset\}| < \infty$.
Let $U \coloneqq \cap_{i=1}^r U_i$ and note that $U$ is non-empty and open. Let $A = \{h \in (\cup_{i=1}^r H_i) \smallsetminus M : hU \cap U \neq \emptyset\}$. Since, for each $h \in A$, the fixed point set $\Fix(h)$ is closed with empty interior, the complement $V := U \smallsetminus (\cup_{h \in A} \Fix(h))$ is non-empty and open. Now enumerate $A = \{a_1, \dots, a_l\}$. If $a_1 V \cap V = \emptyset$, let $V_1 := V$. Otherwise there exists $x_1 \in V$ such that $a_1 x_1 \in V$ and, since $x_1 \notin \Fix(a_1)$, we have $x_1 \neq a_1 x_1$. It follows that there exists an open neighbourhood $V_1 \subseteq V$ of $x_1$ such that $a_1 V_1 \cap V_1 = \emptyset$. Iterate this process to find a sequence of non-empty open sets $V_l \subseteq \dots \subseteq V_1 \subseteq V$ such that for each $1 \leq k \leq l$ and each $1 \leq i \leq k$ we have $a_i V_k \cap V_k = \emptyset$. In particular $a V_l \cap V_l = \emptyset$ for all $a \in A$.

Let $p \in G$ be an infinite order element whose limit points are contained in $V_l \subseteq Y$ and let $P = \Stab(\Lambda(\la p \ra))$. Then $P \cap H_i = M$ for each $i$. This proves part 2 of the lemma.

 Let $u$ be a power of $p$ which commutes with $M$. For each $n \in \mathbb{N}$ let $Q_n := \la u^n, M \ra = \la u^n \ra \times M \leq P$. Then $Q_n \cap H_i = M$ for each $i$ and, since $P$ is quasi-isometrically embedded in $G$, for sufficiently large $n$ and $C_1 \geq 0$ as in Theorem~\ref{thm: MP1}, $|k|_S \geq C_1$ for all $k \in Q_n \smallsetminus M$. Therefore, by Theorem~\ref{thm: MP1}, for sufficiently large $n$, 
\begin{equation} \label{amalgamated free product}
\la H_i, Q_n \ra = H_i \ast_M Q_n \leq G 
\end{equation}
is quasiconvex. 
Moreover, up to replacing $u$ with $u^{n_0}$ for some sufficiently large $n_0$,
each $\la H_i, Q_n \ra$ is an infinite index quasiconvex subgroup of $G$ satisfying (\ref{amalgamated free product}).
Indeed, if $H_i \ast_M Q_{n_0}$ has finite index, then the splitting is non-trivial (if it were trivial, i.e. $H_i=M$, and $G$ would be elementary), thus $\la H_i, Q_{2n_0} \ra$ has infinite index in $\la H_i, Q_{n_0} \ra$ and therefore in $G$. This proves part 3 of the lemma.
\end{proof}

\begin{proof}[Proof of Theorem \ref{thm: kernel QC hyp}]
Let $H \leq G$ be an infinite index quasiconvex subgroup, $N \unlhd G$ be the finite radical of $G$ and $M = H \cap N$. By Proposition~\ref{prop: combination result}, there exists an infinite order element $u \in G$ such that, for each $n \in \mathbb{N}$, $Q_n := \la u^n, M \ra = \la u^n \ra \times M$ and 
\[ H_n := \la H, Q_n \ra = H \ast_M Q_n \in \Sub_{QC[\infty]}(G).\]
Let $T$ be a finite generating set for $H$ and $\lambda \geq 1$ be such that the natural inclusion of the first term of the sequence $H_1 \hookrightarrow G$ is a $\lambda$-quasi-isometric embedding with respect to the generating set $S$ for $G$ and the generating set $T \cup \{u\}$ for $H_1$. For any $n \in \mathbb{N}$, if $o \in H_n \smallsetminus H$ then 
the reduced spelling of $o$ must contain the subword $u^n$, and 
since $o\in H_1\smallsetminus H$ we have
\[ |o|_S \geq \frac{1}{\lambda} |o|_{T \cup \{u\}} - \lambda \geq \frac{n}{\lambda} - \lambda.\]
Hence, given a basic open neighbourhood $\mathcal{V}(\II,\OO)$ of $H$ in $\Sub(G)$, for all sufficiently large $n$, the subgroup $H_n$ contains $H$ but does not contain $o\in \OO$, thus $H_n \in \mathcal{V}(\II, \OO) \cap \Sub_{QC[\infty]}(G)$ so $H$ is not isolated in $\Sub_{\mathrm{QC}[\infty]}(G)$.
\end{proof}

\begin{proof}[Proof of Theorem~\ref{thm: QC hyp dynamique}]
Fix $C \in \Conj(N)^G$. To prove part (i) (i.e. that the action of $G$ on $\mathcal{F}_C$ is topologically transitive) we will show the following. Given two non-empty open subsets $U, V \subseteq \mathcal{F}_C$, there exists $g \in G$ such that $(g \cdot U) \cap V = gUg^{-1} \cap V \neq \emptyset$. As in the proof of Theorem~\ref{th: two edge-stab with finite intersection - top. dyn.}, we can assume, up to replacing $V$ with its image by some element of $G$, that there exist finitely generated subgroups $H \in U$ and $K \in V$ such that $M := H \cap N = K \cap N$. 

To prove part (ii), we need to show that, if $C = \{M\}$ (i.e. $M$ is normal in $G$) then, for all $r$ and for all non-empty open sets $U_1, \dots, U_r, V_1, \dots, V_r \subseteq \mathcal{F}_C$, there exists $g \in G$ such that $gU_i \cap V_i \neq \emptyset$ for each $i$. Since finitely generated subgroups are dense and, for all $H_i \in U_i$ and $K_i \in V_i$, we have $H_i \cap N = K_i \cap N = M$, both part (i) and part (ii) follow from the following statement: 

{\em Fix $C \in \Conj(N)^G$ and $M \in C$ (we don't assume that $|C|=1$). Let $r \geq 1$ and $U_1, \dots, U_r, V_1, \dots V_r \subseteq \mathcal{F}_C$ be open subsets such that, for each $i \in \{1, \dots, r\}$, there exist finitely generated subgroups $H_i \in U_i$ and $K_i \in V_i$ where $H_i \cap N = K_i \cap N = M$. Then there is a common element $g \in G$ such that $gU_i \cap V_i \neq \emptyset$ for all $i$.} 

Let us prove that this statement indeed holds.

\setcounter{claim}{0}
\begin{claim} 
\label{claim: p and K i p splitting}
There exists an infinite order element $p \in G$ which commutes with $M$ and such that, for each $i$, $\la H_i, K_i^p \ra = H_i \ast_M K_i^p \leq G$ is quasiconvex and has infinite index in $G$.
\end{claim}

\begin{proof}[Proof of Claim~\ref{claim: p and K i p splitting}]
\renewcommand{\qedsymbol}{$\diamond$}

If $G$ has infinitely many ends, then it admits a Stallings tree, a minimal irreducible $G$-tree $T$ with finite edge stabilisers. Note that the kernel $N$ of the action on $T$ is the finite radical of $G$. The infinite index subgroups $H_i,K_i$ have proper minimal subtrees $T_{H_i},T_{K_i}$ in $T$; let $Z_e$ be the half-tree separated from the union $\cup_i T_{H_i}\cup \cup_iT_{K_i}$ of these minimal subtrees by some edge $e$ such that $\Stab_{H_i}(e) = \Stab_{K_i}(e) = M$ for each $i$ (see Claim~\ref{Claim: half-trees} in the proof of Theorem~\ref{th: two edge-stab with finite intersection - top. dyn.}
Take a vertex $x$ of $T$. The map $G\ni g\mapsto g\cdot x$ extends to a $G$-equivariant map $\Psi:G\cup \partial G\to T\cup \partial T$.
By Lemma~\ref{lem: existence of hyperbolic in Z} there is a hyperbolic (with respect to the action on $T$) element $p \in G$ whose axis $\mathrm{Axis}(p)$ is contained in $Z_e$. 
Since it is fixed by $\la p \ra$, the image $\Psi(\Lambda(\la p \ra)$ is the set of limit points of $\mathrm{Axis}(p)$.
Thus, the stabiliser $P:= \Stab_G(\Lambda(\la p \ra))$ stabilises globally $\mathrm{Axis}(p)$.
It follows that any $g \in P \cap H_i$ (respectively of $P \cap K_i$), for any $i$, must stabilise $\mathrm{Axis}(p)$ and $T_{H_i}$ (respectively $T_{K_i}$),
therefore $g \in \Stab_{H_i}(e) = \Stab_{K_i}(e) = M$. 

Otherwise, if $G$ has only one end, let $Y$ any open subset of $\partial G$.
By Proposition~\ref{prop: combination result} there exists an infinite order element $p \in G$ whose limit points are contained in $Y$ such that, if $P \coloneqq \Stab(\Lambda(\la p \ra))$, then $P \cap H_i = P \cap K_i = M$ for each $i$. 

In either case, replace $p$ with a proper power so that it commutes with $M$.

Now let $C_2 \geq 0$ be as in Theorem~\ref{thm: MP2}. Since $\la p \ra$ is quasi-isometrically embedded in $G$ and $M$ is finite, we can replace $p$ with a proper power so that $|p m| \geq |p| - |m| \geq C_2$ for all $p m \in p M = MpM$. Thus by Theorem \ref{thm: MP2}, for each $i$, the subgroup generated by $H_i$ and $K_i^{p}$ splits as
\[ \la H_i, K_i^{p} \ra = H_i \ast_M K_i^{p},\]
and is quasiconvex in $G$.

Moreover, up to taking a proper power of $p$, we can assume that for each $i$,  $H_i \ast_M K_i^{p}$ has infinite index in $G$. 
 Indeed, suppose that $H_i \ast_M K_i^p$ has finite index in $G$ for some $i$. Then the splitting is non-trivial (if $M$ is equal to one of the factors then the other factor has finite index in $G$) so $G$ has at least two, thus infinitely many, ends. In this case, $p$ acts hyperbolically on the Stallings tree with axis $\mathrm{Axis}(p)$ contained in $Z_e$. Let $n \geq 2$ and let $\pi(e)$ be the projection of $e$ on $\mathrm{Axis}(p)$. Then any edge $e'\not\subset \mathrm{Axis}(p)$ but adjacent to $\mathrm{Axis}(p)$ at a vertex strictly between $\pi(e)$ and $p^n\cdot \pi(e)$
defines a half-tree $Z'$ (the one that does not contain $e$) which lies outside any minimal invariant $H_i \ast_M K_i^{p^n}$-subtree of $T$. Thus the quotient $H_i \ast_M K_i^{p^n} \bs T$ is infinite. Since the action of $G$ on $T$ is minimal and $G$ is finitely generated, the quotient $G \bs T$ is compact (see Remark \ref{rem: H fg minimal -> T/H compact}) 
so this implies that $H_i \ast_M K_i^{p^n}$ has infinite index in $G$.  
\end{proof}

By Proposition~\ref{prop: combination result}-(3), there exists an infinite order element $u \in G$ which commutes with $M$ and such that, for all $i$, \begin{align*} 
&A^{(i)} := \la H_i, K_i^p, u \ra = (H_i \ast_M K_i^p) \ast_M \la u, M \ra \leq G \quad \text{and} \\
&B^{(i)} := (A^{(i)})^{p^{-1}} = \la H_i^{p^{-1}}, K_i, u^{p^{-1}} \ra = (H_i^{p^{-1}} \ast_M K_i) \ast_M \la u^{p^{-1}}, M \ra \leq G
\end{align*}
are quasiconvex and have infinite index in $G$. For ease of notation, we now fix $i \in \{1, \dots, r\}$ and denote $A = A^{(i)}$ and $B = B^{(i)}$.

\begin{claim} \label{claim: An Bn} 
For all large enough $n$, we have $A_n := H_i \ast_M (K_i)^{p u^{n}} \in U_i$ while $B_n = (A_n)^{ u^{-n} p^{-1}} = (H_i)^{u^{-n} p^{-1}} \ast_M K_i \in V_i$.
\end{claim}
\begin{proof}[Proof of Claim~\ref{claim: An Bn}]
\renewcommand{\qedsymbol}{$\diamond$}
First observe, for each $n$, that $A_n \leq A$, $B_n \leq B$ and the inclusions $A_n \hookrightarrow A$ and $B_n \hookrightarrow B$ are quasi-isometric embeddings so
 $A_n, B_n$ are quasiconvex in $G$ and have infinite index. Moreover $A_n \cap N = B_n \cap N = M$ so $A_n, B_n \in \mathcal{F}_C$. Now let $\mathcal{V}(\II,\OO) \subseteq \Sub(G)$ be a basic neighbourhood of $H_i$ such that $\mathcal{V}(\II,\OO) \cap \mathcal{F}_C \subseteq U_i$. Let $T, T'$ be finite generating sets for $H_i, K_i$ respectively and let $\lambda \geq 1$ be such that $A \hookrightarrow G$ is a $\lambda$-quasi-isometry with respect to the word metric on $G$ defined by $S$ and the word metric on $A$ defined by $S_A := T \cup p T' p^{-1} \cup \{u\}$. Then for all $a \in A_n \smallsetminus H_i$, the reduced spelling of $a$ must contain as a subword $p u^n t' u^{-n} p^{-1}$ for some $t' \in T'$. Thus 
\[ |a|_S \geq \frac{1}{\lambda} |a|_{S_A} - \lambda \geq \frac{2n}{\lambda} - \lambda.\]
Therefore if $n$ is sufficiently large then $|a|_S > |o|_S$ for all $a \in A_n \smallsetminus H_i$ and $o \in \OO$. Since $\II \subseteq H_i \subseteq A_n$, this implies that $A_n \in \mathcal{V}(\II, \OO) \cap \mathcal{F}_C \subseteq U_i$. By a symmetric argument, $(u^{-n} p^{-1})\cdot A_n = B_n \in V_i$ for sufficiently large $n$.
\end{proof}

Hence, by applying Claim~\ref{claim: An Bn}  to each $i \in \{1, \dots, r\}$, we find $N \in \mathbb{N}$ such that, for each $i$, the intersection $(u^{-N}p^{-1} \cdot U_i) \cap V_i$ is non-empty. 
\end{proof}

\subsection{Hyperbolic 3-manifold groups}
\label{Sect: Hyperb. 3-mfld}

In the case of hyperbolic 3-manifold groups the results of the previous section, together with celebrated results about 3-manifolds, allow us to describe the perfect kernel completely.

The following dichotomy, which is a consequence of Canary's Covering Theorem \cite{Canary96} and the Tameness Theorem \cite{Agol04,Calegari-Gabai06} (see \cite[Theorem 5.2]{AFW15} for a precise statement) allows us to understand the status of non-quasiconvex subgroups in the space of subgroups of a hyperbolic 3-manifold groups. If $G = \pi_1(M)$, where $M$ is a hyperbolic 3-manifold, and $H \leq G$ is finitely generated then one (and only one) of the following holds:
\begin{enumerate}
\item[(a)] $H$ is quasiconvex or
\item[(b)] $H$ is a \defin{virtual fiber}, i.e. there exists a finite index subgroup $G' \leq G$ such that  $H'\coloneqq G'\cap H$ is a hyperbolic surface group and $G'$ fits in a short exact sequence
\[ \begin{tikzcd}
1 \arrow[r] &H' \arrow[r] &G' \arrow{r} &\mathbb{Z} \arrow[r] &1.
\end{tikzcd} \] 
\end{enumerate}

\begin{theorem} \label{theorem: kernel hyp 3-manifolds}
Let $M$ be a closed hyperbolic 3-manifold and $G = \pi_1(M)$. Then 
\[ \PK(G) = \overline{\Sub_{\mathrm{QC}[\infty]}(G)}. \]
More precisely, the perfect kernel of $G$ is the set of infinite index subgroups $H \leq G$ such that either $H$ is quasiconvex or $H$ is not finitely generated.
\\
Moreover, the action of $G$ on $\PK(G)$ is highly topologically transitive and the Cantor-Bendixson rank of $G$ is either 2 or 3.
\end{theorem}

\begin{proof}
By Theorem \ref{thm: kernel QC hyp} and since $G$ is finitely generated $\overline{\Sub_{\mathrm{QC}[\infty]}(G)} \subseteq \PK(G) \subseteq \Sub_{[\infty]}(G)$. Let us show that the first inclusion is in fact an equality.

If $H\leq G$ is a virtual fiber subgroup then, by Corollary \ref{cor: virtual fibre}, $H \notin \PK(G)$. 

It remains to show that every non-finitely generated subgroup of $G$ is the limit of a sequence of quasiconvex subgroups with infinite index.

It follows from  Proposition \ref{prop: G>K>H QI w. Z -> finite index} that a non-finitely generated  subgroup does not contain any virtual fiber subgroup.
Indeed, for any tower of subgroups $K_1 \leq K_2 \leq G$ where $K_1$ is a virtual fiber subgroup, either $[K_2: K_1] < \infty$ or $[G: K_2] < \infty$ and in either case $K_2$ is finitely generated.
Thus if $H = \la s_n : n \in \mathbb{N} \ra \leq G$ is not finitely generated, it is the limit of the non-stationary sequence $(\la s_1, \dots, s_n \ra)_n$ of infinite index subgroups and each $\la s_1, \dots, s_n \ra$ is quasiconvex.

To see that $\rkCB(G) = 2$ or 3, first note that the Cantor-Bendixson rank of $G$ is the supremum of the Cantor-Bendixson erasing ranks of its finite index and virtual fiber subgroups. But virtual fiber subgroups are separable. Indeed if $G' \leq G$, $H' \leq H$ are finite index subgroups such that $G' / H' = \mathbb{Z}$ then $H' = \cap_{n \in \N} \pi^{-1}(n \Z)$ where $\pi: G' \rightarrow \Z$ is the quotient map. For each $n \in \N$, $\pi^{-1}(n \Z)$ has finite index in $G'$ and therefore in $G$, so $H'$ is closed with respect to the profinite topology on $G$. Since $H$ is the union of finitely many cosets of $H'$, it follows that $H$ is also closed. Thus $\rkCBe(H;G) \geq 2$. Since by \cite[Theorem 9.2]{Agol13} every closed hyperbolic 3-manifold group contains a virtual fiber subgroup, this implies that $\rkCB(G)$ is in fact the supremum of the Cantor-Bendixson erasing ranks of the virtual fiber subgroups. For any such subgroup $H \leq G$, any Schreier graph of $H$ is quasi-isometric to $\Z$ or $\N$ (see \cite[Corollary~3.3]{Azuelos})
so it follows from Theorem \ref{th: H co-QI Z then CB-e-rk(H) <4} that $\rkCBe(H;G) \leq 3$. 
\end{proof}

Note that there exist hyperbolic 3-manifold groups with Cantor-Bendixson rank 3. Indeed, it was pointed out to us by Ian Biringer that there are hyperbolic 3-manifold groups which algebraically fiber over the infinite dihedral group $D_\infty$ and the kernel of this fibering has Cantor-Bendixson erasing rank $3$. Recall that a group $G$ \defin{algebraically fibers} over $Q$ if it fits into a short exact sequence $1 \to H \to G \to Q \to 1$ where $H$ is finitely generated. He moreover convinced us (personal communication) that if a virtual fiber subgroup is not virtually of this form then its Cantor-Bendixson erasing rank is $2$.

It follows that the Cantor-Bendixson rank of a hyperbolic 3-manifold group $G$ is 3 if and only if $G$ virtually fibers over $D_\infty$, otherwise it is 2. This leads us to the following question.

\begin{question} 
Does there exist a closed hyperbolic 3-manifold $M$ such that no finite index subgroup of $\pi_1(M)$ algebraically fibers  over $D_\infty$?
Equivalently, does there exist a hyperbolic 3-manifold group $G$ with $\rkCB(G) = 2$?
\end{question}

\section{Some facts and more questions} \label{Section: Facts and questions}

Let us consider the supremum of the Cantor-Bendixson ranks among various classes of  groups
\begin{align*}
\alpha_{\text{fp}}&\coloneqq \sup\{\rkCB(G)\mid G \text{ finitely presented}\}\\
\alpha_{\text{fg}}&\coloneqq \sup\{\rkCB(G)\mid G \text{ finitely generated}\}\\
\alpha_{\text{c}}&\coloneqq \sup\{\rkCB(G)\mid G \text{ countable}\}
\end{align*}

\begin{question} Clearly $\alpha_{\text{fp}}\leq \alpha_{\text{fg}}\leq \alpha_{\text{c}}$.
\begin{enumerate}
\item Which inequalities are strict?
\item What are the values of $\alpha_{\text{fp}}, \alpha_{\text{fg}}$ and $\alpha_\text{c}$?
\item What are the possible values of Cantor-Bendixson ranks for the various classes of groups?
\end{enumerate}
\end{question}

Recall the \defin{Rips construction} and the \defin{Haglund-Wise construction}:\\
Given any finitely presented group $Q$ there exists a specific group $G$ and a finitely generated normal subgroup $N$ with a short exact sequence 
 \[
 \begin{tikzcd}
 1 \arrow[r] & N \arrow[r] & G \arrow[r, "\phi"] & Q \arrow[r] & 1,
 \end{tikzcd}
 \]
 where  the group $G$ can be chosen to be
 \\
--  
a hyperbolic group \cite{Rips-1982-small-canc-}, or 
\\-- 
a finitely presented subgroup of $\mathrm{SL}(n,\Z)$ for some $n$ \cite[Th. 1.6]{Haglund-Wise-2008-SCC}.
 \medskip
 
It follows that:
\begin{proposition}
\begin{align*}\alpha_{\text{\rm fp}}&= \sup\{\rkCB(G)\mid G \text{\rm\  hyperbolic}\}.\\
\alpha_{\text{\rm fp}}&= \sup\{\rkCB(G)\mid G \text{\rm\   finitely presented linear}\}.
\end{align*}
\end{proposition}

Recall from Proposition~\ref{prop: f.g. normal subgr and Sub(Q)} that given an exact sequence $1\rightarrow N\rightarrow G\overset{\phi}{\rightarrow} Q\rightarrow 1$ where $N$ is finitely generated, then $\Sub_{N\leq}(G)$ is a $G$-invariant clopen subset of $\Sub(G)$ and $\phi$ induces a homeomorphism $\Sub_{N\leq}(G)\simeq \Sub(Q)$ that is $G$-equivariant for the $G$-actions, where $G$ acts on $\Sub(Q)$ via $\phi$, so that $\PK(G)\cap \Sub_{N\leq}(G) \simeq \PK(Q)$ and 
$\rk(G)\geq \rk(Q)$ (Proposition~\ref{prop: f.g. normal subgr and Sub(Q)} \eqref {eq: equality of rk CB X}).

If follows from the Rips and Haglund-Wise constructions that  for any finitely presented $Q$ {\em we can embed the action $Q\acting \Sub_{[\infty]}(Q)$, regardless of its complexity, inside the space of non quasiconvex subgroups (see Corollary~\ref{cor: K(Q) non QC}) of some hyperbolic group, or alternatively inside the space of subgroups of a finitely presented linear group}.

\begin{question}These observations give rise to the questions:
\begin{enumerate}
\item How ``pathological'' can the action $H\acting \Sub(H)$ be for a finitely presented group $H$?
\item
Which bounds can be set for the Cantor-Bendixson ranks of finitely presented groups?

\end{enumerate}
\end{question}
Observe that the lamplighter groups, the Grigorchuk group and the Gupta-Sidki $3$-group mentioned in the introduction are not finitely presented. The first examples of a finitely presented group with infinite Cantor-Bendixson rank (namely the Baumslag group or $\FF_2\times \FF_2$) have been announced by Bontemps, Frisch and the second author \cite{BFG23}. This implies in particular the existence of hyperbolic groups with infinite Cantor-Bendixson rank.

\begin{corollary}\label{cor: K(Q) non QC}
If $G$ is produced using the Rips construction above with a non-hyperbolic $Q$, then 
\[\Sub_{N\leq}(G)\cap \overline{\Sub_{\mathrm{QC}[\infty]}(G)}=\emptyset .\]
Thus, if $\PK(Q)\not=\emptyset$, the reverse inclusion 
$ \overline{\Sub_{\mathrm{QC}[\infty]}(G)} \supseteq \PK(G)$ of Theorem \ref{thm: kernel QC hyp}  does not hold since
\[\PK(Q)\simeq \PK(G)\cap \Sub_{\leq N}(G)\subseteq \PK(G)\smallsetminus \overline{\Sub_{\mathrm{QC}[\infty]}(G)}.\]
\end{corollary}
The corollary follows from the following result of Mihalik and Towle \cite[{Theorem~1}]{Mihalik-Towle1994}:
\\
{\em If a subgroup of a hyperbolic group $G$ has infinite index and is quasiconvex, then it does not contain any infinite normal subgroup of $G$}.
The fact that $Q$ is not hyperbolic ensures that $N$ is infinite.

\bigskip

The Theorems~\ref{th: two edge-stab with finite intersection}, \ref{th: two edge-stab with finite intersection - top. dyn.}, \ref{thm: kernel QC hyp}, \ref{thm: QC hyp dynamique}
all use the property that the subgroups we consider fix a proper quasiconvex subset of an ambient hyperbolic space. 
This leads to the following question, which we have not pursued in depth:
\begin{question} \label{question: generalisations}
Is there a common generalisation of our results in the context of group actions on hyperbolic spaces with suitable constraints?
\end{question}

Concerning the above Question~\ref{question: generalisations} and Remark~\ref{rem: rel hyp. gps}, 
notice that results in this direction have been obtained since a first version of this paper was circulated, see \cite{HMO}.

\section{An application: amenable faithful transitive actions}

The class $\mathcal{A}$ is the class of countably infinite groups that admit a faithful, transitive and amenable action.

Recall from \cite{Glasner-Monod-2007} that if a triple of groups $K_1\leq K_2\leq K_3$ satisfies that $K_i$ is co-amenable in $K_{i+1}$, then $K_1$ is co-amenable in $K_3$
and if $K_1$ belongs to $\mathcal{A}$, then $K_3$ belongs to $\mathcal{A}$. 
In particular, if $H\in \mathcal{A}$ and $G\geq H$ is a finite index supergroup, then $G\in \mathcal{A}$. Moreover, if $H_1,H_2 \in \mathcal{A}$ then $H_1 \times H_2 \in \mathcal{A}$.

Moon observed that Baumslag-Solitar groups belong $\mathcal{A}$ since they admit a co-amenable free subgroup (or they are amenable, depending on the parameters).

\begin{proposition}
\label{prop: conditions for being in A}
Let $G$ be a countable group and $\Sub_{[\infty]}(G)$ its space of infinite index subgroups.
Assume that there is
\begin{enumerate}[(a)]

\item 
\label{it: faithful}
 a subgroup $H_1\in \Sub_{[\infty]}(G)$ such that $G\acting G/H_1$ is faithful;

\item \label{it: co-amenability condition for A}
a co-amenable subgroup $H_2\in \Sub_{[\infty]}(G)$ (or more generally a family $(H_2^{i})_{i\in \N}$ of subgroups $H_2^{i}\in \Sub_{[\infty]}(G)$ such that the action $G\acting G/H_2^1\sqcup G/H_2^2\sqcup \cdots \sqcup G/H_2^{i}\sqcup\cdots$ is amenable);
\item 
\label{it: cond on H0}
a subgroup $H_0\in \Sub_{[\infty]}(G)$ whose orbit under the $G$-action by conjugation on $\Sub_{[\infty]}(G)$ accumulates on both $H_1$ and $H_2$  (or more generally on $H_1$ and all the $H_2^{i}$'s).

\end{enumerate}
Then the action $G\acting G/H_0$ is faithful, transitive and amenable. In particular,
$G\in \mathcal A$.

\smallskip
\noindent
Moreover, if $\RRR\subseteq \Sub(G)$ is a $G$-invariant closed subspace containing 
both $H_1$ and $H_2$ (resp. $H_1$ and all the $H_2^{i}$) together with a dense orbit, then the set of $H_0$ satisfying \eqref{it: cond on H0} is a dense G$_\delta$ subset of $\RRR$.
\end{proposition}

\begin{proof}[Proof of Proposition~\ref{prop: conditions for being in A}]
It is enough to show that $G\acting G/H_0$ is faithful and amenable.
Let $\gamma\in G\smallsetminus \{\id\}$. By faithfulness of $G\acting G/H_1$, there is $g_1\in G$ such that $g_1^{-1}\gamma g_1 \not\in H_1 $.
Let $\RRR$ be the closure of the orbit of $H_0$ and let $V$ be the open neighbourhood of $H_1$ in $\RRR$ consisting of the subgroups that do not contain $g_1^{-1}\gamma g_1$.
By density, there is $g\in G$ such that $g H_0 g^{-1}\in V$, i.e. $\gamma \not\in g_1 g H_0(g_1 g)^{-1}$.
 Thus $G\acting G/H_0$ is faithful.

Consider a F{\o}lner sequence $(A_i)_i$ with $A_i\subseteq G/H_2^{i}$ (or $A_i\subseteq G/H_2$).
More precisely, one can (up to re-indexing) write $G$ as an increasing union of finite subsets $G=\ubigcup_i S_i$ such that for each $i$ and $\gamma\in S_i$ we have $\vert \gamma  A_i\triangle A_i\vert\leq \frac{1}{i} \vert A_i \vert$ with $A_i=\{g_j^{i} H_2^{i}\colon j=1, 2, \cdots , \vert A_i\vert\}$.
Let $\vert A_i\vert = n_i$ and $J_i=\{1, 2, \cdots , n_i\}$.
The intersection $\gamma A_i\cap A_i$ is encoded by the equations $(g_k^{i})^{-1} \gamma g_j^{i}\in H_2^{i}$ (for $j,k\in J_i$ and $\gamma\in S_i$) while $\vert A_i\vert $ is encoded by $(g_k^{i})^{-1} g_j^{i}\not\in H_2^{i}$ (for $j,k\in J_i$, $j\not=k$).
These conditions define an open neighbourhood $V_i$ of $H_2^{i}$ in $\RRR$ and by density of the orbit of $H_0$, there is $k_i\in G$ such that $k_i H_0 k_i^{-1}\in V_i$, i.e.
$(g_k^{i})^{-1} \gamma g_j^{i}\in H_2^{i}$ $\implies$ $(g_k^{i})^{-1} \gamma g_j^{i}\in k_i H_0 k_i^{-1}$ and 
$(g_k^{i})^{-1} g_j^{i}\not\in  k_i H_0 k_i^{-1}$.
It follows that the $B_i\colon=\{g_j^{i}k_i H_0\colon j\in J_i \}$ satisfy  $\vert \gamma  B_i\triangle B_i\vert\leq \frac{1}{i} \vert B_i \vert$.
Thus the sequence $(B_i)_i$ is a F{\o}lner sequence for the action $G\acting G/H_0$, which is therefore amenable.

 Concerning the ``moreover" part, since $\RRR$ admits a countable basis $(U_\ell)_{\ell\in \N}$ of open sets, the set $\cap_\ell G\cdot U_\ell$ of points with dense orbits in $\RRR$ is a G$_\delta$. Since their orbit accumulates to some infinite index subgroup (namely $H_1$), they have infinite index themselves.
\end{proof}

\begin{corollary} \label{cor: graphs of free groups in A} 
Any finite graph of free groups with cyclic edges with at least one non-cyclic vertex group is in $\mathcal{A}$.
\end{corollary}

This generalises a result of Moon \cite{Moon-2011-free-prod-A}.

\begin{proof}
We will apply Proposition~\ref{prop: conditions for being in A}.
By Example~\ref{ex: class Q}, $G$ belongs to $\mathcal Q'$.
Thus, $G\acting \Sub_{[\infty]}(G)=:\RRR\ni \{\id\}$ is topologically transitive.
It remains to find a co-amenable infinite index subgroup  $H_2\in \RRR$.

If the graph is not a tree, then $G$ surjects onto $\Z$.
Take $H_2$ to be the kernel of this morphism.
Otherwise, if one of the leaves $v$ of the finite graph has a non-cyclic stabiliser $G_v$, then $G$ surjects onto the abelianisation of $G_v$ modulo the group $G_e$ of the incident edge $e$. 
Eventually, if $v_1,v_2$ are two leaves of the finite graph, with cyclic stabilisers $G_{v_1}, G_{v_2}$ (and the stabilisers $G_{e_1}\lneqq G_{v_1}$ and $G_{e_2}\lneqq G_{v_2}$ of the incident edges $e_1, e_2$ are proper subgroups). Then $G$ maps onto the non-trivial free product $G_{v_1}/G_{e_1}*G_{v_2}/G_{e_2}$, that itself maps virtually onto $\Z$.
Thus $G$ virtually maps onto $\Z$.
The kernel of this virtual morphism is co-amenable and has infinite index.
\end{proof}

\begin{corollary}
 \label{cor: G acts on a tree, edges w. triv. intersect of stab. in A}
Let $G$ be a countable group which admits a minimal and irreducible action $G\acting T$ on a tree. Suppose that there are two edges $f_1,f_2$ of $T$ with $\Stab_G(f_1)\cap \Stab_G(f_2)=N_0$ (the kernel of the $G$-action on $T$), and assume moreover that $G$ admits a co-amenable subgroup $G_0$ (possibly a finite index subgroup) such that $G_0\cap N_0=\{\id\}$ and such that the graph $G_0\bs T$ is not a finite tree.
Then $G \in \mathcal{A}$.
\end{corollary}

\begin{proof}[Proof of Corollary~\ref{cor: G acts on a tree, edges w. triv. intersect of stab. in A}]
If $G_0\bs T$ is not a tree, $G_0$ admits a homomorphism onto $\Z$ whose kernel $N$ gives rise to a infinite quotient $N\bs T$: indeed, pick a simple closed loop $c$ in the graph  $G_0\bs T$, then  the fundamental group $G_0$ of the graph of groups associated with $G_0\acting T$ maps onto $\pi_1(c)\simeq \Z$ with kernel $N\leq G_0$. The graph of groups associated with $N\acting T$ contains an embedded copy of the universal cover of the loop $c$. 
It follows that $H_2\coloneqq N$ is co-amenable in $G_0$ and thus in $G$ and that $H_2\in \overline{\Sub_{\vert \bullet\bs T\vert \infty}(G)}$ and $H_2\cap N_0=\{\id\}$ (i.e. $H_2\in \mathcal{F}_{\{\{\id\}\}}^{T}$ in the terminology of  Equation~\ref{eq: def F-C} and Theorem~\ref{th: two edge-stab with finite intersection - top. dyn.}). In case $G_0\bs T$ is infinite, then simply set $H_2\coloneqq G_0$.
By Theorem~\ref{th: two edge-stab with finite intersection - top. dyn.}, $\RRR \coloneqq \mathcal{F}_{\{\{\id\}\}}^{T}$ 
and  $H_2$ satisfy the assumptions of Proposition~\ref{prop: conditions for being in A}.
\end{proof}

\begin{corollary}
\label{cor: non-tree of groups w. a w.w.malnormal incident edge group}
Let $G$ be a group which decomposes non-trivially as the fundamental group of a graph of groups where the graph is not a finite tree. Assume that, for some edge $e$ and incident vertex $v$, the associated edge group $G_e$ and vertex group $G_v$ satisfy: $\exists g\in G_v$ such that $g G_e g^{-1}\cap G_e=\{\id\}$. Then 
$G\in \mathcal{A}$.
\end{corollary}
\begin{proof}
The condition $g G_e g^{-1}\cap G_e=\{\id\}$ ensures that in the associated Bass-Serre tree $T$, the edge $f_1$ stabilised by $G_e\leq G$ and its image by $g\in G_v$ have stabilisers that intersect trivially. Thus the kernel $N_0$ of the action is trivial. The corollary follows by applying Corollary~\ref{cor: G acts on a tree, edges w. triv. intersect of stab. in A} with $G_0 \coloneqq G$.
\end{proof}

\bigskip

Our final application is to show that the class of (virtually) \defin{compact special groups} introduced by Haglund and Wise in \cite{Haglund-Wise-2008-SCC} is contained in $\mathcal{A}$. These are groups $G$ which (virtually) act co-specially, freely and cocompactly on a CAT(0) cube complex $X$: a cell complex constructed by attaching $n$-cubes $[0,1]^n$ isometrically along faces, which is CAT(0) when equipped with the piecewise $\ell^2$-metric. 
The action $G \acting X$ is \defin{co-special} if the quotient $G\bs X$ is special in the sense of  \cite[Chapter 6, p. 167]{Wise2021}.

We recall some key definitions and properties related to CAT(0) cube complexes and co-special actions.

Let $X$ be a CAT(0) cube complex. The \defin{dimension} of $X$ is the supremum of the set of $n \in \mathbb{N}$ such that $X$ contains an $n$-cube. We say that $X$ is \defin{reducible} if it splits non-trivially as a product $X = X_1 \times X_2$ and $X$ is \defin{irreducible} otherwise. 

A \defin{hyperplane} $h \subseteq X$ is a non-empty connected subspace of $X$ whose intersection with each $n$-cube $\sigma$ of $X$ is either empty or a midcube of $\sigma$ (the space obtained from $[0,1]^n$ by restricting one coordinate to $1/2$). 
It was proved in \cite{Sageev1995} that each hyperplane of $X$ separates $X$ into two connected components, called \defin{halfspaces}; we denote these $h^+$ and $h^-$.
 Given two non-empty subspaces $U,V \subseteq X$, we say that $h$ \defin{separates} $U$ from $V$ if $U$ and $V$ lie in distinct connected components of $X \setminus h$. Given two hyperplanes $h,k \subseteq X$, we say that $h$ and $k$ are \defin{strongly separated} if no hyperplane in $X$ intersects both $h$ and $k$.

  Let $d$ denote the piecewise $\ell^1$-metric on $X$. 
 The distance between any two vertices of $X$ is precisely the number of hyperplanes separating them; this coincides with the restriction to the 0-skeleton of the path metric on the 1-skeleton of $X$. There is a well-defined \defin{gate map} $\mathfrak{g}_h: X \rightarrow h$ which maps each vertex of $X$ to its (unique) closest point projection in $h$ (see \cite[Section~2.1]{BHS2017}). This map has the property that {\em if $x,y$ are vertices of $X$ and $k$ is a hyperplane then $k$ separates $\mathfrak{g}_h(x)$ and $\mathfrak{g}_h(y)$ if and only if $k$ intersects $h$ and separates $x$ from $y$}. In particular, if $h$ and $k$ are strongly separated hyperplanes then $\mathfrak{g}_h(k)$ is a single point.

We consider actions by cellular automorphisms, these are isometries with respect to both $d$ and the CAT(0) metric (i.e. the piecewise $\ell^2$-metric). The group of all cellular automorphisms of $X$ is denoted by $\Aut(X)$. An action $G \acting X$ on a CAT(0) cube complex is said to be \defin{proper} if the stabiliser of each $n$-cube is finite. It is \defin{essential} if for any orbit $G \cdot v$ and for each halfspace $h^+$ bounded by a hyperplane $h$, there are points in $(G \cdot v) \cap h^+$ arbitrarily far from $h$. 

We recall the Double Skewering Lemma of \cite{Caprace-Sageev2011}:
{\em Suppose that $G$ is a group acting essentially on a finite dimensional CAT(0) cube complex without a fixed point at infinity. Then for any halfspaces $h^+ \subseteq k^+$, there is an element $\gamma \in G$ such that $\gamma k^+ \subsetneq h^+ \subseteq k^+$.}

Suppose that $h \subseteq X$ is a hyperplane such that the set of hyperplanes $G \cdot h$ is pairwise disjoint (i.e. for all $g \in G$ either $gh=h$ or $gh \cap h = \emptyset$) and such that $G$ does not invert $h$ (i.e. if $g \in \Stab_G(h)$ then $gh^+ = h^+$). Let $N(h)$ be the open carrier of $h$; that is, the union of all open cubes which intersect $h$.
Let $T_h$ be the graph whose vertex set is the set of connected components of $X \smallsetminus (G \cdot N(h))$ and where $v, w \in V(T)$ are connected by an edge if they are both bounded by a common hyperplane carrier in $G \cdot N(h)$. This defines a continuous surjection $\pi_h: X \rightarrow T_h$ which maps each connected component of $X \smallsetminus (G \cdot N(h))$ to a vertex and each open hyperplane carrier $N(k)$ in $G \cdot N(h)$ to an open edge whose midpoint is $\pi_h(k)$. Note that edges of $T_h$ separate the space into two connected components so $T_h$ is a tree. We will refer to $T_h$ as the \defin{dual tree} of $h$. The action $G \acting (G \cdot N(h))$ induces an action $G \acting T_h$ via $\pi_h$. Since $G$ does not invert hyperplanes in the orbit of $h$, the action of $G$ on $T_H$ is without edge inversions.

Recall that if $G \acting X$ is co-special then $G$ acts with pairwise disjoint hyperplane orbits and no hyperplane inversions. By the above, this implies that $G$ acts on the dual tree $T_h$ for any hyperplane $h$ of $X$. 

\begin{lemma} \label{lem: action on a tree}
Let $G$ be a countable group acting on an irreducible CAT(0) cube complex $X$, where $X$ is not quasi-isometric to $\mathbb{R}$ and the automorphism group $\Aut(X)$ acts cocompactly. Suppose that the action is proper, essential and does not fix a point in the visual boundary $\partial_\infty X$. Suppose moreover that there is a hyperplane $h \subseteq X$ such that the set of hyperplanes $G \cdot h$ is pairwise disjoint and such that $G$ does not invert $h$. Let $T_h$ be the corresponding dual tree.
Under these conditions, the action $G \acting T_h$ is minimal and irreducible and there are edges $f_1, f_2 \in E(T)$ such that $\Stab_G(f_1) \cap \Stab_G(f_2)$ is finite.
\end{lemma}
\begin{proof}
    Since the action of $\Aut(X)$ is cocompact, $X$ is finite dimensional.

    We first show that there is a pair of strongly separated hyperplanes in the orbit $G \cdot h$. Since $X$ is irreducible and finite dimensional and since $G$ acts essentially without a fixed point at infinity, there is a pair of strongly separated hyperplanes $k_1, k_2 \subseteq X$ by \cite[Theorem~5.1]{Caprace-Sageev2011}. 
    Since $h$ does not intersect both $k_1$ and $k_2$, then up to relabelling, we can  assume that $h^- \subseteq k_2^-$ and that $k_2^+ \subseteq k_1^+$ where $k_i^+, k_i^-$ are the halfspaces of $k_i$.
    By the Double Skewering Lemma, there exists $\gamma, \psi \in G$, such that $\gamma k_2^- \subsetneq k_1^-$ and $\psi h^+ \subsetneq k_2^+$. Thus $\gamma h \subseteq  k_1^-$ and $\psi h \subseteq k_2^+$. If a hyperplane intersects both $\gamma h$ and $\psi h$ then it must intersect $k_1$ and $k_2$ so no such hyperplane exists. Therefore $h_1 \coloneqq \gamma h$ and $h_2 \coloneqq \psi h$ are strongly separated.

    Recall that, under the gate map, the image $\mathfrak{g}_{h_1}(h_2)$ is a single point, say $x_1$. If $g \in \Stab_G(h_1) \cap \Stab_G(h_2)$ then $d(g x_1, h_2) = d(x_1, h_2)$ and, since $g x_1 \in h_1$ and $x_1$ is the unique point in $\mathfrak{g}_{h_1}(h_2)$, we have $g x_1 = x_1$. Since the action of $G$ on $X$ is proper, this implies that $\Stab_G(h_1) \cap \Stab_G(h_2)$ is finite. Thus if $f_i$ is the edge of $T_h$ with midpoint $\pi_h(h_i)$ for $i = 1,2$ then $\Stab_G(f_1) \cap \Stab_G(f_2)$ is finite, as required.

    Let us now show that the action $G \acting T_h$ is minimal and irreducible. Firstly, by the Double Skewering Lemma, $T_h$ is unbounded and the $G$-action has no global fixed point. Since the action has a single orbit of edges, it is minimal. 
    
    Since $X$ is irreducible and not a quasi-line and since the action of $\Aut(X)$ is cocompact, \cite[Lemma~7.1]{Caprace-Sageev2011} implies that $\Aut(X)$ does not stabilise any Euclidean flat $\mathbb{R}^n$ in $X$. By \cite[Theorem~7.2]{Caprace-Sageev2011}, this implies that there is a \defin{facing triple} of hyperplanes in $X$: i.e. a pairwise disjoint triple of hyperplanes $k_0, k_1, k_2$ such that no $k_j$ separates the remaining two hyperplanes. Suppose that for each $j$, $k_j^+$ is the halfspace containing the remaining $k_{\ell}$'s. Since $h_1, h_2$ are strongly separated, there is $i \in \{1,2\}$ and $j \in \mathbb{Z} / 3\mathbb{Z}$ such that $h_i \cap k_j = \emptyset$. By the Double Skewering Lemma, there exists $g \in G$ such that $g h_i \subseteq k_j^-$ (set $g = \id$ if this was already the case). Then, using the Double Skewering Lemma again, there are $g_{j+1}, g_{j+2} \in G$ such that $g_{j+1} g h_i \subseteq k_{j+1}^-$ and  $g_{j+2} g h_i \subseteq k_{j+2}^-$. The set $\{gh_i, g_{j+1}gh_i, g_{j+2}gh_i\}$ is a facing triple in the $G$-orbit of $h$. Thus there is a facing triple of hyperplanes (in this case mid-points of edges) in $T_h$, so $T_h$ cannot be a line. By minimality of the action, $G$ cannot fix a line in $T_h$.

    It remains to show that $G$ does not fix an end. Let $(e_n)_{n \in \mathbb{N}}$ be a geodesic ray such that $e_1$ is the edge whose midpoint is $\pi_h(h)$. There is a halfspace of $h$ which contains $\pi_h^{-1}(e_n)$ for all $n \geq 2$, suppose without loss of generality that this halfspace is $h^+$. By the Flipping Lemma of \cite{Caprace-Sageev2011} there is an element $\gamma \in G$ such that $\gamma h^+ \subsetneq h^-$. Thus the paths $(e_n)_{n \geq 2}$ and $(\gamma e_n)_{n \geq 2}$ are contained in distinct components of $T_h \setminus \{e_1\}$ and therefore represent distinct ends of $T_h$.
\end{proof}

\begin{theorem} \label{thm: virtually compact special in A}
    If $G$ is an infinite virtually compact special group then $G \in \mathcal{A}$.
\end{theorem}
Examples of virtually compact special groups include 
right-angled Artin groups (essentially by definition); 
right-angled and hyperbolic Coxeter groups \cite{Niblo-Reeves-2003,Haglund-Wise-2010}; limit groups \cite[Theorem~18.6]{Wise2021}; one-relator groups with torsion \cite[Corollary~19.2]{Wise2021}.
Many groups are known to act properly and
cocompactly on CAT(0) cube complexes.
By \cite{Agol13}, those that are moreover hyperbolic are then virtually compact special. This includes 
C'(1/6) small cancellation groups \cite{Wise-2004};  random groups at density $d<1/6$ \cite{Ollivier-Wise-2011}, hyperbolic-by-cyclic groups which are themselves hyperbolic \cite{DMM-2023}.

\begin{proof}
    Recall that class $\mathcal{A}$ contains every infinite amenable group and is closed under taking direct products. Moreover if $G$ is virtually in $\mathcal{A}$ then $G$ itself is in $\mathcal{A}$.
    Thus we can assume that $G$ is compact special; in particular, $G$ is torsion-free. Let $X$ be a CAT(0) cube complex on which $G$ acts freely, cocompactly and co-specially.
    By \cite[Proposition~3.5]{Caprace-Sageev2011}, there is a convex subcomplex of $X$,
    called the \defin{essential core}, on which $G$ acts essentially. The action of $G$ on $Y$ is moreover free and cocompact and co-special.
    Thus by replacing $X$ with its essential core and $G$ with a finite index subgroup, we can assume that the action $G \acting X$ is free, cocompact, co-special and essential.
    
    If $X$ splits non-trivially as a product $X = X_1 \times X_2$ then by \cite[Lemma~3.8]{Zalloum2023} there is a finite index subgroup $G' \leq G$ which splits as a direct product $G' = G_1 \times G_2$, where $G_1 = \Stab_{G'}(X_1 \times \{x_2\})$ and $G_2 = \Stab_{G'}(\{x_1\} \times X_2)$ for any $x_i \in X_i$.
    The actions $G_i \acting X_i$ are free, cocompact and co-special. Moreover, for each $i$, the dimension of $X_i$ is strictly inferior to that of $X$. 
    
    By induction on the dimension, if we iterate the process of passing to the essential core, then applying \cite[Lemma~3.8]{Zalloum2023} if the subspace is reducible, we arrive at a set of subgroups $H_1, \dots, H_k$ of $G$ such that:
    \\
    -- $G$ can be constructed from $H_1, \dots, H_k$ by taking direct products and finite index supergroups and 
    \\
    -- each $H_i$ acts on an irreducible CAT(0) cube complex $X_i$ such that the action is free, cocompact, co-special and essential.
    \\
    We can therefore assume that $X$ is irreducible and $G \acting X$ has the above properties.
    
    Let $\rho:X \rightarrow X \bs G$ be the quotient map. Fix a hyperplane $h \subseteq X$ and let $T_h$ be the corresponding dual tree. By \cite[Lemma~14.6]{Wise2021} we can replace $G$ with a finite index subgroup so that $\rho(h)$ does not separate $G \bs X$. This implies that $G \bs T_h$ is a loop.
    
    By \cite[Corollary~4.9]{Caprace-Sageev2011}, 
    and since $X$ is irreducible, either $X$ is quasi-isometric to a line
    or there is no fixed point at infinity. In the former case, $G$ is virtually cyclic (in fact, cyclic since it is torsion free) so $G \in \mathcal{A}$ since it is amenable.
    On the other hand, if $G$ does not fix a point in $\partial_\infty X$ then it follows from Lemma~\ref{lem: action on a tree} that the action of $G$ on $T_h$ is minimal and irreducible, and that there are edges $f_1, f_2 \in E(T_h)$ such that $\Stab_G(f_1) \cap \Stab_G(f_2)$ is finite. Since $G$ is torsion-tree this intersection is in fact trivial. Thus $G \in \mathcal{A}$ by Corollary~\ref{cor: G acts on a tree, edges w. triv. intersect of stab. in A}.
\end{proof}

\appendix
\section{Appendix: A list of examples}

A list of examples of finitely generated groups for which $\PK(G)=\Sub_{[\infty]}(G)$, thus $\rkCB(G)=1$, and for which $G\acting \PK(G)$ is highly topologically transitive.

\begin{enumerate}
\item Non-abelian limit groups  (Corollary~\ref{cor: groups in CCC, K(G) and HTT} and Example~\ref{ex: class Q}).

\item 
Non-elementary free products $A*B$ and groups with infinitely many ends and no non-trivial finite normal subgroup (Corollary~\ref{cor:infinitely many ends}). As for $\mathrm{SL}(2,\Z)$, see Example~\ref{ex: SL(2,Z)}.

\item Graphs of free groups with cyclic edges such that at least one vertex group is non-cyclic (Corollary~\ref{cor: groups in CCC, K(G) and HTT} and Example~\ref{ex: class Q}).

\item 
Groups in the class $\mathcal{Q}'$ (Corollary~\ref{cor: groups in CCC, K(G) and HTT}).

\item
Non-elementary  and locally quasiconvex hyperbolic groups with trivial finite normal subgroup (Corollary~\ref{cor: locally quasiconvex}).
For instance, non-elementary hyperbolic groups with a small hierarchy and trivial finite normal subgroup (Example~\ref{ex: hyperbolic gps w. a small hierarchy}).

\item 
One-relator group $G = \la S \: | \: w^m \ra$ where $w$ is a cyclically reduced word in $S \cup S^{-1}$ and $m \geq |w|_S$ (Example~\ref{ex: one-relator group}).

\item 
 Coxeter groups $G=\la s_1, \dots, s_n | s_i^2, (s_i s_j)^{m_{i,j}} \; \forall \; i < j \ra$ 
where $\infty\geq m_{i,j}> n\geq 3$ (Example~\ref{ex: Coxeter groups}).

\end{enumerate}

\bibliographystyle{plain}

  \bigskip
{\footnotesize
  \noindent
  {P.A., \textsc{University of Bristol, School of Mathematics, Bristol, UK}}\par\nopagebreak \texttt{penelope.azuelos@bristol.ac.uk}
  
\medskip 

  \noindent
{D.G., \textsc{CNRS, ENS-Lyon, 
Unité de Mathématiques Pures et Appliquées,  69007 Lyon, France}}
\par\nopagebreak \texttt{damien.gaboriau@ens-lyon.fr}
}

\end{document}